\documentstyle[11pt,leqno]{article}

\include{psfig}
\include{epsf}

\newtheorem{theorem}{Theorem}[section]
\newtheorem{lemma}{Lemma}[section]

\newtheorem{remark}{Remark}[section]
\newtheorem{example}{Example}[section]

\catcode`\@=11
\renewcommand{\section}{
         \setcounter{equation}{0}
         \@startsection {section}{1}{\z@}{-3.5ex plus -1ex minus
         -.2ex}{2.3ex plus .2ex}{\normalsize\bf}
}
\renewcommand{\subsection}{
         \@startsection {subsection}{1}{\z@}{-3.5ex plus -1ex minus
         -.2ex}{2.3ex plus .2ex}{\normalsize\bf}
}
\catcode`\@=12

\def\reals{{\rm\vrule depth0ex width.4pt\kern-.08em R}}
\def\bbbz{{\mathchoice {\hbox{$\sf\textstyle Z\kern-0.4em Z$}}
{\hbox{$\sf\textstyle Z\kern-0.4em Z$}}
{\hbox{$\sf\scriptstyle Z\kern-0.3em Z$}}
{\hbox{$\sf\scriptscriptstyle Z\kern-0.2em Z$}}}}

\newcommand{\nc}{\newcommand}
\nc{\W}{{\bf W}}
\nc{\A}{{\bf A}}
\nc{\bL}{{\bf L}}
\nc{\bH}{{\bf H}}
\nc{\C}{{\cal C}}

\def\eq#1{(\ref{e:#1})}

\def\elabel#1{\label{e:#1}}

\begin{document}
\begin{center}
\Large\bf Optimal Rate Scheduling via Utility-Maximization for
$J$-User MIMO Markov Fading Wireless Channels with
Cooperation\footnote{The author gratefully acknowledges the support
from National Natural Science Foundation of China under grant No.
10971249.}
\end{center}
\begin{center}
\large\bf Wanyang Dai
\end{center}
\begin{center}
\small Department of Mathematics and State Key Laboratory of Novel
Software Technology\\
Nanjing University, Nanjing 210093, China\\
Email: nan5lu8@netra.nju.edu.cn\\
Originally submitted on June 17, 2010\\
Revised version submitted on
December 24, 2010
\end{center}

\vskip 0.1 in
\begin{abstract}

We design a dynamic rate scheduling policy of Markov type via the
solution (a social optimal Nash equilibrium point) to a
utility-maximization problem over a randomly evolving capacity set
for a class of generalized processor-sharing queues living in a
random environment, whose job arrivals to each queue follow a doubly
stochastic renewal process (DSRP). Both the random environment and
the random arrival rate of each DSRP are driven by a finite state
continuous time Markov chain (FS-CTMC). Whereas the scheduling
policy optimizes in a {\em greedy} fashion with respect to each
queue and environmental state and since the closed-form solution for
the performance of such a queueing system under the policy is
difficult to obtain, we establish a reflecting diffusion with
regime-switching (RDRS) model for its measures of performance and
justify its asymptotic optimality through deriving the stochastic
fluid and diffusion limits for the corresponding system under heavy
traffic and identifying a cost function related to the utility
function, which is minimized through minimizing the workload process
in the diffusion limit. More importantly, our queueing model
includes both $J$-user multi-input multi-output (MIMO) multiple
access channel (MAC) and broadcast channel (BC) with cooperation and
admission control as special cases. In these wireless systems, data
from the $J$ users in the MAC or data to the $J$ users in the BC is
transmitted over a common channel that is fading according to the
FS-CTMC. The $J$-user capacity region for the MAC or the BC is a
set-valued stochastic process that switches with the FS-CTMC fading.
In any particular channel state, we show that each of the $J$-user
capacity regions is a convex set bounded by a number of linear or
smooth curved facets. The random arrival rate to each user for these
systems is designed to switch with the FS-CTMC fading via admission
control. At the transmit end, packets to each user are queued and
served under the policy. Therefore our queueing model can perfectly
match the dynamics of these wireless systems.\\

\end{abstract}

\section{Introduction}

In the current cellular systems, each base station is considered as
a separate entity with no cooperation among base stations,
infrastructure cooperation among base stations has been proposed in
the literature such as \cite{acabha:bescas,kumvis:joipow,
viskum:ratsch}, which is to consider the base stations as one end of
a MIMO system that has received a great deal of attention as a
method to achieve high data rates over wireless links. Thus, in this
paper, we study a $J$-user MIMO MAC uplink system and a $J$-user
MIMO BC downlink system. Both of them can be seen as a cellular
system with multiple users and multiple cooperating base station
antennas: either multiple cooperating base stations each with a
single antenna or a single-cell cellular system with a multi-antenna
base station or a combination thereof. In the MAC or the BC, data is
buffered at the transmit end and the channel is time-varying due to
multipath fading, which is a typical feature of wireless channel and
brings additional complexity for system design and performance
analysis. We suppose that the fading process is a FS-CTMC whose
discrete time version is widely used in modeling wireless channels
(see, e.g.,
\cite{wanmoa:finsta,sto:maxsch,viskum:ratsch,daiwan:optcon}, and
references therein). Therefore, the $J$-user capacity regions of the
MAC and the BC are both time-varying set-valued stochastic processes
driven by the FS-CTMC and in each state of the Markov chain, it is
well known that one can obtain the improved capacity by cooperation,
e.g., the sum of the rates at which data can be served for the $J$
users is greater than the single-user capacity for any user (see,
e.g., \cite{bhawil:peruse}). Moreover, due to the impact of the
random environmental fading factor and the cooperated design, the
service rates of the corresponding queueing system for the $J$ users
in the MAC or in the BC are also random processes driven by the
FS-CTMC.

So, motivated by the above observations, we consider a type of
generalized processor-sharing queues living in a random environment,
whose job arrivals to each queue follow a DSRP. Both the random
environment and the random arrival rate of each DSRP are driven by a
FS-CTMC. Presently, for such a queueing system, it is not known how
to choose a reasonable online rate scheduling policy to minimize the
average delay for a given load and exact solutions for average delay
are not available even for many simple policies, which implies that
any meaningful comparison has to be done by simulations. Therefore,
to make the gap between the dynamic rate scheduling and the
performance optimization for the system be filled to some extent, we
design a dynamic rate scheduling policy of Markov type via the
solution (a social optimal Nash equilibrium point) to an
optimization problem that maximizes a general utility function over
each of the randomly evolving capacity regions through the
Karush-Kuhn-Tucker (KKT) optimality conditions (see, e.g.,
\cite{lue:linnon}). Moreover, to overcome the intractability of
performance evaluation for the system under the designed policy, we
develop stochastic fluid and diffusion models through suitable
scaling of time and space and justifying related limit theorems for
a heavily loaded queueing system operating under this policy. The
limit models for queue lengths (or workloads) are respectively a
random process driven by the FS-CTMC and a RDRS (i.e., a reflecting
stochastic differential equation (SDE) with regime-switching). In
addition, we identify a cost function related to the utility
function, which is minimized through minimizing the workload process
in the diffusion limit and hence provides a useful means in
illustrating our policy to be asymptotically optimal.

Finally, in order to incorporate the $J$-user MIMO MAC and MIMO BC
into our general queueing framework, we justify that the $J$-user
capacity region for the MAC or the BC in any particular channel
state is a convex set formed by a number of linear or smooth curved
facets through applying the method of convex optimization, the
implicit function theorem, and the duality of capacity regions
between the MAC and the BC. Moreover, to realize the DSRP in the MAC
or in the BC, we adopt a cross-layer design methodology to switch
the arrival rates with the FS-CTMC channel fading process according
to the current channel state information (CSI) through admission
control.

\vskip 0.25cm \noindent{\bf Literature Review}

The randomly evolving capacity region used in designing our
utility-maximization rate scheduling policy is a generalization of
the so-called MIMO channel capacity region in the Shannon theoretic
sense. For a single-user time-invariant channel, the Shannon
capacity is defined as the maximum mutual information between input
and output, which is shown by Shannon's capacity theorem to be the
maximum data rate that can be transmitted over the channel with
arbitrarily small error probability. For a $J$-user time-invariant
MIMO channel, the corresponding capacity region is a $J$-dimensional
set of all rate vectors $(c_{1},...,c_{J})'$ simultaneously
achievable by all $J$ users. In particular, the region for the
Gaussian MAC is a convex set that is the union of rate regions
corresponding to every product input distribution satisfying the
user-by-user power constraints (see, e.g., \cite{covtho:eleinf},
\cite{chever:gaumul}, \cite{yurhe:itewat}, \cite{goljaf:caplim}).
The Gaussian BC differs from the Gaussian MAC in two fundamental
aspects (see, e.g., \cite{jinvis:duagau}). In the MAC, each
transmitter has an individual power constraint, whereas in the BC
there is only a single power constraint on the transmitter.
Moreover, signal and interference come from different transmitters
in the MAC and are multiplied by different channels gains (known as
the near-far effect) before being received, whereas in the BC, the
entire received signal comes from the same source and therefore has
the same channel gain. Nevertheless, the capacity region for the
Gaussian BC can be obtained through the duality between the Gaussian
MAC and the Gaussian BC (see, e.g., \cite{jinvis:duagau} and
\cite{goljaf:caplim}), i.e., it is the convex hull of the union over
the set of capacity regions of the dual Gaussian MACs such that the
total MAC power is the same as the power in the BC. Moreover, the
authors in \cite{liuhou:weipro} provide an analytical and numerical
characterization in terms of the shape of the capacity boundaries
for both the MAC and the BC.

However, in both the Gaussian MAC and the Gaussian BC, the exact
characterization concerning piecewise smoothness of the capacity
boundaries is not available until now, which motivates us to give
more accurate analysis about the capacity region in order to apply
our utility maximization rate scheduling algorithm to these wireless
systems. In addition, when the $J$-user MIMO channels are stochastic
and time-varying fading ones, the capacity regions have multiple
definitions (see, e.g., \cite{goljaf:caplim}). Nevertheless, to
capture the exact capacity region at each time instant for the MAC
or the BC, we consider the capacity regions as a set-valued
stochastic process evolving with the FS-CTMC rather than think of it
as a fixed one in an average sense such as an ergodic capacity
region (see, e.g., \cite{goljaf:caplim}).

Concerning the scheduling algorithms, the authors in
\cite{acabha:bescas,bhawil:peruse,bhawil:difapp} considered a
quasi-static downlink channel, where the channel is assumed to be
fixed for all transmissions over the period of interest. In this
case, the FS-CTMC and the random packet arrival rates assumed in the
current paper reduce to constants, and moreover, without considering
utility and cost optimization, the authors in
\cite{acabha:bescas,bhawil:peruse,bhawil:difapp} designed a simple
rate scheduling policy of Markov type, which was shown to be
throughput-optimal for a fixed convex capacity region in
\cite{acabha:bescas} and a limit theorem was proved to justify the
diffusion approximation of the queue length process for a heavily
loaded system operating under their policy with two users in
\cite{bhawil:peruse} and with multiple users in
\cite{bhawil:difapp}. Their approximating model is a RBM living in
the two-dimensional positive quadrant or in the general-dimensional
positive orthant.

In the studies of \cite{sto:maxsch,shasri:patopt,daiwan:optcon},
some scheduling policies were considered for certain heavily loaded
wireless systems with finite state discrete time Markov fading
process. In particular, a MaxWeight scheduling policy was considered
in \cite{sto:maxsch} for a generalized switch and it was shown that
the workload process converges to a one-dimensional RBM and
MaxWeight policy asymptotically minimizes the workload under certain
conditions. Moreover, an exponential scheduling rule was designed
for wireless channels in \cite{shasri:patopt} and for a generalized
switch in \cite{daiwan:optcon}, which was proved to be
throughput-optimal and under which, the similar results concerning
the workload process were obtained and justified as in
\cite{sto:maxsch}. In addition, \cite{yeyao:heatra} designed a
utility-maximizing resource allocation policy for a class of
stochastic networks with concurrent occupancy of resources and
established its asymptotic optimality for the associated heavily
loaded queueing system. Their policy covers the generalized
$c\mu$-rule in~\cite{mansto:schfle} and the MaxWeight policy
in~\cite{sto:maxsch} as special cases.

The differences between the current study and those in
\cite{sto:maxsch,shasri:patopt,daiwan:optcon,yeyao:heatra} are in
three aspects as follows.

First, their scheduling policies in
\cite{sto:maxsch,shasri:patopt,daiwan:optcon,yeyao:heatra} depend
only on a fixed capacity region that is a convex polyhedral and ours
depends on a time-varying and stochastic evolving capacity region
process (a random environment) that, at each time instant, is a more
general convex region rather than a convex polyhedral.

Second, the rates of packet arrivals to the $J$ users are random
processes rather than a constant as used in
\cite{sto:maxsch,shasri:patopt,daiwan:optcon,yeyao:heatra}. Hence
our input traffic to each user is a DSRP whose particular case is
the well-known doubly stochastic Poisson process (see, e.g.,
\cite{bre:poipro}) that is widely used to model voice, video and
data source traffics in telecommunication systems and is called
Markovian modulated Poisson process (MMPP) or ON/OFF source (see,
e.g., \cite{jairou:pactra}, \cite{nikaky:ovesou},
\cite{taqwil:profun}, \cite{dai:contru}) and \cite{dai:heatra})).

Third, our discussion is based on a continuous time horizon rather
than a discrete one as in
\cite{sto:maxsch,shasri:patopt,daiwan:optcon}. Therefore our
vector-valued random service rate process depends on the FS-CTMC
whose holding time at each environmental state has an important
impact on the limiting processes, e.g., the limiting fluid model is
a random process driven by the FS-CTMC rather than a deterministic
function of time and the limiting diffusion model is a more general
RDRS rather than a RBM as derived in
\cite{sto:maxsch,shasri:patopt,daiwan:optcon}. If one wants to
directly generalize the studies in
\cite{sto:maxsch,shasri:patopt,daiwan:optcon} to the corresponding
ones in a discrete time random environment, a geometric distribution
may be imposed on the holding time at each environmental state.

Finally, without considering optimal dynamic scheduling with
utility/cost and performance optimizations as the goals. CTMCs have
been used to model the random environments in the studies of some
queueing systems under certain static service disciplines, see,
e.g., \cite{choman:fludif} and references therein for more details.

The rest of the paper is organized as follows. In
Section~\ref{aqsy}, we introduce our generalized processor-sharing
queues under random environment and design our optimal rate
scheduling policy. In Section~\ref{apmodel}, we introduce our heavy
traffic condition and present our main asymptotic optimality
theorem. In Section~\ref{wnm}, we illustrate the usages of our
optimal policy and our main results in the $J$-user MIMO uplink and
downlink wireless channels and present the associated results
concerning the piecewise smoothness of capacity boundaries of the
$J$-user MIMO MAC and MIMO BC. In
Sections~\ref{proofmains}-\ref{prooflemmas}, we prove our main
theorem and associated lemmas.

\section{Optimizing Processor-Sharing Queues under Random Environment}
\label{aqsy}

\subsection{Primitive Data}

The queueing system under consideration is a type of generalized
processor-sharing queues that live in a random environment evolving
according to a stationary FS-CTMC
$\alpha=\{\alpha(t),t\in[0,\infty)\}$, which takes value in a finite
state space ${\cal K}\equiv\{1,...,K\}$ with generator matrix
$G=(g_{il})$ ($i,l\in{\cal K}$) and
\begin{eqnarray}
&&g_{il}=\left\{\begin{array}{ll}
-\gamma(i)&\mbox{if}\;\;i=l,\\
\gamma(i)q_{il}&\mbox{if}\;\;i\neq l
\end{array}
\right. \elabel{generatorm}
\end{eqnarray}
where $\gamma(i)$ is the holding rate for the chain in an
environmental state $i\in\{1,...,K\}$ and $Q=(q_{il})$ is the
transition matrix of its embedded discrete time Markov chain (see,
e.g., \cite{res:advsto}). Moreover, the queueing system has $J$
queues in parallel, which correspond to $J$ users for a given
positive integer $J$. Each queue that is of infinite buffer capacity
buffers packets (jobs) arrived for a given user. The queues can be
served simultaneously by a single server with rate allocation vector
$c(t)=(c_{1}(t),...,c_{J}(t))'$ that takes values in a time-varying
and randomly evolving capacity set ${\cal R}(\alpha(t))$.

Concretely, for each state $i\in{\cal K}$, ${\cal R}(i)$ is a convex
set that contains the origin and has $L\;(>J)$ boundary pieces of
which $J$ are $(J-1)$-dimensional linear facets along the coordinate
axes while the remaining ones are in the interior of $R^{J}_{+}$ and
form the so-called {\it capacity surface} denoted by ${\cal O}(i)$,
which consists of $B=L-J\;(>0)$ linear or smooth curved facets
$h_{k}(c,i)$ on $R_{+}^{J}$ for $k\in{\cal U}\equiv\{1,2,...,B\}$,
i.e.,
\begin{eqnarray}
&&{\cal R}(i)\equiv\left\{c\in R_{+}^{J}:\;h_{k}(c,i)\leq
0,\;k\in{\cal U}\right\}. \elabel{capsur}
\end{eqnarray}
Moreover, if we let $C_{U}$ denote the sum capacity upper bound for
the capacity region, then the facet in the center of the capacity
surface is linear and can be expressed by
\begin{eqnarray}
&&h_{k_{U}}(c,i)=\sum_{j=1}^{J}c_{j}-C_{U}\elabel{hmiddle}
\end{eqnarray}
where $k_{U}\in{\cal U}$ is the index corresponding to $C_{U}$.
Moreover, we suppose that any one of the $J$ linear facets along the
coordinate axes forms a $(J-1)$-user capacity region corresponding
to a particular group of $J-1$ users who are the only users in the
systems. Similarly, we can define the $(J-j)$-user capacity region
for each $j\in\{2,...,J-1\}$. Examples of such capacity sets in two-
and three-dimensional spaces for a particular state $i\in{\cal K}$
are shown in Figures~\ref{twouserregion} and~\ref{threeuserregion}.
\begin{figure}[tbh]
\centerline{\epsfxsize=3.0in\epsfbox{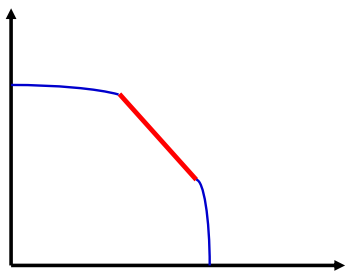}} \caption{\small
A 2-user capacity set in the 2-dimensional space in a particular
environmental state} \label{twouserregion}
\end{figure}
\begin{figure}[tbh]
\centerline{\epsfxsize=3.0in\epsfbox{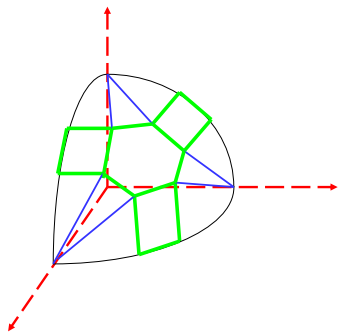}}
\caption{\small A 3-user capacity set in the 3-dimensional space in
a particular environmental state} \label{threeuserregion}
\end{figure}

In addition, we suppose that the system starts empty and that there
is a $J$-dimensional packet arrival process
$A=\{A(t)=(A_{1}(t),...,A_{J}(t))',t\geq 0\}$, where $A_{j}(t)$ with
$j\in{\cal J}$ and $t\geq 0$ is the number of packets arrived to the
$j$th queue during $(0,t]$ and the prime denotes the transpose of a
vector or a matrix. For each $j\in{\cal J}$, $A_{j}(\cdot)$ is
assumed to be a DSRP with random arrival rate process
$\lambda_{j}(\alpha(\cdot))$ and squared coefficient of variation
process $\alpha^{2}_{j}(\alpha(\cdot))\in(0,\infty)$. The packet
interarrival times are assumed to be i.i.d. during the time interval
corresponding to a specific environmental state $i\in{\cal K}$.
Moreover, let $\{u_{j}(k),k=1,2,...\}$ denote the sequence of times
between the arrivals of the $(k-1)$th and the $k$th packets to the
$j$th queue and let $\{v_{j}(k),k=1,2,...\}$ denote the sequence of
packet lengths (in bits) for the successive arrivals to queue $j$,
which is assumed to be a sequence of strictly positive i.i.d. random
variables with average packet length $1/\mu_{j}\in(0,\infty)$ and
squared coefficient of variation $\beta_{j}^{2}\in(0,\infty)$. In
addition, we suppose that all interarrival and service time
processes are mutually (conditionally) independent when the
environmental state is fixed. For each $j\in{\cal J}$ and each
nonnegative constant $h$ (in bits), we use $S_{j}(\cdot)$ to denote
the renewal counting process associated with
$\{v_{j}(k),k=1,2,...\}$, i.e.,
\begin{eqnarray}
&&S_{j}(h)=\sup\left\{n\geq 0:\sum_{k=1}^{n}v_{j}(k)\leq h\right\}.
\elabel{sjvc}
\end{eqnarray}
The reasonability about the DSRP assumption on the packet arrivals
and about the i.i.d. assumption on the packet sizes in a
communication system is due to the large-scale computer experiments
and statistical analysis conducted by Bell Labs scientists
\cite{caocle:inttra}, and recent findings by \cite{dai:contru}) and
\cite{dai:heatra}).

\subsection{A Utility-Maximization Scheduling Algorithm and Queueing
Dynamics}

First of all, we remark that the service discipline used in this
paper is the so-called head of line discipline under which the
service goes to the packet at the head of the line for a serving
queue where packets are stored in the order of their arrivals. The
service rates are determined by a function of the environmental
state and the number of packets in each of the queues. At each state
$i\in{\cal K}$ and for a given queue length vector
$q=(q_{1},...,q_{J})'$, let $\Lambda(q,i)$ denote the corresponding
rate vector (in bps) of serving the $J$ queues, which is a solution
of the following utility maximization problem
\begin{eqnarray}
&&\max_{c\in{\cal R}(i)}\sum_{j\in{\cal J}}U_{j}(q_{j},c_{j})
\elabel{srates}
\end{eqnarray}
where $c=(c_{1},...,c_{J})'$ is a $J$-dimensional vector and
$U_{j}(q_{j},c_{j})$ for each $j\in{\cal J}$ is a utility function
defined on $R_{+}^{J}$, which is second-order differentiable and
satisfies the following conditions
\begin{eqnarray}
&&U_{j}(0,c_{j})=0,\elabel{uconI}\\
&&U_{j}(q_{j},c_{j})=\Phi_{j}(q_{j})\Psi(c_{j})\;\;\mbox{is strictly
increasing and concave in}
\;\;c_{j}\;\;\mbox{for each}\;\;q_{j}>0,\elabel{uconII}\\
&&\frac{\partial U_{j}(q_{j},c_{j})}{\partial c_{j}}\;\;\mbox{is
strictly increasing in}\;\;q_{j}\geq 0,\;\; \elabel{uconIII}\\
&&\frac{\partial U_{j}(0,c_{j})}{\partial c_{j}}=0\;\;\mbox{and}\;\;
\lim_{q_{j}\rightarrow\infty}\frac{\partial
U_{j}(q_{j},c_{j})}{\partial c_{j}}=+\infty\;\;\mbox{for
each}\;\;c_{j}>0. \elabel{uconIV}
\end{eqnarray}
Due to condition \eq{uconII}, we know that there must exist an
optimal solution in the following form for a given $q$,
\begin{eqnarray}
&&\Lambda(q,i)=\left\{\begin{array}{ll} \Lambda^{{\cal
Q}(k_{1},...,k_{m})}(q,i)'&\mbox{if}\;\;q\in{\cal
Q}(k_{1},...,k_{m})\;\;
\mbox{and a given}\;\;m\in{\cal J},\\
(0,0,...,0)&\mbox{if}\;\;q=0,
\end{array}\right.
\elabel{zerozero}
\end{eqnarray}
where $k_{j}\in{\cal J}$ for each $j\in\{1,...,m\}$ and $k_{j}\neq
k_{l}$ if $j\neq l$. Moreover, ${\cal Q}(k_{1},...,k_{m})$ denotes
the set of all $q\in R_{+}^{J}$ that have exactly $m$ components
$q_{k_{j}}$ ($j\in\{1,...,m\}$) to be zero, and the components of
$\Lambda^{{\cal Q}(k_{1},...,k_{m})}(q,i)'$ corresponding to $k_{j}$
($j\in{\cal J}\setminus\{1,...,m\}$) consist of the optimal solution
to \eq{srates} with the capacity region ${\cal R}(i)$ replaced by
the corresponding $(J-m)$-user capacity region and all other
components of $\Lambda^{{\cal Q}(k_{1},...,k_{m})}(q,i)'$ are zero.
For example, when there are only two users in the system,
\eq{zerozero} is of the following form,
\begin{eqnarray}
&&\Lambda(q,i)=\left\{\begin{array}{ll}
(c_{1}(q,i),c_{2}(q,i))'&\mbox{if}\;\;q>0,\\
(c_{1}^{*}(i),0)&\mbox{if}\;\;q_{1}>0,\;\;q_{2}=0,\\
(0,c_{2}^{*}(i))&\mbox{if}\;\;q_{1}=0,\;\;q_{2}>0,\\
(0,0)&\mbox{if}\;\;q_{1}=q_{2}=0.
\end{array}\right.
\nonumber
\end{eqnarray}
\begin{remark}\label{resetzero}
The optimal solution to \eq{srates} may not be unique when $q_{j}=0$
for some $j\in{\cal J}$, however, if $\Lambda_{j}(q,i)>0$ with
$q_{j}=0$ for some $j\in{\cal J}$, we can reset $\Lambda_{j}(q,i)$
to zero without violating the constraints or decreasing the
objective value (referred to \eq{uconII}). Hence, whenever the
solution to \eq{srates} is concerned, we will always suppose that
$\eq{zerozero}$ is true. Moreover, for each $q>0$ (and similarly,
for a lower dimensional case), it follows from \eq{uconII} that
every point on the capacity surface defined in \eq{capsur} is a Nash
equilibrium point to a concave game in the sense of
\cite{ros:exiuni} and therefore the solution to \eq{srates} is a
social optimal Nash equilibrium point to the concave game.
\end{remark}

In addition, we assume that $\{U_{j}(q_{j},c_{j}),j\in{\cal J}\}$
satisfies the so-called radial homogeneity condition, i.e., for any
scalar $a>0$, each $q>0$ and each $i\in{\cal K}$, its maximizer
satisfies
\begin{eqnarray}
&&\Lambda_{j}(aq,i)=\Lambda_{j}(q,i). \elabel{homcon}
\end{eqnarray}
Interested readers are referred to \cite{yeyao:heatra} for numerous
examples of the utility function that satisfies conditions
\eq{uconI}-\eq{uconIV} and \eq{homcon}, such as, the so-called
proportional fair allocation, minimal delay allocation, and
$(\beta,\alpha)$-proportionally fair allocation, which are widely
used in communication protocols.

\subsection{The Dual Cost Minimization Problem}

In this subsection, we consider the following cost minimization
problem for each $i\in{\cal K}$, a given $c\in{\cal R}(i)$ and a
given parameter $w\geq 0$,
\begin{eqnarray}
&&\min_{q}V(q,c)
\elabel{costminp}\\
&&\mbox{s.t.}\;\;\sum_{j=1}^{J}\frac{q_{j}}{\mu_{j}}\geq w,
\nonumber\\
&&\;\;\;\;\;\;\;q_{j}\geq 0\;\;\mbox{for each}\;\;j\in{\cal J}
\nonumber
\end{eqnarray}
where the function $V$ is defined by
\begin{eqnarray}
V(q,c)=\sum_{j=1}^{J}C_{j}(q_{j},c_{j})\elabel{vcost}
\end{eqnarray}
and $C_{j}$ is the cost function associated with the utility
function $U_{j}$ in \eq{srates}, i.e.,
\begin{eqnarray}
&&C_{j}(q_{j},c_{j})=\frac{1}{\mu_{j}}\int_{0}^{q_{j}}\frac{\partial
U_{j}(u,c_{j})}{\partial c_{j}}du.\elabel{costf}
\end{eqnarray}
In other words, when the environment is in state $i\in{\cal K}$, we
try to identify a queue state $q$ corresponding to a given
$c\in{\cal R}(i)$ and a given parameter $w\geq 0$ such that the
total cost over the system is minimized and the (average) workload
meets or exceeds $w$.

\subsection{Performance Measure Processes}

Let $Q_{j}(t)$ denote the queue length for the $j$th queue with
$j\in{\cal J}$ at each time $t\in[0,\infty)$, i.e.,
\begin{eqnarray}
&&Q_{j}(t)=Q_{j}(0)+A_{j}(t)-D_{j}(t)\elabel{queuelength}
\end{eqnarray}
where $D_{j}(t)$ is the number of packet departures from the $j$th
queue in $(0,t]$, i.e., $D_{j}(t)=S_{j}(T_{j}(t))$, where
\begin{eqnarray}
&&T_{j}(t)=\int_{0}^{t}\Lambda_{j}(Q(s),\alpha(s))ds
\elabel{tjqalpha}
\end{eqnarray}
which denotes the cumulative amount of service (measured in bits)
given to the $j$th queue up to time $t$. Moreover, let $W(t)$ denote
the (expected) workload at time $t$ and $Y(t)$ denote the unused
capacity up to time $t$, i.e.,
\begin{eqnarray}
&&W(t)=\sum_{j=1}^{J}\frac{Q_{j}(t)}{\mu_{j}},\;\;
Y(t)=\sum_{j=1}^{J}\left(\int_{0}^{t}\rho_{j}(\alpha(s))ds
-T_{j}(t)\right) \elabel{wyte}
\end{eqnarray}
where, for each $i\in{\cal K}$,
$\rho(i)=(\rho_{1}(i),...,\rho_{J}(i))'$ is a given point on the
capacity surface ${\cal O}(i)$ and it is chosen to satisfy
\begin{eqnarray}
&&\sum_{j=1}^{J}\rho_{j}(i)=\max_{c\in{\cal
R}(i)}\left(\sum_{j=1}^{J}c_{j}\right)={\cal
C}_{U}\;\;\mbox{and}\;\;\rho_{1}(i)=...=\rho_{J}(i). \elabel{rhoji}
\end{eqnarray}
Here we remark that the second condition in \eq{rhoji} and the
separable condition in \eq{uconIII} are required in proving
Lemmas~\ref{fluidlemma}-\ref{uniattrack}. However, when only a
constant environment (e.g., a pseudo channel in a wireless system)
is concerned, these two conditions can be removed. Obviously,
\begin{eqnarray}
&&Y(t)\;\;\mbox{is non-decreasing in}\;\;t\geq 0
\elabel{ynondecrease}
\end{eqnarray}
since, for each $t\geq 0$, we have
\begin{eqnarray}
&&\sum_{j=1}^{J}\Lambda_{j}(Q(t),\alpha(t))\leq\sum_{j=1}^{J}
\rho_{j}(\alpha(t)). \elabel{bmaxcap}
\end{eqnarray}

\section{Main Theorem: Asymptotic Optimality}\label{apmodel}

In this section, we present the optimality result for our scheduling
policy by considering the operation of the queueing system in the
asymptotic regime where it is heavily loaded. Concretely, we define
three sequences of diffusion-scaled processes $\hat{Q}^{r}(\cdot)$,
$\hat{W}^{r}(\cdot)$ and $\hat{Y}^{r}(\cdot)$ by
\begin{eqnarray}
&&\hat{Q}_{j}^{r}(t)\equiv\frac{Q_{j}^{r}(r^{2}t)}{r},\;\;\;
\hat{W}^{r}(t)\equiv\frac{W^{r}(r^{2}t)}{r},\;\;\;\hat{Y}^{r}(t)
\equiv\frac{Y^{r}(r^{2}t)}{r} \elabel{rsqueue}
\end{eqnarray}
for each $t\geq 0$ and $j\in{\cal J}$, which associate with a
sequence of independent Markov processes
$\{\alpha^{r}(\cdot),r\in\{1,2,...\}\}$. These systems indexed by
$r$ all have the same  basic structure as described in the last
section except the arrival rates $\lambda^{r}_{j}(i)$ and the
holding time rates $\gamma^{r}(i)$ for all $i\in{\cal K}$, which may
vary with $r\in\{1,2,...\}$ and satisfy the following heavy traffic
condition
\begin{eqnarray}
&& r\left(\lambda_{j}^{r}(i)-\lambda_{j}(i)\right)
\rightarrow\theta_{j}(i)\;\;\mbox{as}\;\;r\rightarrow \infty,\;\;
\gamma^{r}(i)=\frac{\gamma(i)}{r^{2}} \elabel{heavytrafficc}
\end{eqnarray}
for each $j\in{\cal J}$, where $\theta_{j}(i)\in R$ are some
constants and $\lambda_{j}(i)\equiv\mu_{j}\rho_{j}(i)$ are the
nominal average packet arrival rates when the channel is in state
$i\in{\cal K}$.

Note that, due to the heavy traffic condition in \eq{heavytrafficc}
for the $r$th environmental state process $\alpha^{r}(\cdot)$ with
$r\in\{1,2,...\}$, we know that $\alpha^{r}(r^{2}\cdot)$ and
$\alpha(\cdot)$ equal each other in distribution since they own the
same generator matrix (see, e.g., the definition in pages 384-388 of
\cite{res:advsto}). Hence, in the sense of distribution, all of the
systems indexed by $r\in\{1,2,...\}$ in \eq{rsqueue} share the same
random environment over any time interval $[0,t]$.

Moreover, let $B^{E}(\cdot)$ and $B^{S}(\cdot)$ denote the two
independent $J$-dimensional standard Brownian motions, and for each
$i\in{\cal K}$, let
\begin{eqnarray}
\lambda(i)&=&(\lambda_{1}(i),..., \lambda_{J}(i))',
\elabel{lambdae}\\
\rho(i)&=&\left(\rho_{1}(i),...,\rho_{J}(i)\right)',\elabel{barvc}\\
\theta(i)&=&(\theta_{1}(i),...,\theta_{J}(i))', \elabel{vectorpar}\\
\Gamma^{E}(i)&=&\left(\Gamma^{E}_{kl}(i)\right)_{J\times
J}\equiv\mbox{diag}\left(\lambda_{1}(i)\alpha_{1}^{2}(i),...,
\lambda_{J}(i)\alpha^{2}_{J}(i)\right),\elabel{covmatrixo}\\
\Gamma^{S}(i)&=&\left(\Gamma^{S}_{kl}(i)\right)_{J\times
J}\equiv\mbox{diag}\left(\lambda_{1}(i)\beta_{1}^{2},...,
\lambda_{J}(i)\beta_{J}^{2}\right), \elabel{covmatrix}\\
H^{e}(t)&=&\left(H^{e}_{1}(t)',...,H^{e}_{J}(t)\right)'\;\;\mbox{with}
\;\;e\;\;\mbox{denotes}\;\;E\;\;\mbox{or}\;\;S,
\elabel{bbvector}\\
H^{e}_{j}(t)&=&\int_{0}^{t}\left(\Gamma^{e}_{jj}(\alpha(s))
\right)^{\frac{1}{2}}dB^{e}_{j}(s). \elabel{bbvectorI}
\end{eqnarray}

In addition, let $\hat{Q}^{r,G}(\cdot)$ and $\hat{W}^{r,G}(\cdot)$
denote the diffusion-scaled queue-length and workload processes
under an arbitrarily feasible rate scheduling policy $G$, e.g., a
simple Markovian policy as studied in~\cite{bhawil:difapp} or a
policy $\Lambda^{G}(Q^{r}(t),\alpha(t))$ that may not be the optimal
solution to the utility maximization problem \eq{srates}. Then we
have the following theorem.
\begin{theorem}\label{rsdifth}
Suppose $Q^{r}(0)=0$ for all $r\in\{1,2,...\}$ and the heavy traffic
condition \eq{heavytrafficc} holds, then under the scheduling policy
\eq{zerozero}, we have the claims as stated in the following two
parts:\\ {\bf Part A:} Along $r\in\{1,2,...\}$, the following
convergence in distribution is true,
\begin{eqnarray}
&&(\hat{Q}^{r}(\cdot),\hat{W}^{r}(\cdot),\hat{Y}^{r}(\cdot))
\Rightarrow(\hat{Q}(\cdot),\hat{W}(\cdot),\hat{Y}(\cdot))
\elabel{qwyweakc}
\end{eqnarray}
and the limits $\hat{Q}(\cdot),\hat{W}(\cdot)$ and $\hat{Y}(\cdot)$
are continuous a.s., which satisfy the following RDRS
\begin{eqnarray}
&&\hat{W}(t)=\hat{X}(t) +\hat{Y}(s)\geq 0 \elabel{mainmodel}
\end{eqnarray}
where
\begin{eqnarray}
&&d\hat{X}(t)=\sum_{j=1}^{J}\frac{1}{\mu_{j}}
\left(\theta_{j}(\alpha(t))dt +dH_{j}^{E}(t) +dH^{S}_{j}(t)\right)
\elabel{hatx}
\end{eqnarray}
Moreover, $(\hat{W}(\cdot),\hat{Y}(\cdot))$ is the unique solution
of \eq{mainmodel} with the following complementary property:
\begin{enumerate}
\item $\hat{Y}(0)=0$,
\item $\hat{Y}(\cdot)$ is non-decreasing,
\item $\hat{Y}(\cdot)$ can increase only at a
time $t\in[0,\infty)$ that $\hat{W}(t)=0$.
\end{enumerate}
In addition, we have
\begin{eqnarray}
&&\hat{Q}(t)=q^{*}(\hat{W}(t),\rho(\alpha(t)))\elabel{qqast}
\end{eqnarray}
with $q^{*}(w,\rho(i))$ being the solution to the cost minimization
problem \eq{costminp} in terms of each given $w$ and $i\in{\cal
K}$.\\
{\bf Part B:} The workload $\hat{W}(\cdot)$ and the cost
$\sum_{j=1}^{J}C_{j}(\hat{Q}_{j}(\cdot),\rho(\alpha(\cdot)))$ are
minimal with probability one in the sense that, for all $t\geq 0$,
\begin{eqnarray}
&&\liminf_{r\rightarrow\infty}\hat{W}^{r,G}(t)\geq\hat{W}(t),
\elabel{optimaleqn}\\
&&\liminf_{r\rightarrow\infty}\sum_{j=1}^{J}
C_{j}(\hat{Q}_{j}^{r,G}(t),\rho_{j}(\alpha(t)))
\geq\sum_{j=1}^{J}C_{j}(\hat{Q}_{j}(t),\rho_{j}(\alpha(t))).
\elabel{optimaleqnI}
\end{eqnarray}
\end{theorem}
\begin{remark}
Comparing with the RBM widely studied in queueing literature, the
RDRS model derived in \eq{mainmodel} exhibits its new feature in the
sense that it corresponds to a more realistic FS-CTMC fading process
in certain applications such as in a wireless system and indicates
that the random process $\alpha(\cdot)$ is a non-ignorable random
environmental factor to the system performance even in the limiting
approximation model. From the model, we can also see that, when a
constant environment (e.g., a quasi-static channel in a wireless
system) is concerned, the model in \eq{mainmodel} reduces to a RBM
since the state process $\alpha(\cdot)$ keeps a constant. Moreover,
by the discussions in \cite{cheyao:funque}, \cite{dai:broapp},
\cite{daidai:heatra}, and \cite{harrei:refbro}, we know that the
unique solution $(\hat{W}(t),\hat{Y}(t))$ to \eq{mainmodel} can be
represented by $(\hat{W},\hat{Y})=(\Phi(\hat{X}),\Psi(\hat{X}))$,
where $\Phi(\cdot)$ and $\Psi(\cdot)$ are Lipschitz continuous
mappings. In addition, a RDRS is different from a conventional SDE
since its drift and diffusion coefficients are not adapted to the
filtration generated by the driving Brownian motions. This type of
SDEs without boundary reflections has received a great attention in
the area of financial engineering (see, e.g., \cite{zhoyin:marmea}).
\end{remark}

\section{Applications to $J$-user MIMO Uplink and
Downlink Wireless Channels} \label{wnm}

In this section, we apply the discussions in the previous sections
to a cellular system where base stations cooperate among noise-free
infinite capacity links. We do not make any distinction between a
single-cell cellular system having multiple base-station antennas
and the traditional cellular system with cooperating single-antenna
base stations. Here, the cooperation means that the base stations
can perform joint beamforming and/or power control but there is a
constraint on the total power that the base stations can share.
Therefore, our wireless system can be considered consisting of a
base station having $M$ antennas and $J$ users (mobiles), each of
which has $N$ antennas. Thus the uplink channel can be modeled as a
$J$-user MIMO MAC and the downlink channel can be modeled as a
$J$-user MIMO BC (see, e.g., Figure~\ref{tthuserregion}).
\begin{figure}[tbh]
\centerline{\epsfxsize=3.0in\epsfbox{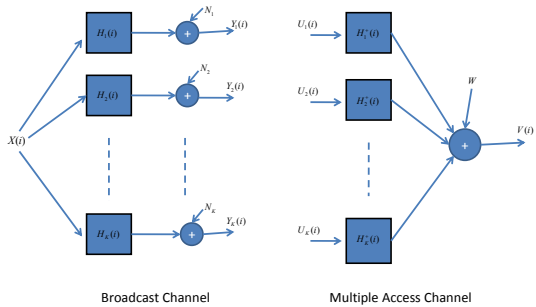}} \caption{\small The
BC and MAC channels in a particular environmental state}
\label{tthuserregion}
\end{figure}
The channel fading is supposed to obey the stationary FS-CTMC
$\alpha=\{\alpha(t),t\in[0,\infty)\}$ that is described in the
previous sections. Moreover, we suppose that the receive or transmit
end (the cooperating base stations) has perfect CSI. For each
channel state $i\in{\cal K}$, we let $H_{j}(i)$ ($j\in{\cal
J}\equiv\{1,...,J\}$) denote the downlink channel matrix from the
base station to user $j$. Assuming the same channel is used on the
uplink and downlink, then the uplink matrix of user $j$ is
$H^{\dagger}_{j}(i)$ that is the conjugate transpose of $H_{j}(i)$.

Moreover, at the transmit end, arriving packets for each user are
buffered before transmission and the rate of arrivals is a random
process that switches with the FS-CTMC channel fading through
admission control. Therefore, the processor-sharing queues presented
in the previous section can be used to model the channel dynamics
for both $J$-user MIMO MAC and $J$-user MIMO BC. The remaining issue
is about how to characterize the MAC and BC capacity region
processes, which is also a central topics in information theory
literature.

\subsection{The MIMO MAC Capacity Region}

In the MAC and for each channel state $i\in{\cal K}$, let
$U_{j}(i)\in{\cal C}^{N\times 1}$ be the transmitted signal of user
$j$, where ${\cal C}^{N\times 1}$ denotes the $N\times 1$ complex
matrix, and let $V(i)\in{\cal C}^{M\times 1}$ denote the received
signal, $W\in{\cal C}^{M\times 1}$ denote the noise vector where
$W\sim\bar{N}(0,I)$ is circularly symmetric complex Gaussian with
identity covariance (note that the notation $W$ here has the
different meaning from the workload process $W(t)$ defined in
\eq{wyte}). Then the received signal at the base station is equal to
\begin{eqnarray}
&&V(i)=H^{\dagger}(i)U'(i)+W \elabel{vhuw}
\end{eqnarray}
where $H^{\dagger}(i)=[H_{1}^{\dagger}(i),...,H_{J}^{\dagger}(i)]$
and $U(i)=[U'_{1}(i),...U'_{J}(i)]$ (see, e.g.,
Figure~\ref{tthuserregion}). Moreover, each user $j$ is subject to
an individual power constraint $P_{j}$. The transmit covariance
matrix of user $j$ is defined to be $\Gamma_{j}(i)\equiv
E[U_{j}(i)U_{j}^{\dagger}(i)]$. The power constraint implies that
Tr$(\Gamma_{j}(i))\leq P_{j}$ for $j\in{\cal J}$. During the period
of each channel state $i\in{\cal K}$, it follows from
\cite{goljaf:caplim} and \cite{yurhe:itewat} that the MAC capacity
region is a $J$-dimensional closed convex set in
$R_{+}^{J}\equiv\{c\in R^{J}:c_{j}\geq 0,j\in{\cal J}\}$, i.e.,
\begin{eqnarray}
&&{\cal R}(i)={\cal
C}_{MAC}(P_{1},...,P_{J},H^{\dagger}(i))=\elabel{macform}\\
&&\bigcup_{\{\Gamma_{j}(i)\geq 0,\mbox{Tr}(\Gamma_{j}(i))\leq
P_{j},j\in{\cal J}\}}\left\{c\in R^{J}_{+}:\sum_{j\in
S}c_{j}\leq\frac{1}{2}\mbox{log}\left|I+\sum_{j\in
S}H_{j}^{\dagger}(i)\Gamma_{j}(i)H_{j}(i)\right|,\forall \;\;
S\subset{\cal J}\right\} \nonumber
\end{eqnarray}
where $S$ is a subset of ${\cal J}$ and $|\cdot|$ denotes the
determinant of a matrix. Moreover, every point in ${\cal R}(i)$ can
be achieved by Shannon's source coding theorem and successive
decoding (see, e.g., \cite{gamcov:muluse} and \cite{goljaf:caplim}).
However, in designing a utility maximization based rate scheduling
policy, we need to know more detailed boundary characterization of
the MAC capacity region since it frequently relies on the KKT
optimality conditions (see, e.g., \cite{lue:linnon} and
\cite{liuhou:weipro}). Thus we have the following lemma.
\begin{lemma}\label{smoothsurfaces}
For the $J$-user MIMO MAC and each channel state $i\in{\cal K}$,
${\cal R}(i)$ contains the origin and has $L$ linear or smooth
curved facets with $L$ given by
\begin{eqnarray}
&&L=J!+\sum_{j=2}^{J}C_{J}^{j}(J-j+1)!+J.\elabel{numberL}
\end{eqnarray}
Moreover, $J$ of these pieces are $(J-1)$-dimensional linear facets
along the coordinate axes while the remaining $B=L-J$ ones are in
the interior of $R^{J}_{+}$ and form ${\cal O}(i)$, which are linear
or smooth curved facets $h_{k}(c,i)$ on $R_{+}^{J}$ for $k\in{\cal
U}\equiv\{1,2,...,B\}$, i.e.,
\begin{eqnarray}
&&{\cal R}(i)\equiv\left\{c\in R_{+}^{J}:\;h_{k}(c,i)\leq
0,\;k\in{\cal U}\right\}. \elabel{capsuro}
\end{eqnarray}
Moreover, if $C_{MAC}(P,H(i))$ is used to denote the sum capacity
upper bound for the MAC capacity region, then
\begin{eqnarray}
&&h_{k_{MAC}}(c,i)=\sum_{j=1}^{J}c_{j}-C_{MAC}(P,H(i))\elabel{gmiddle}
\end{eqnarray}
where $k_{MAC}\in{\cal U}$ is the index corresponding to
$C_{MAC}(P,H(i))$.
\end{lemma}
\begin{example}\label{twomace}
For the MAC channel and each $i\in{\cal K}$, when $J=2$ and $N=1$
(i.e., each of the user's mobiles has only single transmit antenna),
it follows from~\cite{goljaf:caplim} that
\begin{eqnarray}
&&g_{1}(c,i)=c_{1}-\log\left|I+H_{1}^{\dagger}(i)P_{1}H_{1}(i)\right|,
\nonumber\\
&&g_{2}(c,i)=c_{1}+c_{2}-\log\left|I+H_{1}^{\dagger}(i)P_{1}H_{1}(i)
+H_{2}^{\dagger}(i)P_{2}H_{2}(i)\right|,\nonumber\\
&&g_{3}(c,i)=c_{2}-\log\left|I+(I+H_{1}^{\dagger}(i)P_{1}H_{1}(i))^{-1}
H_{2}^{\dagger}(i)P_{2}H_{2}(i)\right|.\nonumber
\end{eqnarray}
\end{example}

\subsection{The MIMO BC Capacity Region}

In the MIMO BC and for each channel state $i\in{\cal K}$, let
$X(i)\in{\cal C}^{M\times 1}$ denote the transmitted vector signal
from the base station and let $Y_{j}(i)\in{\cal C}^{N\times 1}$ be
the received signal at the user $j$. The noise at user $j$ is
represented by $N_{j}\in{\cal C}^{N\times 1}$ and is assumed to be
circularly symmetric complex Gaussian noise $(N_{j}\sim N(0,I))$.
The received signal of user $j$ (see, e.g.,
Figure~\ref{tthuserregion}) is equal to
\begin{eqnarray}
&&Y_{j}(i)=H_{j}(i)X(i)+N_{j}.\nonumber
\end{eqnarray}
The transmit covariance matrix of the input signal is
$\Gamma_{X}(i)\equiv E\left[X(i)X^{\dagger}(i)\right]$. The base
station is subject to an average power constraint, which implies
that Tr($\Gamma_{X}(i))\leq P$. During each channel state $i\in{\cal
K}$, the $J$-user MIMO BC capacity region denoted by ${\cal R}(i)$
can be calculated by the duality of the MAC and the BC in
\cite{jinvis:duagau} and \cite{goljaf:caplim}, where the BC capacity
region is obtained by taking the convex hull of the union over the
set of capacity regions of the dual MIMO MACs such that the total
MAC power is the same as the power in the BC, i.e.,
\begin{eqnarray}
&&{\cal R}(i)={\cal C}_{BC}(P,H(i))
=\bigcup_{\{(P_{1},...,P_{J}):\sum_{j=1}^{J}P_{j}=P\}} {\cal
C}_{MAC}(P_{1},...,P_{J},H^{\dagger}(i)). \elabel{bccapcityI}
\end{eqnarray}
Moreover, the Dirty Paper Coding (DPC) proposed in \cite{cos:wridir}
achieves the capacity for the MIMO BC (see, e.g.,
\cite{weiste:capreg}). In particular, if each user has only single
receive antenna, we have the following lemma.
\begin{lemma}\label{smoothsurfacesI}
For the $J$-user MIMO BC with $N=1$, each $i\in{\cal K}$ and $L$
given in \eq{numberL}, ${\cal R}(i)$ contains the origin and has $L$
boundary pieces of which $J$ are ($(J-1)$-dimensional linear facets
along the coordinate axes while the remaining $B=L-J$ ones are in
the interior of $R^{J}_{+}$ and form ${\cal O}(i)$, which are linear
or smooth curved facets $h_{k}(c,i)$ on $R_{+}^{J}$ for $k\in{\cal
U}\equiv\{1,2,...,B\}$, i.e.,
\begin{eqnarray}
&&{\cal R}(i)\equiv\left\{c\in R_{+}^{J}:\;h_{k}(c,i)\leq
0,\;k\in{\cal U}\right\}. \elabel{capsurBC}
\end{eqnarray}
Moreover, if $C_{Sato}(P,H(i))$ denotes the sum capacity upper bound
(called the Sato upper bound) for the BC capacity region, then
\begin{eqnarray}
&&h_{k_{Sato}}(c,i)=\sum_{j=1}^{J}c_{j}-C_{Sato}(P,H(i))\elabel{hmiddle}
\end{eqnarray}
where $k_{Sato}\in{\cal U}$ is the index corresponding to
$C_{Sato}(P,H(i))$.
\end{lemma}
\begin{remark}
The concept of {\em Sato upper bound} of the $J$-user MIMO BC can be
found in such as \cite{goljaf:caplim}, \cite{caisha:achrat},
\cite{vistse:sumcap}, \cite{visjin:capmul}, \cite{yucio:sumcap}.
Moreover, here we conjecture that, in the MIMO BC channel, if each
user has multiple receive antennas, the corresponding property
stated in the lemma should also be true.
\end{remark}
\begin{example}\label{twobce}
Considering the BC channel with $J=2$ and $N=1$ for each $i\in{\cal
K}$, we can derive $h_{k}(c,i)$ explicitly for $k=1,2,3$ by
employing the results in \cite{goljaf:caplim}, \cite{visjin:duaach},
\cite{weiste:capreg} as follows
\begin{eqnarray}
h_{1}(c,i)&=&e^{2(c_{1}+c_{2})}-\left|I
+\left(H_{1}^{\dagger}(i)H_{1}(i)-H_{2}^{\dagger}(i)H_{2}(i)\right)
\right.\elabel{hfuncI}\\
&&\frac{e^{2c_{1}}-1}{|H_{111}|^{2}(i)+|H_{122}|^{2}(i)}
+\left.H_{2}^{\dagger}(i)H_{2}(i)P \right|^{2}, \nonumber\\
h_{2}(c,i)&=&c_{1}+c_{2}-{\cal C}_{\mbox{Sato}}(P,H(i)),\elabel{hfuncII}\\
h_{3}(c,i)&=&e^{2(c_{1}+c_{2})}-\left|I
+\left(H_{2}^{\dagger}(i)H_{2}(i)-H_{1}^{\dagger}(i)H_{1}(i)\right)
\right.\elabel{hfuncIII}\\
&&\frac{e^{2c_{2}}-1}{|H_{211}|^{2}(i)+|H_{222}|^{2}(i)}
+\left.H_{1}^{\dagger}(i)H_{1}(i)P\right|^{2}.\nonumber
\end{eqnarray}
\end{example}

\section{Proof of Theorem~\ref{rsdifth}}\label{proofmains}

To be convenient for readers, we first outline the proof of
Theorem~\ref{rsdifth}, which consists of the following five parts.

Firstly, in Subsection~\ref{prelim}, we first justify a dual
relationship between the utility-maximization problem in \eq{srates}
and the cost-minimization problem in \eq{costminp}, which is
summarized in Lemma~\ref{equioptso}. Then we prove a claim in
Lemma~\ref{costcc}, which states that when the system state is close
to the unique optimal solution to the cost minimization problem
(called {\it a fixed point}), the capacity of the system will be
fully utilized. The claims stated in
Lemmas~\ref{equioptso}-\ref{costcc} are similar to their
counterparts in \cite{yeyao:heatra}, nevertheless, their concrete
proofs are different due to the different problem formulations and
the difference of the capacity constraints between the two studies.

Secondly, in Subsection~\ref{equidis}, we present an equivalent
queueing model due to the assumption \eq{heavytrafficc} imposed on
the FS-CTMC and justify a functional central limit theorem
(Lemma~\ref{primconlm}) for a DSRP whose arrival rate process is
driven by the FS-CTMC. The main idea used in proving
Lemma~\ref{primconlm} is stemmed from the related discussion in
\cite{dai:broapp}, \cite{daidai:heatra} and the concrete proving
techniques include the conventional functional central limit theorem
(see, e.g., \cite{iglwhi:equfun} and \cite{pro:conran}), random
change of time lemma (see, e.g., \cite{bil:conpro}), establishment
of oscillation inequality (see, e.g., \cite{dai:broapp},
\cite{daidai:heatra}), equivalent conditions of relative compactness
and Skorohod representation theorem (see, e.g.,
\cite{ethkur:marpro}), and etc.

Thirdly, in Subsection~\ref{fluidlim}, we derive the fluid limit
processes for the physical processes under fluid scaling in
Lemma~\ref{fluidlemma} and study the asymptotic behavior for the
fluid limit processes as time evolves in Lemma~\ref{uniattrack}.
Fluid limits are widely used as an intermediate step in justifying
diffusion approximations (see, e.g., \cite{bradai:heatra},
\cite{sto:maxsch}, \cite{yeyao:heatra}, \cite{dai:difapp},
\cite{bhawil:peruse}, \cite{bhawil:difapp}, and references therein).
Nevertheless, our fluid limit is a random process driven by the
FS-CTMC rather than a deterministic function of time as obtained in
the existing studies. This new feature brings us additional
complexity in proving Lemma~\ref{fluidlemma} and
Lemma~\ref{uniattrack}, e.g., comparing with the study in
\cite{yeyao:heatra}, it requires more technical treatment in
handling the FS-CTMC based jumps for the constructed Lyapunov
function. Therefore, by noticing this new feature and the difference
between our optimal scheduling policy and the one in
\cite{yeyao:heatra}, we develop a theory through combining and
generalizing the discussions in \cite{yeyao:heatra},
\cite{dai:poshar}, \cite{bhawil:peruse}, and \cite{bhawil:difapp} to
finish the justifications of Lemma~\ref{fluidlemma} and
Lemma~\ref{uniattrack}.

Fourthly, in Subsection~\ref{slowerfs}, we study the convergence of
the workload and queue length processes on a finer time-scale, which
is an important step in justifying the main result of the paper.
This method has appeared in queueing literature for a while (see,
e.g., \cite{bra:staspa}, \cite{yeyao:heatra}, \cite{sto:maxsch},
\cite{mansto:schfle}, \cite{shasri:patopt}, and etc.) The main
difference between ours and the existing works is as follows: all
the processes concerned in our study involve the jumps introduced by
the random environment and in the meanwhile the processes in
existing studies do not involve this type of jumps. Therefore we
develop a scheme and incorporate it into the framework as used in
\cite{yeyao:heatra} to finish the proof of the convergence
properties for the processes on a finer time-scale.

Finally, in Subsection~\ref{mainthp}, we combine the results
obtained in the previous subsections with the uniqueness of solution
to an associated Skorohod problem and the minimality of the Skorohod
problem to provide a proof for Theorem~\ref{rsdifth}. This type of
techniques have been used in the studies concerning network
scheduling (see, e.g., \cite{yeyao:heatra}, \cite{sto:maxsch},
\cite{mansto:schfle}, \cite{shasri:patopt}, and etc.) Nevertheless,
our justification logic and technical treatment are somewhat
different.

\subsection{Preliminary Lemmas on the Utility-Maximization and
Dual Cost Minimization Problems}\label{prelim}

\begin{lemma}\label{firstl}
Consider the utility-maximization problem in \eq{srates} and suppose
conditions \eq{uconI}-\eq{uconIV} and \eq{homcon} are imposed, then
for a sequence of queue states, $\{q^{l},l=1,2,...\}$, which
satisfies $q^{l}\rightarrow q\in R_{+}^{J}$ as $l\rightarrow\infty$,
we have
\begin{eqnarray}
&&\Lambda_{j}(q^{l},i)\rightarrow\Lambda_{j}(q,i)\;\;
\mbox{as}\;\;l\rightarrow\infty \elabel{ralloccon}
\end{eqnarray}
for each $i\in{\cal K}$ and any $j\in{\cal J}$ such that $q_{j}>0$.
\end{lemma}
\begin{proof} Consider each specific state $i\in{\cal K}$, then the proof
can be accomplished similarly as for Lemma 6.2 in~\cite{yeou:stadat}
and hence we omit it. $\Box$
\end{proof}

\begin{lemma}\label{equioptso}
For each state $i\in{\cal K}$, the following claims are true.
\begin{enumerate}
\item Under the policy in \eq{zerozero} and the associated convention,
if $c^{*}=\rho(i)$ is an optimal solution to the maximization
problem in \eq{srates} with a queue state $q$ in the utility
function. Then $q^{*}=q$ must be the optimal solution to the
minimization problem in \eq{costminp} with $c=c^{*}$ in the cost
function and with $w=\sum_{j=1}^{J}q^{*}_{j}/\mu_{j}$ in the
constraints.
\item Conversely, if $q^{*}$ is the optimal solution to the minimization
problem in \eq{costminp} with $w>0$ and $\Lambda(q^{*},i)=\rho(i)$
for each $i\in{\cal K}$ in the cost function. Then $q^{*}>0$ and
$\Lambda^{*}(q^{*},i)=\Lambda(q^{*},i)$ must be an optimal solution
to the maximization problem in \eq{srates} with $q=q^{*}$ in the
utility function.
\end{enumerate}
\end{lemma}
\begin{proof} First of all, without loss of generality, we suppose
that $q>0$. Then it follows from the KKT optimality conditions (see,
e.g., \cite{lue:linnon}) that the solution to the utility
maximization problem in \eq{srates} can be obtained through the
following equations,
\begin{eqnarray}
&&c_{j}\left(\frac{\partial U_{j}(q_{j},c_{j})}{\partial
c_{j}}+\sum_{k=1}^{B}\eta_{k}\frac{\partial h_{k}(c,i)}{\partial
c_{j}}
\right)=0\;\;\mbox{for}\;\;j\in{\cal J},\elabel{uoptsoI}\\
&&\eta_{k}h_{k}(c,i)=0\;\;\mbox{for each}\;\;k\in{\cal U}
\elabel{uoptsoII}
\end{eqnarray}
where $B$ and ${\cal U}$ are defined in \eq{capsur}, $\eta_{k}\geq
0$ for all $k\in{\cal U}$ are the Lagrangian multipliers and
$h_{k}(c,i)$ for each $k\in{\cal U}$ and $i\in{\cal K}$ is defined
in \eq{capsur}. Similarly, the solution to the cost minimization
problem \eq{costminp} can be obtained through the following
equations,
\begin{eqnarray}
&&q_{j}\left(\frac{\partial C_{j}(q_{j},c_{j})}{\partial
q_{j}}+\frac{\theta}{\mu_{j}}\right)=0\;\;\mbox{for
each}\;\;j\in{\cal J},\elabel{optsoII}\\
&&\theta\left(w-\sum_{j=1}^{J}\frac{q_{j}}{\mu_{j}}\right)=0,
\elabel{optsolIII}
\end{eqnarray}
where $\theta\geq 0$ is the Lagrangian multiplier. Moreover, it
follows from \eq{costf} that
\begin{eqnarray}
&&\frac{\partial C_{j}(q_{j},c_{j})}{\partial
q_{j}}=\frac{1}{\mu_{j}}\frac{\partial U_{j}(q_{j},c_{j})}{\partial
c_{j}}. \elabel{difequal}
\end{eqnarray}
Thus, based on the above facts, the claim in the first part of the
lemma can be proved as follows.
By condition \eq{uconII}, we know that
$\sum_{j=1}^{J}U_{j}(q_{j},c_{j})$ is strictly concave in $c$ for
each $q>0$. Therefore $c^{*}=\rho(i)$ is the unique optimal solution
to the utility maximization problem in \eq{srates} for the given
$q>0$ in the utility function, which satisfies
\eq{uoptsoI}-\eq{uoptsoII}. Thus, if we take
\begin{eqnarray}
&&\theta=-\sum_{k=1}^{B}\eta_{k}\frac{\partial
h_{k}(\rho(i),i)}{\partial c_{j}},\nonumber
\end{eqnarray}
then it follows from \eq{uoptsoI} and \eq{difequal} that
\eq{optsoII} holds. Due to condition \eq{uconIV}, we know that
$V(q,c)$ is strictly convex in $q$ for each $c>0$. So the cost
minimization problem in \eq{costminp} has a unique optimal solution
$q^{*}=q$ when $c=c^{*}=\rho(i)$ is in the cost function and
$w=\sum_{j=1}^{J}q^{*}_{j}/\mu_{j}$ is in the constraints.

Conversely, the claim in the second part of the lemma can be proved
as follows. Due to the conditions \eq{uconIII}-\eq{uconIV} and the
relationship \eq{costf}, we know that $V(q,\rho(i))$ is strictly
convex in $q$. Therefore $q^{*}$ is the unique optimal solution to
the cost minimization problem \eq{costminp} with
$\Lambda(q^{*},i)=\rho(i)$. Thus we can prove $q^{*}>0$ by showing a
contradiction.

In fact, without loss of generality, we suppose that there is some
$m\in{\cal J}$ with $m<J$ such that $q^{*}\in{\cal
Q}(k_{1},...,k_{m})$ with $k_{1}\neq 1$ and $k_{m}=J$, where ${\cal
Q}(k_{1},...,k_{m})$ is defined in \eq{zerozero}. Then we can
construct a $2$-dimensional line for some constant $\epsilon\geq w$,
\begin{eqnarray}
&&P_{1}:\;\;\frac{q_{1}}{\mu_{1}}+\frac{q_{J}}{\mu_{J}}+\sum_{j\neq
1,J,j\in{\cal J}}\frac{q^{*}_{j}}{\mu_{j}}=\epsilon\geq w
\elabel{planeI}
\end{eqnarray}
such that it passes through the point $q^{*}$. Now it follows from
\eq{costf} that the function $f(q_{1},\rho(i))$ ($=V(q,\rho(i))$)
with the constraint $P_{1}$ for all
$q=(q_{1},q_{2}^{*},...,q_{J-1}^{*},q_{J})'\in{\cal R}_{+}^{J}$) is
of the following derivative function in $q_{1}\in{\cal R}_{1}^{+}$,
\begin{eqnarray}
\;\;\;\;\frac{\partial f(q_{1},\rho(i))}{\partial
q_{1}}&=&\frac{1}{\mu_{1}}\frac{\partial
U_{1}(q_{1},\rho_{1}(i))}{\partial
c_{1}}-\frac{1}{\mu_{1}}\frac{\partial
U_{J}((\epsilon-\frac{q_{1}}{\mu_{1}}-\sum_{j\neq 1,J,j\in{\cal
J}}\frac{q^{*}_{j}}{\mu_{j}})\mu_{J}, \rho_{J}(i))}{\partial c_{J}},
\elabel{diffine}
\end{eqnarray}
which is strictly increasing in $q_{1}\in{\cal R}_{1}^{+}$ due to
\eq{uconIII}. Moreover, it follows from \eq{diffine} and \eq{uconIV}
that
\begin{eqnarray}
&&\frac{\partial f(0,\rho(i))}{\partial
q_{1}}=-\frac{1}{\mu_{1}}\frac{\partial U_{J}((\epsilon-\sum_{j\neq
1,J,j\in{\cal J}}\frac{q^{*}_{j}}{\mu_{j}})\mu_{J},
\rho_{J}(i))}{\partial c_{J}}<0,
\elabel{negderiv}\\
&&\frac{\partial f(q_{1}^{*},\rho(i))}{\partial
q_{1}}=\frac{1}{\mu_{1}}\frac{\partial
U_{1}(q^{*}_{1},\rho_{1}(i))}{\partial c_{1}}>0. \elabel{posderiv}
\end{eqnarray}
Then, by \eq{negderiv} and \eq{posderiv}, we know that there is a
$\tilde{q}_{1}\in(0,q_{1}^{*})$ such that
\begin{eqnarray}
&&\frac{\partial f(\tilde{q}_{1},\rho(i))}{\partial q_{1}}=0
\end{eqnarray}
which implies that, on the curve $f(q,\rho(i))$ with
$q=(q_{1},q_{2}^{*},...,q_{J-1}^{*},q_{J})'\in R^{J}_{+}$, there
exists a minimal point $\tilde{q}\in{\cal R}_{+}^{J}$ with
$\tilde{q}=(\tilde{q}_{1},q_{2}^{*},...,q_{J-1}^{*},\tilde{q}_{J})'$
such that $V(\tilde{q},\rho(i))<V(q^{*},\rho(i))$
\begin{eqnarray}
&&\tilde{q}_{J}=\left(\epsilon-\frac{\tilde{q}_{1}}{\mu_{1}}
-\sum_{j\neq 1,J,j\in{\cal
J}}\frac{q^{*}_{j}}{\mu_{j}}\right)\mu_{J}. \nonumber
\end{eqnarray}
This contradicts the assumption that $q^{*}$ is the optimal solution
to the cost minimization problem in \eq{costminp}. Hence we can
conclude that $q^{*}>0$.

Finally, if $q^{*}$ is the optimal solution to \eq{costminp} with
$c=\Lambda(q^{*},i)=\rho(i)$ in the cost function, we see that
\eq{optsoII}-\eq{optsolIII} hold with $q=q^{*}$ and
$c=\Lambda(q^{*},i)=\rho(i)$. Therefore we can take
$\eta_{k_{U}}=\theta$ and $\eta_{k}=0$ when $k\neq k_{U}$ in
\eq{uoptsoI}-\eq{uoptsoII} since $q^{*}>0$ and $\rho$ is on the
curve $h_{k_{U}}(c,i)=0$. Hence
$\Lambda^{*}(q^{*},i)=\Lambda(q^{*},i)=\rho(i)$ for each $i\in{\cal
K}$ is an optimal solution to \eq{srates} with $q=q^{*}$ in the
utility function. $\Box$
\end{proof}

\vskip 0.15cm Next, let $\|\cdot\|$ denote the norm of a vector
$q\in R^{J}$ in the sense that $\|q\|=\sum_{j=1}^{J}|q_{j}|$. Then
we have the following lemma.
\begin{lemma}\label{costcc}
For each state $i\in{\cal K}$, the following claims are true.
\begin{enumerate}
\item The cost minimization problem \eq{costminp} has a unique
optimal solution $q^{*}=q$ when $c=\rho(i)$ is in the cost function,
and moreover, $q^{*}(w,\rho(i))$ is continuous in terms of $w$.
\item Assuming that, for any given constant
$\epsilon>0$, there exists another constant $\sigma>0$, which
depends only on $\epsilon$, such that, for any $q\in{\cal
V}(\epsilon,\sigma,i)$ with
\begin{eqnarray}
&&{\cal V}(\epsilon,\sigma,i)\equiv\left\{q\in
R_{+}^{J}:\left\|q-q^{*}(w,\rho(i))
\right\|\leq\sigma\;\;\mbox{and}\;\;w=\sum_{j=1}^{J}
\frac{1}{\mu_{j}}q_{j}\;\;\geq\epsilon\right\}, \elabel{qqwi}
\end{eqnarray}
we have
\begin{eqnarray}
&&\sum_{j=1}^{J}\Lambda_{j}(q,i)=\sum_{j=1}^{J}\rho_{j}(i).
\elabel{sumrhoi}
\end{eqnarray}
\end{enumerate}
\end{lemma}
\begin{remark}
The unique optimal solution $q^{*}(w,\rho(i))$ will be referred to
as a {\em fixed point} in the following discussion.
\end{remark}
\begin{proof} For the first part in the lemma, we have the following
observations. Due to condition \eq{uconIV}, we know that $V(q,c)$ is
strictly convex in $q$ for each $c>0$. So the cost minimization
problem in \eq{costminp} has a unique optimal solution $q^{*}=q$
when $c=\rho(i)$ is in the cost function. Moreover, the continuity
of $q^{*}(w,\rho(i))$ in terms of $w$ for each $i\in{\cal K}$ can be
proved similarly as in \cite{yeyao:heatra}.

For the second part in the lemma, it can be proved by showing a
contradiction. As a matter of fact, if the claim is not true for
some $i\in{\cal K}$ and some $\epsilon>0$, then for a sequence of
$\sigma^{l}\downarrow 0$ along $l\in\{1,2,...\}$, there is a
sequence of states $\tilde{q}^{l}\in{\cal V}(\epsilon,\sigma^{l},i)$
with $l=1,2,...$ satisfying
\begin{eqnarray}
&&\left\|\tilde{q}^{l}-q^{*}(\tilde{w}^{l},\rho(i))\right\|\rightarrow
0\;\;\mbox{as}\;\;l\rightarrow\infty,
\elabel{contradI}\\
&&\tilde{w}^{l}=\sum_{j=1}^{J}\frac{1}{\mu_{j}}\tilde{q}_{j}^{l}
\geq\epsilon\;\;\mbox{for all}\;\;l=1,2,... \elabel{contradII}
\end{eqnarray}
such that
\begin{eqnarray}
&&\sum_{j=1}^{J}\Lambda_{j}(\tilde{q}^{l},i)<\sum_{j=1}^{J}\rho_{j}(i)\;\;
\mbox{for all}\;\;l=1,2,... \elabel{contradIII}
\end{eqnarray}
Otherwise, if there is some $l_{0}\in\{1,2,...\}$ such that ${\cal
V}(\epsilon,\sigma^{l},i)$ are empty for all $l\geq l_{0}$, then
\eq{sumrhoi} is automatically true for the given $i\in{\cal K}$ and
$\epsilon>0$, which is a contradiction. Now, let
\begin{eqnarray}
&&q^{l}=(\epsilon/\tilde{w}^{l})\tilde{q}^{l}\;\;\mbox{so
that}\;\;w^{l}=\sum_{j=1}^{J}\frac{q_{j}^{l}}{\mu_{j}}=\epsilon\;\;
\mbox{for all}\;\;l=1,2,..., \elabel{contradIV}
\end{eqnarray}
then it follows from \eq{contradI}-\eq{contradII}, \eq{contradIV},
\eq{homcon} and Lemma~\ref{equioptso} that, along $l\in\{1,2,...\}$,
\begin{eqnarray}
&&q^{l}=\frac{\epsilon
q^{*}(\tilde{w}^{l},\rho(i))}{\tilde{w}^{l}}+\frac{\epsilon}{\tilde{w}^{l}}
\left(\tilde{q}^{l}-q^{*}(\tilde{w}^{l},\rho(i))\right)\rightarrow
\hat{q}=q^{*}(\epsilon,\rho(i))>0, \elabel{qlqsconI}
\end{eqnarray}
which implies that $q^{l}>0$ for all large enough $l\in\{1,2,...\}$,
and moreover, we also have
\begin{eqnarray}
&&\frac{\epsilon
q^{*}(\tilde{w}^{l},\rho(i))}{\tilde{w}^{l}}\rightarrow \hat{q}.
\elabel{qlqsconII}
\end{eqnarray}
Hence, by \eq{qlqsconI}, \eq{uconII}, \eq{zerozero}, and the similar
proof as used for the second part of Lemma~\ref{equioptso}, we have,
for all large enough $l\in\{1,2,...\}$,
\begin{eqnarray}
&&\Lambda(\hat{q},i)>0\;\;\;\mbox{and}\;\;\; \Lambda(q^{l},i)>0.
\elabel{doubgeqo}
\end{eqnarray}
Furthermore, by \eq{homcon} and \eq{contradIII}, we have, for each
$l\in\{1,2,...\}$,
\begin{eqnarray}
&&\sum_{j=1}^{J}\Lambda_{j}(q^{l},i)<\sum_{j=1}^{J}\rho_{j}(i).
\elabel{qlqsconII}
\end{eqnarray}
Thus it follows from \eq{qlqsconII} and Lemma~\ref{firstl} that
\begin{eqnarray}
&&\sum_{j=1}^{J}\Lambda_{j}(\hat{q},i)\leq\sum_{j=1}^{J}\rho_{j}(i).
\elabel{hatqlqscon}
\end{eqnarray}
Notice that the condition in \eq{uconII} and the fact in
\eq{qlqsconI} imply that $\Lambda(\hat{q},i)$ and $\Lambda(q^{l},i)$
for large enough $l$ can only locate on the capacity surface of
${\cal R}(i)$ (that is defined in \eq{capsur}). Then, by combining
this fact with \eq{qlqsconII}-\eq{hatqlqscon} and
Lemma~\ref{firstl}, we can see that $\Lambda(\hat{q},i)$ can not be
in the interior of the facet corresponding to $C_{U}$. Hence we can
conclude that there is some $j\in{\cal J}$, e.g., without loss of
generality, take $j=1$ such that
\begin{eqnarray}
&&\Lambda_{1}(\hat{q},i)<\rho_{1}(i).\elabel{compI}
\end{eqnarray}

So, on the one hand, it follows from \eq{doubgeqo} and \eq{uoptsoI}
that there exists a set of Lagrange multipliers
$\{\eta_{jk}\geq0,k\in\{1,...,B+J\},j\in{\cal J}\}$ such that
\begin{eqnarray}
\sum_{j=1}^{J}\sum_{k=1}^{B}\eta_{k}\frac{\partial
h_{k}(\Lambda(\hat{q},i),i)}{\partial c_{j}} &=&-\sum_{j=1}^{J}
\frac{\partial U_{j}(\hat{q}_{j},\Lambda_{j}(\hat{q},i))}{\partial
c_{j}}
\elabel{etainf}\\
&=&-\lim_{l\rightarrow\infty}\sum_{j=1}^{J} \frac{\partial
U_{j}(\hat{q}_{j},\Lambda_{j}(q^{*}(\tilde{w}^{l},\rho(i)),i))}
{\partial c_{j}}
\nonumber\\
&=&-\sum_{j=1}^{J} \frac{\partial U_{j}(\hat{q}_{j},
\rho_{j}(i))}{\partial c_{j}}.\nonumber
\end{eqnarray}
where  $B$ is defined in \eq{capsur}, the first equality of
\eq{etainf} follows from \eq{uoptsoI}-\eq{uoptsoII} and
\eq{doubgeqo}, the second equality follows from \eq{qlqsconII},
\eq{homcon} and \eq{ralloccon}, and the third equality follows from
Lemma~\ref{equioptso}.

On the other hand, due to the strict concavity of
$U_{j}(q_{j},c_{j})$ in $c_{j}$ for each $j\in{\cal J}$ as stated in
\eq{uconII}, it follows from \eq{doubgeqo}-\eq{compI} that
\begin{eqnarray}
&&\sum_{j=1}^{J}\sum_{k=1}^{B}\eta_{k}\frac{\partial
h_{k}(\Lambda(\hat{q},i),i)}{\partial c_{j}} \elabel{etainfII}\\
&<&-\frac{\partial U_{1}(\hat{q}_{1},\rho_{1}(i))}{\partial
c_{1}}-\lim_{l\rightarrow\infty}\sum_{j\neq 1,j\in{\cal
J}}\frac{\partial
U_{j}(\hat{q}_{j},\Lambda_{j}(q^{*}(\tilde{w}^{l},\rho(i)),i))}
{\partial c_{j}}\nonumber\\
&=&-\sum_{j=1}^{J} \frac{\partial U_{j}(\hat{q}_{j},
\rho_{j}(i))}{\partial c_{j}}.\nonumber
\end{eqnarray}

Obviously, there is a contradiction between \eq{etainf} and
\eq{etainfII}. Thus the assumption stated in
\eq{contradI}-\eq{contradIII} is not true, which implies that the
second claim in the lemma holds for the third case. Hence we finish
the proof of Lemma~\ref{costcc}. $\Box$
\end{proof}

\subsection{Equivalent Processes in Distribution}\label{equidis}

First of all, we define a sequence of jump times in terms of the
FS-CTMC process $\alpha(\cdot)$ as follows,
\begin{eqnarray}
&&\tau_{0}\equiv
0,\;\;\tau_{n}\equiv\inf\{t>\tau_{n-1}:\alpha(t)\neq\alpha(t^{-})\}.
\elabel{markovjt}
\end{eqnarray}
So it follows from Proposition 5.2.1 in page 376 of
\cite{res:advsto} that $\tau_{n}\rightarrow\infty$ a.s. as
$n\rightarrow\infty$ and each sample path of $\alpha(\cdot)$ has at
most finitely many jump points over any bounded interval $[0,T]$
since the state space of $\alpha(\cdot)$ is finite. Moreover, we
introduce a new stochastic process $\alpha_{\cdot}(\cdot)$ induced
by $\alpha(\cdot)$ up to each given time $t\in[0,\infty)$, i.e.,
$\alpha_{t}(s)=\alpha(s)$ with $s\in[0,t]$. In addition, for each
$j\in\{1,...,J\}$, let
\begin{eqnarray}
&&A^{r}_{j}(r^{2}\cdot,\alpha_{\cdot}(\cdot))\equiv
\left\{A^{r}_{j}(r^{2}t,\alpha_{t}(\cdot)),t\in[0,\infty)\right\}
\elabel{arjrt}
\end{eqnarray}
denote the counting process that the arrival rate corresponding to
$A^{r}_{j}(r^{2}\cdot,\alpha_{\cdot}(\cdot))$ during time interval
$[r^{2}\tau_{n},r^{2}\tau_{n+1})$ is
$\lambda^{r}_{j}(\alpha(\tau_{n}))$ for $n\in\{0,1,2,...,\}$.
Similarly, let
\begin{eqnarray}
&&A^{r}_{j}(r^{2}\cdot,\alpha^{r}_{r^{2}\cdot}(\cdot))\equiv
\left\{A^{r}_{j}(r^{2}t,\alpha^{r}_{r^{2}t}(\cdot)),t\in[0,\infty)\right\}
\elabel{aalpt}
\end{eqnarray}
denote the corresponding process with arrival rate
$\lambda^{r}_{j}(\alpha^{r}(\tau^{r}_{n}))$ during time interval
$[\tau^{r}_{n},\tau^{r}_{n+1})$ for each $n\in\{0,1,2,...,\}$, where
$\{\tau_{n}^{r},n\in\{0,1,...,\}\}$ is a sequence of jump times in
terms of $\alpha^{r}(\cdot)$. Due to the second condition in
\eq{heavytrafficc}, the definition of DSRP, the assumptions among
the arrival and FS-CTMC fading processes, and Theorem 5.4 in page 85
of \cite{kal:foumod}, we have,
\begin{eqnarray}
&&E^{r}_{j}(\cdot)\equiv A_{j}^{r}(r^{2}\cdot,\alpha_{\cdot}(\cdot))
=^{d}A^{r}_{j}(r^{2}\cdot,\alpha^{r}_{r^{2}\cdot}(\cdot)),
\elabel{enee}
\end{eqnarray}
where the notation $=^{d}$ denotes "{\em equals in distribution}".
Then it follows from \eq{queuelength} and the assumptions among the
arrival, service and FS-CTMC fading processes again that
\begin{eqnarray}
&&\hat{Q}^{r}_{j}(\cdot)=^{d}\frac{1}{r}E^{r}_{j}(\cdot)
-\frac{1}{r}S^{r}_{j}(\bar{T}^{r}_{j}(\cdot)) \elabel{centerhatQ}
\end{eqnarray}
where
\begin{eqnarray}
&&\bar{T}^{r}_{j}(\cdot)\equiv\int_{0}^{\cdot}
\Lambda_{j}\left(\bar{Q}^{r}(s),\alpha(s)\right)ds
=^{d}\frac{1}{r^{2}}T_{j}^{r}(r^{2}\cdot), \elabel{expresslambda}\\
&&\bar{Q}_{j}^{r}(t)\equiv\frac{1}{r^{2}}Q^{r}_{j}(r^{2}t)
\elabel{barqueue}
\end{eqnarray}
and we have used the radial homogeneity of $\Lambda(q,i)$ in
\eq{homcon} for \eq{expresslambda}. Now, let
\begin{eqnarray}
&&\hat{E}^{r}(\cdot)=(\hat{E}^{r}_{1}(\cdot),...,\hat{E}^{r}_{J}(\cdot))'
\;\;\mbox{with}\;\;\hat{E}^{r}_{j}(\cdot)=\frac{1}{r}
\left(A^{r}_{j}(r^{2}\cdot,\alpha(\cdot))
-r^{2}\bar{\lambda}^{r}_{j}(\cdot,\alpha_{\cdot}(\cdot))\right),
\elabel{ecenterrs}\\
&&\hat{S}^{r}(\cdot)=(\hat{S}^{r}_{1}(\cdot),...,\hat{S}^{r}_{J}(\cdot))'
\;\;\;\;\mbox{with}\;\;
\hat{S}_{j}^{r}(\cdot)=\frac{1}{r}\left(S_{j}(r^{2}\cdot)
-\mu_{j}r^{2}\cdot\right) \elabel{scenterrs}
\end{eqnarray}
for each $j\in\{1,...,J\}$ with
\begin{eqnarray}
\bar{\lambda}^{r}_{j}(\cdot,\alpha_{\cdot}(\cdot))
&\equiv&\int_{0}^{\cdot}\lambda_{j}^{r}(\alpha(s))ds
\elabel{averagerate}\\
&=^{d}&\int_{0}^{\cdot}\lambda_{j}^{r}(\alpha^{r}(r^{2}s))ds
\nonumber\\
&=&\frac{1}{r^{2}}\int_{0}^{r^{2}\cdot}\lambda_{j}^{r}(\alpha^{r}(s))ds.
\nonumber
\end{eqnarray}
Moreover, define
\begin{eqnarray}
\bar{\lambda}^{r}(\cdot,\alpha_{\cdot}(\cdot))
&=&\left(\bar{\lambda}^{r}_{1}(\cdot,\alpha_{\cdot}(\cdot)),...,
\bar{\lambda}^{r}_{J}(\cdot,\alpha_{\cdot}(\cdot))\right)'.
\elabel{lambdaeo}
\end{eqnarray}
Then we have the following lemma.
\begin{lemma}\label{primconlm}
For the diffusion-scaled processes in \eq{ecenterrs}-\eq{scenterrs},
the following convergence in distribution is true as
$r\rightarrow\infty$, i.e.,
\begin{eqnarray}
&&\left(\hat{E}^{r}(\cdot),\hat{S}^{r}(\cdot)\right)\Rightarrow
\left(H^{E}(\cdot), (\Gamma^{B})^{1/2}B^{S}(\cdot)\right).
\elabel{esjointc}
\end{eqnarray}
where
$\Gamma^{B}=\mbox{diag}(\mu_{1}\beta_{1}^{2},...,\mu_{J}\beta_{J}^{2})$.
\end{lemma}
\begin{proof} It follows from the heavy traffic condition
\eq{heavytrafficc}, the functional central limit theorem (see, e.g.,
\cite{iglwhi:equfun} and \cite{pro:conran}), and the random change
of time lemma (see, e.g., page 151 of \cite{bil:conpro}), Lemma 8.4
in \cite{daidai:heatra} that, for each $n\in\{0,1,...\}$,
\begin{eqnarray}
&&\left(\hat{E}^{r}(\tau_{n}+t)-\hat{E}^{r}(\tau_{n})\right)I_{\{0\leq
t<\sigma_{n}\}}
\elabel{inttntna}\\
&=&\frac{1}{r}\left(A^{r}(r^{2}(\tau_{n}+t),
\alpha_{\tau_{n}+t}(\cdot))-A^{r}(r^{2}\tau_{n},
\alpha_{\tau_{n}}(\cdot))\right)I_{\{0\leq t<\sigma_{n}\}}
\nonumber\\
&&-r\left(\bar{\lambda}^{r}(\tau_{n}+t, \alpha_{\tau_{n}+t}(\cdot))
-\bar{\lambda}^{r}(\tau_{n},\alpha_{\tau_{n}}(\cdot))\right)
I_{\{0\leq t<\sigma_{n}\}}
\nonumber\\
&=&\frac{1}{r}\tilde{A}^{r}(r^{2}(t-\phi_{n}/r^{2})I_{\{0\leq
t<\sigma_{n}\}})+e_{n}-r\left(\bar{\lambda}^{r}(\tau_{n}+t,
\alpha_{\tau_{n}+t}(\cdot))
-\bar{\lambda}^{r}(\tau_{n},\alpha_{\tau_{n}}(\cdot))\right)
I_{\{0\leq t<\sigma_{n}\}}
\nonumber\\
&\Rightarrow&\left(\Gamma^{E}(\alpha(\tau_{n}))\right)^{\frac{1}{2}}
I_{\{0\leq t<\sigma_{n}\}}B^{E}(t)\;\;\;\;\;\;\;\;\;\;\;\;\;\;\;\;\;
\mbox{as}\;\;r\rightarrow\infty
\nonumber\\
&=^{d}&\left(H^{E}(\tau_{n}+t) -H^{E}(\tau_{n})\right)I_{\{0\leq
t<\sigma_{n}\}}, \nonumber
\end{eqnarray}
where $\sigma_{n}=\tau_{n+1}-\tau_{n}$ is an exponentially
distributed random variable independent of all other random events
concerned since $\alpha(\cdot)$ is a FS-CTMC,
$e_{n}=(e_{n}^{1},...,e_{n}^{J})'$ with $e_{n}^{j}=0$ if
$\hat{E}^{r}_{j}(\cdot)$ has a jump at $\tau_{n}$ and 1 otherwise
for each $j\in\{1,...,J\}$, $\tilde{A}^{r}(\cdot)$ is a renewal
process with rate vector
$\lambda^{r}(\alpha(\tau_{n}))=(\lambda^{r}_{1}(\alpha(\tau_{n})),
...,\lambda^{r}_{J}(\alpha(\tau_{n})))'$ and
\begin{eqnarray}
&&\tilde{A}^{r}(r^{2}(\cdot-\phi_{n}/r^{2}))
=(\tilde{A}_{1}^{r}(r^{2}(\cdot-\phi_{n}^{1}/r^{2})),...,
\tilde{A}^{r}_{J}(r^{2}(\cdot-\phi_{n}^{J}/r^{2})))' \nonumber
\end{eqnarray}
with $\phi_{n}=(\phi_{n}^{1},...,\phi_{n}^{J})'$ being a
$J$-dimensional random vector whose $j$th component $\phi_{n}^{j}$
for each $j\in\{1,...,J\}$ denotes the remaining arrival time
beginning at $\tau_{n}$ for a packet to the $j$th queue with rate
$\lambda_{j}(\alpha(\tau_{n}))$ switched from
$\lambda_{j}(\alpha(\tau_{n-1}))$ at $\tau_{n}$ for each
$n\in\{1,2,...\}$. Moreover, to be convenient for later purpose, we
reexpress \eq{inttntna} as follows, over each
$[\tau_{n},\tau_{n+1})$ and as $r\rightarrow\infty$,
\begin{eqnarray}
\tilde{E}^{r,n}(\cdot)
&\equiv&\hat{E}^{r}(\tau_{n}+\cdot)-\hat{E}^{r}(\tau_{n})
\elabel{inttntn}\\
&\Rightarrow&H^{E}(\tau_{n}+\cdot)-H^{E}(\tau_{n}) \nonumber\\
&\equiv&\tilde{H}^{E,n}(\cdot). \nonumber
\end{eqnarray}
Then, by following \eq{inttntn} and by generalizing the discussion
in the proof for Theorem 3.2 in \cite{dai:broapp} or Lemma 8.2 in
\cite{daidai:heatra}, we can reach a proof for the claim in
\eq{esjointc}.

To do so, we first establish the relative compactness for
$\hat{E}^{r}(\cdot)$ with $r\in\{1,2,...,\}$. As a matter of fact,
define the modulus of continuity in terms of a function
$x(\cdot):[0,\infty)\rightarrow R^{d}$ with some integer $d>0$ for
each given $T>0$ and $\delta>0$ as follows,
\begin{eqnarray}
&&w(x,\delta,T)\equiv\inf_{t_{l}}\max_{l}
\mbox{Osc}\left(x,[t_{l-1},t_{l})\right) \elabel{cmodulus}
\end{eqnarray}
where the infimum takes over the finite sets $\{t_{l}\}$ of points
satisfying $0=t_{0}<t_{1}<...<t_{m}=T$ and $t_{l}-t_{l-1}>\delta$
for $l=1,...,m$, and
\begin{eqnarray}
&&\mbox{Osc}(x,[t_{l-1},t_{l}])=\sup_{t_{1}\leq s\leq t\leq
t_{2}}\|x(t)-x(s)\|_{2}\elabel{osclationd}
\end{eqnarray}
with $\|\cdot\|_{2}$ denoting the Euclidean norm in $R^{d}$. Then it
follows from Corollary 7.4 in page 129 of \cite{ethkur:marpro} that
the justification of the relative compactness is equivalent to
proving the following two conditions:

(a) For each $\eta>0$ and rational $t\geq 0$, there exists a
constant $c(\eta,t)$ such that
\begin{eqnarray}
&&\liminf_{r\rightarrow\infty}P\left\{\left\|\hat{E}^{r}(t)
\right\|_{2}\leq c(\eta,t)\right\}\geq 1-\eta. \nonumber
\end{eqnarray}

(b) For each $\eta>0$ and $T>0$, there exists a $\delta>0$ such that
\begin{eqnarray}
&&\limsup_{r\rightarrow\infty}P\left\{w(\hat{E}^{r},\delta,T)
\geq\eta\right\}\leq\eta. \nonumber
\end{eqnarray}

To show (a), we first define $N(t)\equiv\max\{n,\tau_{n}\leq t\}$
for each $t\in(0,\infty)$. Then, for each rational $t>0$, take a
$T>0$ such that $t\in(0,T]$ and define a sequence of events: ${\cal
S}_{l}\equiv\{\omega:N(T,\omega)\leq l\}$ for each
$l\in\{1,2,...,\}$. Since $\alpha(\cdot)$ has at most finitely many
jumps a.s. over $[0,T]$, we know that the sequence of probabilities
$P\{S_{l}\}$ increases monotonously to the unity as
$l\rightarrow\infty$. Thus, for the given $\eta>0$, there is some
large enough $L>0$ such that
\begin{eqnarray}
&&P\{{\cal S}_{L}\}\geq 1-\frac{\eta}{2}.\elabel{ptail}
\end{eqnarray}

Moreover, it follows from \eq{inttntn} and Remark 7.3 in page 129 of
\cite{ethkur:marpro} that $\tilde{E}^{r}(\cdot)$ satisfies the
following compact containment condition, i.e., for each $\eta>0$ and
$T>0$, there is a constant $K_{n}>0$ for each $n\in\{0,1,...,\}$
such that
\begin{eqnarray}
\inf_{r}P\left\{{\cal T}^{r,n}\right\}\geq
1-\frac{\eta}{2L}\;\;\;\mbox{with}\;\;{\cal
T}^{r,n}\equiv\{\omega:\|\tilde{E}^{r}(t)\|_{2} \leq
K_{n},t\in[0,T]\cap[0,\sigma_{n})\}. \elabel{intaun}
\end{eqnarray}
In addition, for each $n\in\{1,2,...,\}$, let
$\Delta_{n}=(\delta_{n}^{1},...,\delta_{n}^{J})'$ with
$\delta_{n}^{j}=1$ if $\hat{E}^{r}_{j}(\cdot)$ has a jump at
$\tau_{n}$ and zero otherwise for each $j\in\{1,...,J\}$. Then, for
each $t\in[\tau_{N(t)},\tau_{N(t)+1})$, we have
\begin{eqnarray}
&&\hat{E}^{r}(t)=\hat{E}^{r}(\tau_{N(t)})
+\tilde{E}^{r,N(t)}(t-\tau_{N(t)}),
\elabel{sumjI}\\
&&\hat{E}^{r}(\tau_{n}) -\hat{E}^{r}(\tau_{n}^{-})
=\frac{1}{r}\Delta_{n}. \elabel{sumjII}
\end{eqnarray}
Therefore it follows from \eq{sumjI}-\eq{sumjII} that, along each
sample path and for any $t_{1},t_{2}\in[0,T]$,
\begin{eqnarray}
\mbox{Osc}\left(\hat{E}^{r},[t_{1},t_{2}]\right)
\leq\sum_{n=0}^{N(t_{2})}
\mbox{Osc}\left(\tilde{E}^{r,n},[t_{1}-\tau_{n},t_{2}-\tau_{n}]
\cap[0,\sigma_{n})\right)+\frac{1}{r}(N(t_{2})-N(t_{1})).
\elabel{oscineI}
\end{eqnarray}
Thus it follows from \eq{oscineI} that, along each sample path in
${\cal S}_{L}\cap{\cal T}^{r,n}$ with $r,n\in\{1,2,...\}$,
\begin{eqnarray}
\left\|\hat{E}^{r}(t)\right\|_{2}
&\leq&\left\|\hat{E}^{r}(0)\right\|_{2}
+\mbox{Osc}\left(\hat{E}^{r},[0,t]\right)
\elabel{oscineII}\\
&\leq&2\sum_{n=0}^{L}\sup_{t\in[0,T]\cap[0,\sigma_{n}}
\left\|\tilde{E}^{r,n}(t)\right\|_{2} +\frac{L}{r}. \nonumber
\end{eqnarray}
Hence, for the above arbitrarily given $\eta>0$, each rational
$t\in[0,T]$, and large enough $r\in\{1,2,...\}$, we know that
\begin{eqnarray}
&&P\left\{\left\|\hat{E}^{r}(t)\right\|_{2}\leq
2^{L+1}\sum_{n=0}^{L}K_{n}\right\}\elabel{condia}\\
&\geq&P\left\{\left\{\left\|\hat{E}^{r}(t)\right\|_{2}\leq
2^{L+1}\sum_{n=0}^{L}K_{n}\right\}\bigcap{\cal
S}_{L}\right\}\nonumber\\
&\geq&P\left\{{\cal S}_{L}\right\}-\sum_{n=0}^{L}
P\left\{\left\{\left\|\tilde{E}^{r,n}(t)\right\|_{2}
>\left(2K_{n}-\frac{1}{r2^{L}}\right)\right\}
\bigcap{\cal S}_{L}\;\mbox{for some}\;
t\in[0,T]\cap[0,\sigma_{n})\right\}\nonumber\\
&\geq&P\left\{{\cal S}_{L}\right\}-\sum_{n=0}^{L}
P\left\{\left\{\left\|\tilde{E}^{r,n}(t)\right\|_{2}>K_{n}\right\}
\bigcap{\cal S}_{L}\;\mbox{for some}\;
t\in[0,T]\cap[0,\sigma_{n})\right\}\nonumber\\
&>&1-\eta,\nonumber
\end{eqnarray}
where the second inequality follows from \eq{oscineII} and the fact
that
\begin{eqnarray}
&&P\{\|aX+bY\|_{2}\geq K_{1}+K_{2}\}\leq P\left\{\|X\|_{2}\geq
\frac{K_{1}}{2|a|}\right\}+P\left\{\|Y\|_{2}\geq\frac{K_{2}}{2|b|}\right\}
\nonumber
\end{eqnarray}
for any real number $a,b$ and random vectors $X,Y$. Moreover, the
last inequality in \eq{condia} follows from \eq{ptail} and
\eq{intaun}. Thus condition (a) holds.

Next we prove the condition (b) to be true. Due to \eq{inttntn}, we
know that, for each $\eta>0$ and $T>0$, there exists a
$\delta_{n}>0$ for each $n\in\{0,1,...,\}$ such that
\begin{eqnarray}
\limsup_{r\rightarrow\infty}P\left\{w(\tilde{E}^{r,n},
\delta_{n},[0,T]\cap[0,\sigma_{n}))
\geq\frac{\eta}{2^{L}L}\right\}\leq\frac{\eta}{2^{L}L}.
\elabel{maxineI}
\end{eqnarray}
Now take $\delta=\min\{\delta_{0},...,\delta_{L}\}>0$, then for each
$r\in\{1,2,...,\}$ and each sample path in ${\cal S}_{L}$,
\begin{eqnarray}
w(\hat{E}^{r},\delta,T)
&\leq&\sum_{n=0}^{J}w(\tilde{E}^{r,n},\delta,[0,T]
\cap[0,\sigma_{n}))+\frac{L}{r}
\elabel{maxineII}\\
&\leq&\sum_{n=0}^{J}w(\tilde{E}^{r,n},\delta_{n},[0,T]
\cap[0,\sigma_{n}))+\frac{L}{r} \nonumber
\end{eqnarray}
where the first inequality follows from \eq{cmodulus} and
\eq{oscineI}, and the second inequality follows from 1.9 in page 326
of \cite{jacshi:limthe}. Therefore, for each large enough
$r\in\{1,2,...,\}$, it follows from \eq{maxineI}-\eq{maxineII} that
\begin{eqnarray}
&&P\left\{w(\hat{E}^{r},\delta,T)\geq\eta\right\}
\nonumber\\
&<&\frac{\eta}{2}+
P\left\{\left\{w(\hat{E}^{r},\delta,T)\geq\eta\right\}
\bigcap{\cal S}_{L}\right\}\nonumber\\
&<&\frac{\eta}{2}+\sum_{n=0}^{L}P\left\{\left\{w(\tilde{E}^{r,n},
\delta_{n},[0,T]\cap[0,\sigma_{n}))
\geq\frac{1}{2^{L-1}}\left(\frac{\eta}{L}-\frac{1}{r}\right)\right\}
\bigcap{\cal S}_{L}\right\}
\nonumber\\
&<&\frac{\eta}{2}+\sum_{n=0}^{L}P\left\{\left\{w(\tilde{E}^{r,n},
\delta_{n},[0,T]\cap[0,\sigma_{n}))
\geq\frac{\eta}{2^{L}L}\right\}\bigcap{\cal S}_{L}\right\}
\nonumber\\
&\leq&\eta. \nonumber
\end{eqnarray}
So the condition (b) is true and hence we know that
$\hat{E}^{r}(\cdot)$ is relatively compact for $r\in\{1,2,...,\}$.

Finally, consider any subsequence ${\cal R}_{1}\subseteq
\{1,2,...,\}$ such that, along $r\in{\cal R}_{1}$, we have
\begin{eqnarray}
\hat{E}^{r}(\cdot) \Rightarrow\hat{E}(\cdot)\;\;\mbox{(a process to
be identified)}. \elabel{scnm}
\end{eqnarray}
Then it follows from the Skorohod representation theorem (see, e.g.,
Theorem 3.1.8 in page 102 of \cite{ethkur:marpro}) and the random
change of time lemma (see, e.g., page 151 of \cite{bil:conpro})
that, for each $n\in\{0,1,...,\}$ and along $r\in{\cal R}_{1}$,
\begin{eqnarray}
&&\left(\hat{E}^{r}(\cdot)I_{\{\cdot\leq\tau_{n+1}\}},
\hat{E}^{r}(\cdot)I_{\{\cdot\leq\tau_{n}\}}\right)
\Rightarrow\left(\hat{E}(\cdot)I_{\{\cdot\leq\tau_{n+1}\}},
\hat{E}(\cdot)I_{\{\cdot\leq\tau_{n}\}}\right). \nonumber
\end{eqnarray}
Then, by the method of induction in terms of $n\in\{0,1,...,\}$,
\eq{inttntn}, and the continuous-mapping theorem (see, e.g., Theorem
3.4.1 in page 85 of \cite{whi:stopro}), we can conclude that, along
$r\in{\cal R}_{1}$, the limit in \eq{scnm} is $H^{E}(\cdot)$.
Moreover, since ${\cal R}_{1}$ is arbitrarily chosen, we know that
$\hat{E}^{r}\Rightarrow H^{E}(\cdot)$ along $r\in\{1,2,...,\}$.
Moreover, by the independence assumptions and the functional central
limit theorem, we know that the claim in Lemma~\ref{primconlm} is
true. $\Box$
\end{proof}

\subsection{Fluid Limiting Processes}\label{fluidlim}

For each $j\in{\cal J}$, $t\geq 0$ and $r>0$, we define the
fluid-scaled processes as follows,
\begin{eqnarray}
&&\bar{W}^{r}(t)\equiv\frac{1}{r^{2}}W^{r}(r^{2}t),\;\;
\bar{Y}^{r}(t)=\frac{1}{r^{2}}Y^{r}(r^{2}t),\;\;
\bar{E}^{r}_{j}(t)\equiv\frac{1}{r^{2}}E_{j}^{r}(t),\;\;
\bar{S}^{r}_{j}(t)\equiv\frac{1}{r^{2}}S_{j}^{r}(r^{2}t)\elabel{qebar}
\end{eqnarray}
and use $\bar{Q}^{r}(\cdot),\bar{E}^{r}(\cdot),\bar{S}^{r}(\cdot)$,
$\bar{T}^{r}(\cdot)$ to denote the corresponding vector processes.
Further, let
\begin{eqnarray}
&&\bar{Q}_{j}(t)=\bar{Q}_{j}(0)+\bar{\lambda}_{j}(t)-\mu_{j}
\bar{T}_{j}(t)\;\;\mbox{for each}\;\;j\in{\cal J},\elabel{limqtnon}\\
&&\bar{W}(t)=\sum_{j=1}^{J}\frac{\bar{Q}_{j}(t)}{\mu_{j}}
=\bar{W}(0)+\bar{Y}(t),\elabel{limwtnon}\\
&&\bar{Y}(t)=\sum_{j=1}^{J}\left(\int_{0}^{t}\rho_{j}(\alpha(s))ds
-\bar{T}_{j}(t)\right),\elabel{limwtnonI}\\
&&\bar{T}_{j}(t)=\int_{0}^{t}\bar{\Lambda}_{j}(\bar{Q}(s),\alpha(s))ds,
\elabel{limtnon}
\end{eqnarray}
where, for each $i\in{\cal K}$,
\begin{eqnarray}
&&\bar{\Lambda}(q,i)=\left\{\begin{array}{ll} \bar{\Lambda}^{{\cal
Q}(k_{1},...,k_{m})}(q,i)'&\mbox{if}\;\;q\in{\cal
Q}(k_{1},...,k_{m})\;\;
\mbox{and a given}\;\;m\in{\cal J},\\
\rho(i)\;\;\;\;&\mbox{if}\;\;q=0.
\end{array}\right. \elabel{barLAMnon}
\end{eqnarray}
and, for each $j\in{\cal J}$,
\begin{eqnarray}
&&\bar{\Lambda}_{j}^{{\cal
Q}(k_{1},...,k_{m})}(q,i)=\left\{\begin{array}{ll}
\Lambda_{j}(q,i)=\Lambda_{j}^{{\cal
Q}(k_{1},...,k_{m})}(q,i)&\mbox{if}\;\;j\neq k_{l},l\in\{1,...,m\},\\
\rho_{j}(i)\;\;\;\;&\mbox{if}\;\;j=k_{l},l\in\{1,...,m\}.
\end{array}\right. \elabel{barLAMnonI}
\end{eqnarray}
Thus we have the following lemma.
\begin{lemma}\label{fluidlemma}
Suppose $\bar{Q}^{r}(0)\Rightarrow\bar{Q}(0)$ as
$r\rightarrow\infty$, then under the utility-maximization allocation
policy $\Lambda(q,i)$ in \eq{zerozero}, any subsequence of
$\{r,r=1,2,...,\}$ has a further subsequence $\{r_{l},l=1,2,...\}$
such that the following convergence in distribution is true,
\begin{eqnarray}
&&\left(\bar{E}^{r_{l}}(\cdot),\bar{S}^{r_{l}}(\cdot),
\bar{T}^{r_{l}}(\cdot),\bar{Q}^{r_{l}}(\cdot),\bar{W}^{r_{l}}(\cdot),
\bar{Y}^{r_{l}}(\cdot)\right)
\Rightarrow\left(\bar{E}(\cdot),\bar{S}(\cdot),
\bar{T}(\cdot),\bar{Q}(\cdot),\bar{W}(\cdot),\bar{Y}(\cdot)\right)
\elabel{fluidcon}
\end{eqnarray}
as $l\rightarrow\infty$, where the limit in \eq{fluidcon} satisfies
\eq{limqtnon}-\eq{barLAMnon}. Moreover, if $\bar{Q}(0)=0$, then the
convergence in \eq{fluidcon} is true along the whole sequence
$r=1,2,...$ with the limit satisfying
\begin{eqnarray}
&&\bar{E}(t)=\bar{\lambda}(t,\alpha_{t}(\cdot)),\;\;\bar{S}(t)=\mu(t),\;\;
\bar{T}(t)=\bar{c}(t,\alpha_{t}(\cdot)),
\elabel{fluidlimit}\\
&&\bar{Q}_{j}(t)=\bar{W}(t)=\bar{Y}(t)=0,\elabel{fluidlimitI}
\end{eqnarray}
for each $t\geq 0$ and $j\in{\cal J}$, where
$\mu(t)\equiv(\mu_{1},...,\mu_{J})'t$, and
\begin{eqnarray}
\bar{\lambda}(t,\alpha_{t}(\cdot))
&=&\left(\bar{\lambda}_{1}(t,\alpha_{t}(\cdot)),...,
\bar{\lambda}_{J}(t,\alpha_{t}(\cdot))\right)',\;\;
\bar{\lambda}_{j}(t,\alpha_{t}(\cdot))
\equiv\int_{0}^{t}\lambda_{j}(\alpha(s))ds,
\elabel{lambdae}\\
\bar{c}(t,\alpha_{t}(\cdot))
&=&\left(\bar{c}_{1}(t,\alpha_{t}(\cdot)),...,
\bar{c}_{J}(t,\alpha_{t}(\cdot))\right)',\;\;\;\;\;
\bar{c}_{j}(t,\alpha_{t}(\cdot))\equiv\int_{0}^{t}\rho_{j}(\alpha(s))ds.
\elabel{barvc}
\end{eqnarray}
\end{lemma}
\begin{proof} It follows from \eq{expresslambda} and \eq{tjqalpha}
that $\bar{T}^{r}(\cdot)$ is a.s. uniformly Lipschitz continuous
with Lipschitz constant $\max_{i\in{\cal
K}}(\sum_{j=1}^{J}\rho_{j}(i))$ for each $r>0$, which implies that
it is absolutely continuous and differentiable at almost every
$t\in(0,\infty)$ (in other words, almost every $t\in(0,\infty)$ is a
{\em regular point} of $\bar{T}^{r}(\cdot)$). Thus the sequence of
stochastic processes $\{\bar{T}^{r}(\cdot),r=1,2,...\}$ is
$C$-tight, that is, it is tight and each weak limit point is in
$C[0,\infty)^{J}$ a.s., where $C[0,\infty)^{J}$ is the space of all
$J$-dimensional continuous functions over $[0,\infty)$ and is
endowed with the Skorohod $J_{1}$-topology (see, e.g, Page 116 of
\cite{ethkur:marpro}). Moreover, it follows from
Lemma~\ref{primconlm} that
$\left(\bar{E}^{r}(\cdot),\bar{S}^{r}(\cdot)\right)$ is also
$C$-tight. In addition, by \eq{centerhatQ}-\eq{expresslambda} and
\eq{qebar}, we know that
\begin{eqnarray}
&&\bar{Q}_{j}^{r}(t)=\bar{E}^{r}_{j}(t)
-\bar{S}_{j}^{r}(\bar{T}_{j}^{r}(t)). \elabel{rqueuelength}
\end{eqnarray}
Hence it follows from \eq{rqueuelength}, \eq{wyte} and the random
time change lemma in page 151 of \cite{bil:conpro} that the
following sequence is $C$-tight as well,
\begin{eqnarray}
&&\left(\bar{E}^{r}(\cdot),\bar{S}^{r}(\cdot),\bar{T}^{r}(\cdot),
\bar{Q}^{r}(\cdot),\bar{W}^{r}(\cdot),\bar{Y}^{r}(\cdot)\right),
\elabel{barestq}
\end{eqnarray}
where we have used the independent assumption related to
$\{\alpha^{r}(\cdot),r=1,2,...\}$, the second condition in
\eq{heavytrafficc} and the fact that
\begin{eqnarray}
&&\frac{1}{r^{2}}\int_{0}^{r^{2}\cdot}\rho_{j}(\alpha^{r}(s))ds
=^{d}\int_{0}^{\cdot}\rho_{j}(\alpha(s))ds. \elabel{rhoequiv}
\end{eqnarray}
Therefore, any subsequence of the processes in \eq{barestq} has a
further subsequence convergent in distribution. Now suppose
$(\bar{E}(\cdot),\bar{S}(\cdot),\bar{T}(\cdot),\bar{Q}(\cdot),
\bar{W}(\cdot),\bar{Y}(\cdot))$ is a weak limit point corresponding
to the further subsequence indexed by $\{r_{l},l=1,2,...\}$. Then,
by the Skorohod representation theorem (see, e.g., Theorem 3.1.8 in
page 102 of \cite{ethkur:marpro}), there is a common supporting
probability space such that
\begin{eqnarray}
&&\left(\bar{E}^{r_{l}}(\cdot),\bar{S}^{r_{l}}(\cdot),
\bar{T}^{r_{l}}(\cdot),\bar{Q}^{r_{l}}(\cdot),\bar{W}^{r_{l}}(\cdot),
\bar{Y}^{r_{l}}(\cdot)\right)
\rightarrow(\bar{\lambda}(\cdot,\alpha_{\cdot}(\cdot)),
\mu(\cdot),\bar{T}(\cdot),\bar{Q}(\cdot),\bar{W}(\cdot),\bar{Y}(\cdot))
\elabel{subestq}
\end{eqnarray}
u.o.c. a.s. as $l\rightarrow\infty$ and the limiting processes in
\eq{subestq} satisfy \eq{limqtnon}-\eq{barLAMnon}. Here we only need
to justify \eq{barLAMnon} to be true and other equations hold
obviously.

As a matter of fact, due to \eq{subestq}, we know that the limit
processes in \eq{subestq} are uniformly Lipschitz continuous a.s. So
our discussion will base on a fix sample path and each regular point
$t>0$ over an interval $(\tau_{n-1},\tau_{n})$ with
$n\in\{1,2,...\}$ for $\bar{T}_{j}$ with $j\in{\cal J}$. It follows
from \eq{limqtnon} that $\bar{Q}$ is differential at $t$ and
satisfies
\begin{eqnarray}
&&\frac{d\bar{Q}_{j}(t)}{dt}=\lambda_{j}(\alpha(t))-\mu_{j}
\frac{d\bar{T}_{j}(t)}{dt}\elabel{barqbart}
\end{eqnarray}
for each $j\in{\cal J}$. If $\bar{Q}_{j}(t)=0$ for some $j\in{\cal
J}$, then it follows from $\bar{Q}_{j}(\cdot)\geq 0$ that
\begin{eqnarray}
&&\frac{d\bar{Q}_{j}(t)}{dt}=0\;\;\mbox{which implies
that}\;\;\frac{d\bar{T}_{j}(t)}{dt}
=\frac{\lambda_{j}(\alpha(t))}{\mu_{j}} =\rho_{j}(\alpha(t)).
\elabel{firstfc}
\end{eqnarray}
If $\bar{Q}_{j}(t)>0$ for the $j\in{\cal J}$, then there exists a
finite interval $(a,b)\in[0,\infty)$ containing $t$ in it such that
$\bar{Q}_{j}(s)>0$ for all $s\in(a,b)$ and hence we can take small
enough $\delta>0$ such that $\bar{Q}_{j}(t+s)>0$ with
$s\in(0,\delta)$. Thus it follows from \eq{expresslambda} and
\eq{ralloccon} that
\begin{eqnarray}
&&\left|\frac{1}{\delta}\left(\bar{T}^{r_{k}}_{j}(t+\delta)
-\bar{T}^{r_{k}}_{j}(t)\right)-\Lambda_{j}(\bar{Q}(t),\alpha(t))\right|
\elabel{bartevl}\\
&\leq&\frac{1}{\delta}\int_{0}^{\delta}
\left|\Lambda_{j}(\bar{Q}^{r_{k}}(t+s),\alpha(t+s))
-\Lambda_{j}(\bar{Q}(t+s),\alpha(t+s))\right|ds
\nonumber\\
&&+\frac{1}{\delta}\int_{0}^{\delta}
\left|\Lambda_{j}(\bar{Q}(t+s),\alpha(t+s))
-\Lambda_{j}(\bar{Q}(t),\alpha(t))\right|ds \nonumber\\
&\rightarrow&\frac{1}{\delta}\int_{0}^{\delta}
\left|\Lambda_{j}(\bar{Q}(t+s),\alpha(t+s))
-\Lambda_{j}(\bar{Q}(t),\alpha(t))\right|ds\;\;\;\;\mbox{as}\;\;
k\rightarrow\infty\nonumber
\end{eqnarray}
where we have used the Lebesgue dominated convergence theorem for
the last claim in \eq{bartevl}. Due to the right-continuity of
$\alpha(\cdot)$, the Lipschitz continuity of $\bar{Q}(\cdot)$ and
\eq{ralloccon}, the last expression in \eq{bartevl} tends to zero as
$\delta\rightarrow 0^{+}$. Hence we have
\begin{eqnarray}
&&\frac{d\bar{T}_{j}(t)}{dt} =\frac{d\bar{T}_{j}(t^{+})}{dt}
=\bar{\Lambda}_{j}(\bar{Q}(t),\alpha(t))\;\;\mbox{for
each}\;\;j\in{\cal J} \elabel{fracbart}
\end{eqnarray}
which implies that the claims in \eq{limtnon}-\eq{barLAMnon} are
true.

Next, we introduce the following cost objective with $c(i)=\rho(i)$
in \eq{costf} for each $i\in{\cal K}$,
\begin{eqnarray}
&&\psi(q,i)\equiv\sum_{j=1}^{J}C_{j}(q_{j},\rho_{j}(i)).
\elabel{psicj}
\end{eqnarray}
Then, for each regular time $t\geq 0$ of $\bar{Q}(t)$ over time
interval $(\tau_{n-1},\tau_{n})$ with a given $n\in\{1,2,...\}$, we
have
\begin{eqnarray}
\;\;\;\;\;\;\;\;\frac{d\psi(\bar{Q}(t),\alpha(t))}{dt}
&=&\sum_{j=1}^{J}\left(\frac{d\bar{Q}_{j}(t)}{dt}\frac{\partial
C_{j}(\bar{Q}_{j}(t),\rho_{j}(\alpha(t))}{\partial \bar{Q}_{j}(t)}
+\frac{d\rho_{j}(\alpha(t))}{dt}\frac{\partial
C_{j}(\bar{Q}_{j}(t),\rho_{j}(\alpha(t))}{\partial
\rho_{j}(\alpha(t))}\right)
\elabel{lyapunov}\\
&=&\sum_{j=1}^{J}\left(\rho_{j}(\alpha(t))
-\Lambda_{j}(\bar{Q}(t),\alpha(t))\right)\frac{\partial
U_{j}(\bar{Q}_{j}(t),\rho_{j}(\alpha(t)))}{\partial
\rho_{j}(\alpha(t))}I_{\{\bar{Q}_{j}(t)>0\}}
\nonumber\\
&\leq&0\nonumber
\end{eqnarray}
where we have used \eq{barqbart}, \eq{fracbart}, \eq{costf} and the
fact that the sample paths of $\alpha(\cdot)$ are piecewise
constants for the second equality of \eq{lyapunov}, and we have used
the concavity of the utility functions, the fact that
$\Lambda_{j}(\bar{Q}(t),\alpha(t))$ is the optimal solution to
\eq{srates}, and the similar arguments as used in
\cite{yeyao:heatra} for the last inequality of \eq{lyapunov}.
Therefore, for any given $n\in\{0,1,2,...\}$ and each
$t\in[\tau_{n},\tau_{n+1})$, we have,
\begin{eqnarray}
0&\leq&\psi(\bar{Q}(t),\alpha(t)) \elabel{psiine}\\
&\leq&\psi(\bar{Q}(\tau_{n}), \alpha(\tau_{n}))
\nonumber\\
&=&\prod_{l=1}^{n}\left(\frac{d\Psi(\rho_{1}(\alpha(\tau_{l})))}
{dc_{1}}\right)\left(\frac{d\Psi(\rho_{1}(\alpha(\tau_{l-1})))}{dc_{1}}
\right)^{-1}\psi(\bar{Q}(0),\alpha(0)) \nonumber\\
&\leq&\kappa\psi(\bar{Q}(0),\alpha(0)), \nonumber
\end{eqnarray}
where the second inequality in \eq{psiine} follows from
\eq{lyapunov} and the last equality in \eq{psiine} follows from
\eq{costf}, \eq{uconII}-\eq{uconIII}, \eq{rhoji}, the continuity of
$\bar{Q}(t)$ at all jump times $\tau_{n}$ with $n\in\{1,2,...,\}$.
Moreover, the $\kappa$ in the last inequality of \eq{psiine} is a
positive constant given by
\begin{eqnarray}
&&\kappa=\sum_{m=1}^{M}\prod_{l=1}^{m}
\left(\frac{d\Psi(\rho_{1}(i_{l}))}
{dc_{1}}\right)\left(\frac{d\Psi(\rho_{1}(i_{l-1}))}{dc_{1}}
\right)^{-1}\elabel{kappaexp}
\end{eqnarray}
where the indices in each of the $M$ products of \eq{kappaexp}
satisfy that $i_{k}\neq i_{l}$ if $k\neq l$ for
$k,l\in\{0,1,...,m\}$ and the integer $M$ can be explicitly
calculated since the state space of $\alpha(\cdot)$ is finite.

If $\bar{Q}(0)=0$, it follows from \eq{psiine} that $\bar{Q}(t)=0$
for all $t\geq 0$. Thus, by \eq{limwtnon}, we know that
$\bar{W}(t)=\bar{Y}(t)=0$ for all $t\geq 0$. Moreover, it follows
from\eq{firstfc} that the third claim in \eq{fluidlimit} is true.
Hence, under the assumption that $\bar{Q}(0)=0$, all the claims
stated in the lemma are true. $\Box$
\end{proof}

Next, since $y_{j}=C_{j}(q_{j},\rho_{j}(i))$ is strictly increasing
in $q_{j}$ for each $i\in{\cal K}$ and $j\in{\cal J}$, its inverse
$C_{j}^{-1}(y_{j},\rho_{j}(i))$ is well defined and is strictly
increasing in $y_{j}$. So, for each $\kappa\geq 0$, we can define
\begin{eqnarray}
&&\tilde{g}(\kappa)\equiv\sum_{i=1}^{K}\max_{\|q\|\leq\kappa}\psi(q,i)
\;\;\; \mbox{and}\;\;\;g(\kappa)\equiv\sum_{i=1}^{K}
\sum_{j=1}^{J}C_{j}^{-1}(\tilde{g}(\kappa),\rho_{j}(i)).
\elabel{tildeg}
\end{eqnarray}
Then we have the following lemma.
\begin{lemma}\label{uniattrack}
Under the same conditions as used in Lemma~\ref{fluidlemma}, if
$\|\bar{Q}(0)\|\leq\chi$ for some constant $\chi$, then $\bar{Q}(t)$
is bounded for each $t\geq 0$, i.e.,
\begin{eqnarray}
&&\|\bar{Q}(t)\|\leq g(\chi)\;\;\mbox{for each}\;\;t\geq
0.\elabel{barqleqg}
\end{eqnarray}
Moreover, there exists a time $T_{\chi,\epsilon}>0$ for any given
$\epsilon>0$ such that
\begin{eqnarray}
&&\left\|\bar{Q}(t)-q^{*}(\bar{W}(t),\rho(\alpha(t)))\right\|
<\epsilon\;\;\mbox{for all}\;\;t\geq
T_{\chi,\epsilon}\elabel{barepsilon}
\end{eqnarray}
and in particularly, if
$\bar{Q}(0)=q^{*}(\bar{W}(0),\rho(\alpha(0)))$, then
$\bar{Q}(t)=\bar{Q}(0)$ a.s. for all $t\in[\tau_{0},\tau_{1})$.
\end{lemma}
\begin{proof} If $\|\bar{Q}(0)\|\leq\chi$, it follows from \eq{psiine}
that $\bar{Q}(t)$ is bounded for all $t\geq 0$ since
$C_{j}(q_{j},\rho_{j}(i))$ for each $j\in{\cal J}$ and $i\in{\cal
K}$ is strictly increasing and unbounded function in $q_{j}$.
Moreover, it follows from \eq{tildeg} that $g(\kappa)$ is increasing
in $\kappa$ with $g(0)=0$ and hence it follows from \eq{psiine} and
\eq{psicj} that
\begin{eqnarray}
&&C_{j}(\bar{Q}_{j}(t),\rho_{j}(\alpha(t)))\leq\tilde{g}(\chi)\;\;\;
\mbox{and}\;\;\;\bar{Q}_{j}(t)\leq
C_{j}^{-1}(\tilde{g}(\chi),\rho_{j}(\alpha(t)))\elabel{cjcjinv}
\end{eqnarray}
for each $t\geq 0$ and $j\in{\cal J}$, which implies that
\eq{barqleqg} is true and $\bar{W}(t)$ increases to some finite
number as $t$ increases due to \eq{limwtnon}-\eq{limwtnonI} and
\eq{barqleqg}, i.e.,
\begin{eqnarray}
&&\bar{W}(t)\uparrow\bar{W}(\infty)<\infty\;\;\;\mbox{as}\;\;\;
t\rightarrow\infty.\elabel{barwupa}
\end{eqnarray}
Thus we can define the following Lyapunov function with at most
countably many jumps,
\begin{eqnarray}
&&L(\bar{Q}(t),\alpha(t))\equiv\psi(\bar{Q}(t),\alpha(t))
-\psi(q^{*}(\bar{W}(t),\rho(\alpha(t))),\alpha(t)),
\elabel{lyapunovl}
\end{eqnarray}
which is nonnegative and bounded over $t\in[0,\infty)$ due to
Lemma~\ref{costcc}, \eq{limwtnon}, Lemma~\ref{fluidlemma},
\eq{barwupa} and the fact that $\bar{Q}(t)$ and $\rho(\alpha(t))$
are bounded over $[0,\infty)$. Then, for any given regular time
$t>0$ over an interval $(\tau_{n-1},\tau_{n})$ with
$n\in\{1,2,...\}$ and for any $\delta>0$, we can show that there
exists a $\sigma>0$ such that
\begin{eqnarray}
&&\frac{dL(\bar{Q}(t),\alpha(t))}{dt}\leq-\sigma\;\;\;\mbox{if}
\;\;\;\left\|\bar{Q}(t)-q^{*}(\bar{W}(t),\rho(\alpha(t)))\right\|
\geq\delta. \elabel{bsecond}
\end{eqnarray}

As a matter of fact, it follows from \eq{lyapunov} and \eq{barwupa}
that $\psi(\bar{Q}(t),\alpha(t))$ is non-increasing and
$\psi(q^{*}(\bar{W}(t),$ $\rho(\alpha(t))),\alpha(t))$ is
non-decreasing in $t\in(\tau_{n-1},\tau_{n})$ since $\alpha(t)$
keeps flat over the time interval $(\tau_{n-1},\tau_{n})$. So we
only need to show \eq{bsecond} true with respect to
$\psi(\bar{Q}(t),\alpha(t))$. By \eq{lyapunov}, we define
\begin{eqnarray}
\;\;\;\;\;\;h(\bar{Q}(t),\alpha(t))
&\equiv&\frac{d\psi(\bar{Q}(t),\alpha(t))}{dt}
\elabel{hdef}\\
&=&\sum_{j=1}^{J}\left(\rho_{j}(\alpha(t))
-\Lambda_{j}(\bar{Q}(t),\alpha(t))\right)\frac{\partial
U_{j}(\bar{Q}_{j}(t),\rho_{j}(\alpha(t)))}{\partial
\rho_{j}(\alpha(t))}I_{\{\bar{Q}_{j}(t)>0\}} \nonumber
\end{eqnarray}
which is continuous in terms of $\bar{Q}(t)=q\neq 0$ with $q\in
R^{J}_{+}$ due to \eq{ralloccon}, \eq{uconIV} and the second-order
differentiability of $U_{j}(q_{j},c_{j})$. Next, let
\begin{eqnarray}
&&{\cal C}(i)\equiv\{q\in
R^{J}_{+}:\left\|q-q^{*}(w(q),\rho(i))\right\|\geq\delta\}\subset\{q\in
R_{+}^{J}:q\neq 0\} \elabel{cqqst}
\end{eqnarray}
where the workload $w(q)$ corresponding to each $q\in R^{J}_{+}$ is
defined as in \eq{limwtnon} and the set ${\cal C}(i)$ is a closed
subset of $R^{J}_{+}$ due to the first part of Lemma~\ref{costcc}.
Moreover, similar to \eq{lyapunov}, we know that $h(q,i)\leq 0$ and
the equality is true if and only if $q=q^{*}(w(q),\rho(i))$.

In fact, supposing the {\em if part} is true with some $q\in{\cal
Q}(k_{1},...,k_{m})$ that is defined in \eq{zerozero}, then it
follows from \eq{hdef} and the last equality in \eq{lyapunov} that
$\{\rho_{l}(i),l\neq k_{1},...,k_{m}\}$ is the solution to the
corresponding optimization problem in \eq{srates} with
$\{q_{l},l\neq k_{1},...,k_{m}\}$ in the associated
$(J-m)$-dimensional utility function. Thus, it follows from
Lemma~\ref{equioptso} that $\{q_{l},l\neq
k_{1},...,k_{m}\}=\{q^{*}_{l}(w(q),\rho(i)),l\neq
k_{1},...,k_{m}\}$. Moreover, since $q^{*}(w(q),\rho(i))>0$ due to
$w(q)>0$ and Lemma~\ref{equioptso}, we know that
$\psi(\bar{Q}(t),\alpha(t))
-\psi(q^{*}(\bar{W}(t),\rho(\alpha(t))),\alpha(t))<0$ for
$\bar{Q}(t)=q$, which contradicts the fact that
$q^{*}(w(q),\rho(i))$ is the solution to the cost minimization
problem in \eq{costminp}. Conversely, the {\em only if} part is the
direct conclusion of the second part in Lemma~\ref{equioptso}.
Therefore we have that $h(q,i)<0$ over ${\cal C}(i)$. Since $h(q,i)$
is continuous in $q\neq 0$, we know that there exists a $\sigma>0$
such that
\begin{eqnarray}
h(q,i)\leq-\sigma\;\;\mbox{in}\;\;{\cal C}(i)\elabel{endp}
\end{eqnarray}
Moreover, since the state space of $\alpha(\cdot)$ is finite, we can
consider $\sigma$ as the common constant such that \eq{endp} is true
for all $i\in{\cal K}$. So the claim in \eq{bsecond} is proved.

Next, we prove that there exists a time $T_{\chi,\epsilon}>0$ for
any given $\epsilon>0$ such that \eq{barepsilon} is true. To do so,
we first show that
\begin{eqnarray}
L(\bar{Q}(t),\alpha(t))\rightarrow 0\;\;\mbox{as}\;\;
t\rightarrow\infty.\elabel{lqet}
\end{eqnarray}
As a matter of fact, define
\begin{eqnarray}
&&L_{1}(\bar{Q}(t),\alpha(t))\equiv L(\bar{Q}(t),\alpha(t))-e(t)
\elabel{lyapunovlI}
\end{eqnarray}
where $e(t)$ is a step function given by
\begin{eqnarray}
e(t)&\equiv&\sum_{n:\tau_{n}\leq t}\left(\psi(\bar{Q}(\tau_{n}),
\alpha(\tau_{n})))-\psi(\bar{Q}(\tau_{n}^{-}),\alpha(\tau_{n}^{-}))
\right)\nonumber\\
&&-\sum_{n:\tau_{n}\leq t}\left(\psi(q^{*}(\bar{W}(\tau_{n}),
\rho(\alpha(\tau_{n})),\alpha(\tau_{n}))-\psi(q^{*}(\bar{W}(\tau_{n}^{-}),
\rho(\alpha(\tau_{n}^{-})),\alpha(\tau_{n}^{-}))\right).
\elabel{sumjumps}
\end{eqnarray}
Therefore we can see that $L_{1}(\bar{Q}(t),\alpha(t))$ is
continuous and bounded over $t\in[0,\infty)$ since $\bar{Q}(t)$ and
$\rho(\alpha(t))$ are bounded. Thus we know that $e(t)$ is also
bounded over $t\in[0,\infty)$ due to the fact that
$L(\bar{Q}(t),\alpha(t))$ is bounded. Moreover, since
\begin{eqnarray}
&&\frac{dL_{1}(\bar{Q}(t),\alpha(t))}{dt}
=\frac{dL(\bar{Q}(t),\alpha(t))}{dt}\leq 0\;\;\;\mbox{for
a.a.}\;\;\; t\in[0,\infty), \elabel{derivleso}
\end{eqnarray}
we know that $L_{1}(\bar{Q}(t),\alpha(t))$ converges to some
constant as $t\rightarrow\infty$.

Now, since $e(t)$ is a step function and is bounded, any convergent
subsequence of $e(t)$ in terms of $t$ corresponds to a sequence of
holding time intervals as $t\rightarrow\infty$ such that the
convergence of $e(t)$ is true for all $t$ along the sequence of
holding time intervals. Moreover, since the state space of
$\alpha(\cdot)$ is finite, there exists at least one $i\in{\cal K}$
such that the holding time intervals corresponding to this
particular state $i$ appear infinitely many times. To be convenient,
we use $[\tau_{n_{l}},\tau_{n_{l+1}})$ with $l\in\{1,2,...,\}$ to
denote such a sequence of holding time intervals, where
$\tau_{n_{l}}$ is the jump time of $\alpha(\cdot)$ corresponding to
the index $n_{l}$. Notice that $[\tau_{n_{l}},\tau_{n_{l+1}})$ with
$l\in\{1,2,...,\}$ are sampled from a sequence of $i.i.d$ random
variables (actually exponentially distributed). Therefore, due to
the strong law of large number and without loss of generality, we
can assume that
\begin{eqnarray}
&&\sum_{l=1}^{\infty}\left(\tau_{n_{l+1}}-\tau_{n_{l}}\right)=
\infty.\elabel{sumtaunl}
\end{eqnarray}
Therefore, for an arbitrarily given convergent subsequence of
$e(t)$, we can obtain a sequence of holding time intervals
$[\tau_{n_{l}},\tau_{n_{l+1}})$ ($l\in\{1,2,...,\}$) with the
property \eq{sumtaunl} associated with a particular state $i\in{\cal
K}$. Then it follows from the convergence of
$L_{1}(\bar{Q}(t),\alpha(t))$ that
\begin{eqnarray}
&&L(\bar{Q}(t),\alpha(t))\rightarrow L_{\infty}\geq
0\;\;\mbox{as}\;\; t\rightarrow\infty\;\;\mbox{over}\;\;
t\in\cup_{l=1}^{\infty}[\tau_{n_{l}},\tau_{n_{l+1}}). \elabel{lbarq}
\end{eqnarray}
Furthermore, we can claim that $L_{\infty}=0$ by showing a
contradiction. In fact, if we assume $L_{\infty}>0$, then for any
given constant $\epsilon$ satisfying $0<\epsilon<L_{\infty}$, there
exists some large enough time $T_{1}>0$ such that
\begin{eqnarray}
&&L(\bar{Q}(t),\alpha(t))>L_{\infty}-\epsilon>0\elabel{plgreat}\;\;\;
\mbox{for all}\;\;\;t\in[T_{1},\infty)\cap
\left(\cup_{l=1}^{\infty}[\tau_{n_{l}},\tau_{n_{l+1}})\right).
\elabel{lowertb}
\end{eqnarray}
Since $\psi(q,i)$ is continuous and strictly increasing in $q\in
R_{+}^{J}$ for each $i\in{\cal K}$, it follows from \eq{lowertb}
that there exist some $\delta>0$ and $\sigma>0$ such that
\eq{bsecond} is true for all $t\in[T_{1},\infty)\cap
\left(\cup_{l=1}^{\infty}[\tau_{n_{l}},\tau_{n_{l+1}})\right)$. Thus
it follows from \eq{lyapunovlI} and \eq{derivleso}-\eq{sumtaunl}
that
\begin{eqnarray}
L(\bar{Q}(t),\alpha(t))&=&L(\bar{Q}(0),\alpha(0))+e(t)+\int_{0}^{t}
\frac{dL_{1}(\bar{Q}(t),\alpha(t))}{dt}dt
\elabel{fcompare}\\
&\leq&C-\sigma\sum_{l=1}^{N(t)-1}(\tau_{n_{l+1}}-\tau_{n_{l}})
\nonumber\\
&<&0\nonumber
\end{eqnarray}
for all sufficient large $t\in[T_{1},\infty)\cap
\left(\cup_{l=1}^{\infty}[\tau_{n_{l}},\tau_{n_{l+1}})\right)$,
where $N(t)=\max\{l:\tau_{n_{l}}\leq t\}$ and $C$ is a positive
constant since $e(t)$ is bounded. However, the derived result in
\eq{fcompare} contradicts the fact that $L(\bar{Q}(t),\alpha(t))\geq
0$. Therefore the assumption that $L_{\infty}>0$ is not true, which
implies that $L_{\infty}=0$. Since the convergent subsequence of
$\alpha(\cdot)$ is arbitrarily chosen, we know that the convergence
in \eq{lqet} is true (readers are also referred to \cite{dai:poshar}
for related discussion concerning a continuous Lyapunov function
with no jumps.) Hence it follows from \eq{lqet}, the continuity and
strict monotonicity of $\psi(q,i)$ in $q\in R_{+}^{J}$ for each
$i\in{\cal K}$ that there exists a time $T_{\chi,\epsilon}>0$ for
any given $\epsilon>0$ such that \eq{barepsilon} is true.

Finally, if $\bar{Q}(0)=q^{*}(\bar{W}(0),\rho(\alpha(0)))$, then it
follows from \eq{lyapunovl} that the claim that
$\bar{Q}(t)=\bar{Q}(0)$ a.s. for all $t\in[\tau_{0},\tau_{1})$ is
true. Hence we finish the proof of the lemma. $\Box$
\end{proof}

\subsection{A Key Lemma on Finer Time-Scaling}\label{slowerfs}

It follows from \eq{rsqueue}, \eq{wyte}, \eq{queuelength},
\eq{enee}-\eq{averagerate}, and the similar argument as for
\eq{centerhatQ} that
\begin{eqnarray}
&&\hat{W}^{r}(\cdot)=^{d}\hat{X}^{r}(\cdot)+\hat{Y}^{r}(\cdot)
\elabel{difws}
\end{eqnarray}
where, for each $t\geq 0$,
\begin{eqnarray}
&&\hat{Y}^{r}(t)=r\sum_{j=1}^{J}
\left(\int_{0}^{t}\rho_{j}(\alpha(s))ds-\bar{T}^{r}_{j}(t)\right)
\elabel{hatys}
\end{eqnarray}
which is non-decreasing in $t\geq 0$ due to \eq{expresslambda},
\eq{tjqalpha} and \eq{bmaxcap}, and
\begin{eqnarray}
\hat{X}^{r}(t)&=&\sum_{j=1}^{J}\frac{1}{\mu_{j}}
\left(\hat{E}^{r}_{j}(t)-\hat{S}^{r}_{j}(\bar{T}^{r}_{j}(t))\right)
+\sum_{j=1}^{J}r\int_{0}^{t} \left(\rho^{r}_{j}(\alpha(s))
-\rho_{j}(\alpha(s))\right)ds \elabel{hatxs}\\
&\Rightarrow&\hat{X}(t)\;\;\mbox{as}\;\;r\rightarrow\infty\nonumber
\end{eqnarray}
where
$\rho^{r}_{j}(\alpha(\cdot))=\lambda_{j}^{r}(\alpha(\cdot))/\mu_{j}$,
and the weak convergence in \eq{hatxs} is due to
Lemma~\ref{primconlm}, Lemma~\ref{fluidlemma} and the random change
of time lemma (see, e.g., page 151 of \cite{bil:conpro}) with
$\hat{X}(\cdot)$ given by \eq{hatx}. Since $\hat{X}(\cdot)$ is a
continuous process, it follows from the Skorohod  representation
theorem that the convergence in \eq{hatxs} can be assumed u.o.c. So,
in the rest of this subsection, we will only consider an arbitrarily
given sample path for which the above u.o.c. convergence holds.

Now, for a time $\tau\geq 0$, a constant $\delta>0$, a large enough
integer $r$, and a fixed time $T>0$ of certain magnitude to be
specified later, we divide the time interval $[\tau,\tau+\delta]$
into a total of $\lceil r\delta/T\rceil-1$ segments with equal
length $T/r$ except the last one, where $\lceil\cdot\rceil$ denotes
the integer ceiling. The $l$th segment with $l\in\{0,1,...,\lceil
r\delta/T\rceil-1\}$ covers the time interval
$[\tau+lT/r,(\tau+(l+1)T/r)\wedge T]$. Then, for any
$t\in[\tau,\tau+\delta]$, it can be expressed as
\begin{eqnarray}
&&t=\tau+(lT+u)/r\equiv\eta^{r,l}(u) \elabel{tequivexpres}
\end{eqnarray}
for some $l\in\{0,1,...,\lceil r\delta/T\rceil-1\}$ and $u\in[0,T]$.
Hence, due to the explanations in \eq{arjrt}-\eq{averagerate}, we
can define
\begin{eqnarray}
\bar{W}^{r,l}(u)&\equiv&\frac{1}{r}W^{r}((r^{2}\tau+rlT)+ru,
\alpha_{\eta^{r,l}(u)}(\cdot))\elabel{barwhatw}\\
&=&\frac{1}{r}W^{r}(r^{2}t,\alpha_{t}(\cdot))\nonumber\\
&=&\hat{W}^{r}(t)\nonumber
\end{eqnarray}
for each $l\in\{0,1,...,\lceil r\delta/T\rceil-1\}$ and $u\in[0,T]$.
In other words, for each time point, we will study the behavior of
$\hat{W}^{r}(t)$ through the fluid process, $\bar{W}^{r,l}(u)$, over
the time interval $[0,T]$ (see, e.g., \cite{yeyao:heatra} and
references therein). Similarly, we can define $\bar{Q}^{r,l}(u)$ and
$\bar{Y}^{r,l}(u)$ through $\hat{Q}^{r}(t)$ and $\hat{Y}^{r}(t)$.
Moreover, define $c_{1}$ and $c_{2}$ to be the following constants
\begin{eqnarray}
&&c_{1}=\max_{j\in{\cal J}}(1/\mu_{j})\;\;\mbox{and}\;\;
c_{2}=\left(\min_{j\in{\cal
J}}(1/\mu_{j})\right)^{-1}.\elabel{cicii}
\end{eqnarray}
Thus, for any given $w\geq 0$ and all $i\in{\cal K}$, we have
\begin{eqnarray}
&&w\leq
c_{1}\|q^{*}(w,\rho(i))\|\;\;\mbox{and}\;\;\|q^{*}(w,\rho(i))\| \leq
c_{2}w\elabel{ciciiine}
\end{eqnarray}
since
\begin{eqnarray}
&&\min_{j\in{\cal J}}(1/\mu_{j})\sum_{j=1}^{J}q^{*}_{j}(w,\rho(i))
\leq
w=\sum_{j=1}^{J}\frac{q^{*}_{j}(w,\rho(i))}{\mu_{j}}\leq\max_{j\in{\cal
J}}(1/\mu_{j})\sum_{j=1}^{J}q^{*}_{j}(w,\rho(i)).\nonumber
\end{eqnarray}
In addition, for any $\epsilon>0$, define
\begin{eqnarray}
&&T_{1}=\max\left\{T_{(c_{2}+1)\epsilon,\epsilon},\;
T_{\max\{b_{1,1},b_{1,2}\},\epsilon/2},\;
T_{\max\{b_{1,1},b_{1,2}\},\sigma/2}\right\} \elabel{tttone}
\end{eqnarray}
where $\sigma$ is determined in Lemma~\ref{costcc} and
\begin{eqnarray}
&&b_{1,1}=g((c_{2}+1)\epsilon)+\epsilon,\;\;
b_{2,1}=c_{1}b_{1,1}+\epsilon,\elabel{dparamI}\\
&&b_{2,2}=\max\{b_{2,1},\nu+\epsilon\}+C+\epsilon,\;
b_{1,2}=c_{2}b_{2,2}+\epsilon\elabel{dparamII}
\end{eqnarray}
where $g(\cdot)$ is defined in \eq{tildeg}. Then we have the
following lemma.
\begin{lemma}\label{uniattract}
Consider the time interval $[\tau,\tau+\delta]$ with $\tau\geq 0$
and $\delta>0$ and suppose that there is some constant $\nu\geq 0$
such that
\begin{eqnarray}
&&\lim_{r\rightarrow\infty}\hat{W}^{r}(\tau)=\nu\;\;\mbox{and}\;\;
\lim_{r\rightarrow\infty}\hat{Q}^{r}(\tau)=q^{*}(\nu,\rho(\alpha(\tau))).
\elabel{barwnu}
\end{eqnarray}
Moreover, let $C$ be an arbitrarily chosen positive constant such
that
\begin{eqnarray}
&&\sup_{t_{1},t_{2}\in[\tau,\tau+\delta]}
\|\hat{X}(t_{1})-\hat{X}(t_{2})\|\leq C
\;\;\mbox{with}\;\;\hat{X}(\cdot)\;\;\mbox{given by}\;\;\eq{hatx}.
\elabel{unidifb}
\end{eqnarray}
Then, for any given small enough number $\epsilon>0$ and a given
$T\geq T_{1}$, the following claims are true for all large enough
$r\in\{1,2,...\}$ and all $l\in\{0,1,...,\lceil
r\delta/T\rceil-1\}$:
\begin{eqnarray}
&&\|\bar{Q}^{r,l}(u)-q^{*}(\bar{W}^{r,l}(u),
\rho(\alpha(\eta^{r,l}(u))))\|\leq\epsilon\;\;\mbox{for all}\;\;
u\in[0,T],\elabel{barquI}\\
&&\bar{W}^{r,l}(u)\leq\nu+C+O(\epsilon)\;\;\mbox{for
all}\;\;u\in[0,T],\elabel{barquII}\\
&&\bar{Y}^{r,l}(u)-\bar{Y}^{r,l}(0)=0\;\;\mbox{for
all}\;\;u\in[0,T]\;\;\mbox{if}\;\;
\bar{W}^{r,l}(u)>\epsilon\;\;\mbox{for
all}\;\;u\in[0,T],\elabel{barquIII}
\end{eqnarray}
where $\lim_{\epsilon\rightarrow 0}O(\epsilon)=0$.
\end{lemma}
\begin{proof} For convenience, besides \eq{barquI}, we will prove the
following stronger claims instead of showing \eq{barquII} and
\eq{barquIII} directly, that is, for large enough $r\in\{1,2,...\}$
and all nonnegative integers $l\in\{0,1,...,\lceil
r\delta/T\rceil-1\}$,
\begin{eqnarray}
&&\mbox{if}\;\;\bar{W}^{r,l}(u)\leq\epsilon<C\;\;\mbox{for some}\;\;
u\in[0,T],
\elabel{equiclaimI}\\
&&\mbox{then}\;\;\bar{W}^{r,l}(u)\leq
b_{2,1},\;\left\|\bar{Q}^{r,l}(u)\right\|\leq
b_{1,1}\;\;\mbox{for all}\;\;u\in[0,T];\nonumber\\
&&\mbox{if}\;\;\bar{W}^{r,l}(u)>\epsilon\;\;\mbox{for
all}\;\;u\in[0,T],
\elabel{equiclaimII}\\
&&\mbox{then}\;\;\bar{W}^{r,l}(u)\leq b_{2,2},\;\bar{Q}^{r,l}(u)\leq
b_{1,2},\;\bar{Y}^{r,l}(u)-\bar{Y}^{r,l}(0)=0\;\;\mbox{for
all}\;\;u\in[0,T].\nonumber
\end{eqnarray}
Thus the remaining proof of the lemma can be divided into the
following two parts.

{\bf Part One:} We justify the claims stated in the lemma to be true
when $l=0$. As a matter of fact, it follows from \eq{barwhatw} and
\eq{barwnu} that
\begin{eqnarray}
&&(\bar{W}^{r,0}(0),\bar{Q}^{r,0}(0))=(\hat{W}^{r}(\tau),
\hat{Q}^{r}(\tau))\rightarrow(\nu,q^{*}(\nu,\rho(\alpha(\tau)))
\;\;\;\mbox{as}\;\;r\rightarrow\infty. \elabel{barwbarnhat}
\end{eqnarray}
Now, due to the definition of $\tau_{n}$ defined in \eq{markovjt},
we know that $[\tau,\tau+ T/r]\subset[\tau_{n-1},\tau_{n})$ with
some $n\in\{1,2,...\}$ for all large enough $r\in\{1,2,...\}$. Thus
$\alpha(\eta^{r,0}(u))$ keeps some constant $\alpha(\tau)$ for all
$u\in[0,T]$ when $r$ is large enough. So, as $r\rightarrow\infty$,
we have,
\begin{eqnarray}
&&\;\;\;(\bar{W}^{r,0}(u),\bar{Q}^{r,0}(u))\rightarrow(\bar{W}(u),
\bar{Q}(u))=(\nu,q^{*}(\nu,\rho(\alpha(\tau)))
\;\;\;\mbox{u.o.c.}\;\;\mbox{for all}\;\;u\in[0,T]
\elabel{baewbarqconujo}
\end{eqnarray}
where we have used Lemma~\ref{fluidlemma}, the uniqueness of the
limit, and \eq{barwhatw}.

Therefore it follows from the first part of Lemma~\ref{costcc},
\eq{baewbarqconujo}, and the similar argument as used in
\cite{yeyao:heatra} that, for all large enough $r\in\{1,2,...\}$ and
for all $u\in[0,T]$,
\begin{eqnarray}
&&\left\|\bar{Q}^{r,0}(u)-q^{*}(\bar{W}^{r,0}(u),
\rho(\alpha(\eta^{r,0}(u))))\right\|\leq\epsilon \elabel{fbarwbarq}
\end{eqnarray}
Thus \eq{barquI} presented in the lemma holds when $l=0$. Moreover,
it follows from \eq{baewbarqconujo} and \eq{unidifb} that the bound
estimations in \eq{equiclaimI} and \eq{equiclaimII} are true for all
$u\in[0,T]$ and all large enough $r$ when $l=0$. In addition, the
complementarity in \eq{equiclaimII} can be shown as follows. For the
given $\epsilon>0$ in the current lemma, it follows from the first
part of Lemma~\ref{costcc} and \eq{baewbarqconujo} that a $\sigma>0$
can be chosen such that, for large enough $r\in\{1,2,...\}$ and all
$u\in[0,T]$,
\begin{eqnarray}
&&\left\|\bar{Q}^{r,0}(u)-q^{*}(\bar{W}^{r,0}(u),
\rho(\alpha(\eta^{r,0}(u)))\right\|\leq\sigma \elabel{jothreeI}
\end{eqnarray}
since $\alpha(\eta^{r,0}(u))=\alpha(\tau)$ for all $u\in[0,T]$ when
$r$ is large enough. Thus, if $\bar{W}^{r,0}(u)>\epsilon$ for all
$u\in[0,T]$, then
\begin{eqnarray}
\bar{Y}^{r,0}(u)-\bar{Y}^{r,0}(0)
&=&\sum_{j=1}^{J}\int_{0}^{u}\left(\rho(\alpha(\eta^{r,0}(s)))
-\Lambda_{j}(\bar{Q}^{r}(\eta^{r,0}(s)),\alpha(\eta^{r,0}(s)))\right)ds
\elabel{reguyb}\\
&=&\sum_{j=1}^{J}\int_{0}^{u}\left(\rho(\alpha(\eta^{r,0}(s)))
-\Lambda_{j}(\bar{Q}^{r,0}(s),\alpha(\eta^{r,0}(s)))\right)ds
\nonumber\\
&=&0,\nonumber
\end{eqnarray}
for any $u\in[0,T]$, where the first equality of \eq{reguyb} follows
from \eq{hatys}, \eq{expresslambda} and the fact that
$\bar{Q}^{r,0}(s)\neq 0$ for all $s\in[0,u]\subset[0,T]$ due to the
assumption imposed in \eq{barquIII}, furthermore, the second
equality of \eq{reguyb} follows from \eq{homcon}, in addition, the
last equality of \eq{reguyb} follows from \eq{sumrhoi} in the second
part of Lemma~\ref{costcc}.

{\bf Part Two:} we prove the claims in the lemma for the case that
$l\in\{1,...,\lceil r\delta/T\rceil-1\}$ by showing a contradiction.
As a matter of fact, suppose that there is a subsequence ${\cal
R}_{1}$ of $r$ such that at least one of the claims stated in
\eq{barquI} and \eq{equiclaimI}-\eq{equiclaimII} does not hold for
any $r\in{\cal R}_{1}$ and some integer $l\in\{1,...,\lceil
r\delta/T\rceil-1\}$, where for later reference, we use
$l_{r}\in\{1,...,\lceil r\delta/T\rceil-1\}$ with $r\in{\cal R}_{1}$
to denote the smallest integer to have such property. However, we
can show that there is a subsequence ${\cal R}_{2}\subset{\cal
R}_{1}$ such that all the claims stated in \eq{barquI} and
\eq{equiclaimI}-\eq{equiclaimII} are true for $l=l_{r}$ and all
large enough $r\in{\cal R}_{2}$. To do so, we first construct a
subsequence ${\cal R}_{3}$ such that \eq{barquI} is true for
$l=l_{r}$ and all large enough $r\in{\cal R}_{3}$ as follows.

Due to the proof in the first part, we know that the claims stated
in \eq{barquI} and \eq{equiclaimI}-\eq{equiclaimII} are true for all
$l\in\{0,1,...,l_{r}-1\}$ and all large enough $r\in{\cal R}_{1}$.
So, for $l=l_{r}-1$, we have
\begin{eqnarray}
&&\left\|\bar{Q}^{r,l_{r}-1}(0)\right\|\leq\max\{b_{1,1},b_{1,2}\}
\;\;\mbox{for all}\;\;r\in{\cal R}_{1}. \elabel{lrcicii}
\end{eqnarray}
Then we know that $\{\bar{Q}^{r,l_{r}-1}(0),r\in{\cal R}_{1}\}$ has
a convergent subsequence from which we can find a further
subsequence ${\cal R}_{3}'\subset{\cal R}_{1}$ such that, along
$r\in{\cal R}_{3}'$,
\begin{eqnarray}
&&0\leq\frac{l_{r}-1}{r}\downarrow l_{\infty}\equiv\inf_{h\in{\cal
R}_{3}'}\left(\frac{l_{h}-1}{h}\right)<\infty \elabel{downconlr}
\end{eqnarray}
since $0\leq(l_{r}-1)/r\leq\delta/T$. Then it follows from
\eq{tequivexpres} and \eq{downconlr} that, for $r\in{\cal R}'_{3}$
and $u\in[0,2T]$,
\begin{eqnarray}
&&\eta^{r,l_{r}-1}(u)\downarrow\tau+l_{\infty}T\equiv\eta_{\infty}
\;\;\mbox{as}\;\;r\rightarrow\infty.\elabel{subetac}
\end{eqnarray}
Thus it follows from the definition of $\tau_{n}$ defined in
\eq{markovjt}, we know that
$[\eta_{\infty},\eta^{r,l_{r}-1}(u)]\subset[\tau_{n-1},\tau_{n})$
with some $n\in\{1,2,...\}$ for all $u\in[0,2T]$ and large enough
$r\in{\cal R}'_{3}$. Moreover, due to Lemma~\ref{fluidlemma}, there
is a subsequence ${\cal R}_{3}\subset{\cal R}_{3}'$ such that
\begin{eqnarray}
&&(\bar{W}^{r,l_{r}-1}(u),\bar{Q}^{r,l_{r}-1}(u))
\rightarrow(\bar{W}(u),\bar{Q}(u))\;\;\mbox{with}\;\;
\left\|\bar{Q}(0)\right\|\leq\max\{b_{1,1},b_{1,2}\} \elabel{subfbw}
\end{eqnarray}
u.o.c. over $u\in[0,2T]$ along ${\cal R}_{3}$. Hence
\begin{eqnarray}
&&\left\|\bar{Q}^{r,l_{r}-1}(u)-q^{*}(\bar{W}^{r,l_{r}-1}(u),
\rho(\alpha(\eta^{r,l_{r}-1}(u))))\right\|\elabel{difsubbq}\\
&\leq&\frac{\epsilon}{3}+\left\|\bar{Q}(u)-q^{*}(\bar{W}(u),
\rho(\alpha(\eta^{r,l_{r}-1}(u))))\right\|+\frac{\epsilon}{3}
\nonumber
\end{eqnarray}
holds over $u\in[0,2T]$ when $r\in{\cal R}_{3}$ is large enough,
where we have used \eq{subfbw} and the first part of
Lemma~\ref{costcc} for \eq{difsubbq}. Then, by \eq{barepsilon} in
Lemma~\ref{fluidlemma} and \eq{subetac}, we know that, for all
$u\in[T,2T]$ and large enough $r\in{\cal R}_{3}$,
\begin{eqnarray}
&&\left\|\bar{Q}(u)-q^{*}(\bar{W}(u),
\rho(\alpha(\eta^{r,l_{r}-1}(u))))\right\|<\frac{\epsilon}{3}
\elabel{binpqw}
\end{eqnarray}
since $\alpha(\eta^{r,l_{r}-1}(u))$ keeps a constant $i\in{\cal K}$
for all $u\in[0,2T]$ and large enough $r\in{\cal R}_{3}$, and
moreover, since
\begin{eqnarray}
&&T\geq T_{1}\geq T_{\max\{b_{1,1},b_{1,2}\},\epsilon/2}\nonumber
\end{eqnarray}
where $T_{1}$ is defined in \eq{tttone}. So, for large enough
$r\in{\cal R}_{3}$ and $u\in[0,T]$, it follows from \eq{difsubbq}
and \eq{binpqw} that
\begin{eqnarray}
&&\left\|\bar{Q}^{r,l_{r}}(u)-q^{*}(\bar{W}^{r,l_{r}}(u),
\rho(\alpha(\eta^{r,l_{r}}(u))))\right\|
\elabel{lastbqbw}\\
&=&\left\|\bar{Q}^{r,l_{r}-1}(T+u)
-q^{*}(\bar{W}^{r,l_{r}-1}(T+u),\rho(\alpha(\eta^{r,l_{r}-1}(T+u))))
\right\|\nonumber\\
&<&\epsilon\nonumber
\end{eqnarray}
where we have used \eq{difsubbq} for the inequality in
\eq{lastbqbw}. Then we know that the claim in \eq{barquI} is true
with $l=l_{r}$ for large enough $r\in{\cal R}_{3}$.

Next we divide ${\cal R}_{3}$ into the union of the following two
sets, that is, ${\cal R}_{3}={\cal R}_{4}\cup{\cal R}_{5}$, where
\begin{eqnarray}
&&{\cal R}_{4}\equiv\left\{r\in{\cal
R}_{3}:\bar{W}^{r,l_{r}}(u)\leq\epsilon\;\; \mbox{for
some}\;\;u\in[0,T]\right\},
\elabel{divsetI}\\
&&{\cal R}_{5}\equiv\left\{r\in{\cal
R}_{3}:\bar{W}^{r,l_{r}}(u)>\epsilon\;\;\mbox{for
all}\;\;u\in[0,T]\right\}.\elabel{dissetII}
\end{eqnarray}
Here we remark that at least one of ${\cal R}_{4}$ and ${\cal
R}_{5}$ must contain infinite numbers. So the remaining proof can be
divided into the following two parts.

Firstly, if ${\cal R}_{4}$ is infinite, then there is a fixed
$u_{r}\in[0,T]$ for each $r\in{\cal R}_{4}$ such that
\begin{eqnarray}
&&\bar{W}^{r,l_{r}}(u_{r})\leq\epsilon.\elabel{fbarle}
\end{eqnarray}
Moreover, there is a subset ${\cal R}_{4}'\subset{\cal R}_{4}$ such
that $u_{r}\rightarrow u'$ as $r\rightarrow\infty$ for $r\in{\cal
R}_{4}'$ and some $u'\in[0,T]$. Therefore we have
\begin{eqnarray}
&&\bar{W}(0)\leq\bar{W}(u')=\lim_{r\rightarrow\infty}
\bar{W}^{r,l_{r}}(u_{r})\leq\epsilon\elabel{ininreset}
\end{eqnarray}
where the first inequality in \eq{ininreset} follows from the
increasing property of $\bar{W}(\cdot)$, the equality in
\eq{ininreset} follows from \eq{subfbw} since
$\bar{W}^{r,l_{r}}(u_{r})=\bar{W}^{r,l_{r}-1}(T+u_{r})$, and the
second inequality in \eq{ininreset} follows from \eq{fbarle}. Thus
we have
\begin{eqnarray}
&&\left\|\bar{Q}(0)-q^{*}(\bar{W}(0),
\rho(\alpha(\eta_{\infty})))\right\|<\epsilon \elabel{estmnw}
\end{eqnarray}
where $\eta_{\infty}$ is defined in \eq{subetac} and the inequality
in \eq{estmnw} follows from \eq{subfbw}, the first part of Lemma
\ref{costcc}, and the fact that \eq{barquI} is true with $l=l_{r}$
for all large enough $r\in{\cal R}_{4}'\subset{\cal R}_{3}$ as
discussed above. Therefore it follows from \eq{estmnw},
\eq{ciciiine} and \eq{ininreset} that
\begin{eqnarray}
&&\left\|\bar{Q}(0)\right\|\leq\left\|q^{*}(\bar{W}(0),
\rho(\alpha(\eta_{\infty})))\right\|+\epsilon\leq(c_{2}+1)\epsilon.
\elabel{reesbnd}
\end{eqnarray}
Then, for all large enough $r\in{\cal R}_{4}'$ and all $u\in[0,T]$,
we have
\begin{eqnarray}
&&\left\|\bar{Q}^{r,l_{r}}(u)\right\|
\leq\left\|\bar{Q}(u)\right\|+\epsilon\leq b_{1,1}\elabel{iqcomp}
\end{eqnarray}
where $b_{1,1}$ is defined in \eq{dparamI} and the two inequalities
in \eq{iqcomp} follow from \eq{subfbw}, the similar argument as in
\eq{lastbqbw} and Lemma~\ref{uniattrack} respectively. Similarly,
for large enough $r\in{\cal R}_{4}'$ and all $u\in[0,T]$, we have,
\begin{eqnarray}
&&\bar{W}^{r,l_{r}}(u)\leq\bar{W}(u)+\epsilon\leq
c_{1}\left\|q^{*}(\bar{W}(u),\rho(\alpha(\eta_{\infty})))\right\|
+\epsilon\leq b_{2,1},\elabel{reswq}
\end{eqnarray}
where the two inequalities in \eq{reswq} follows from \eq{subfbw}
and the similar argument as in \eq{lastbqbw}. Then it follows from
\eq{iqcomp}-\eq{reswq} that \eq{equiclaimI} is true for $l=l_{r}$
for large enough $r\in{\cal R}_{4}$.

Secondly, if ${\cal R}_{5}$ is infinite, we can choose
$\sigma=\sigma(\epsilon)$ as in Lemma~\ref{costcc}. Then, it follows
from Lemma~\ref{uniattrack} that, for all $u\in[0,T]$,
\begin{eqnarray}
&&\left\|\bar{Q}(T+u)-q^{*}(\bar{W}(T+u),\rho(\alpha(\eta_{\infty})))
\right\|\leq\frac{\sigma}{2}\elabel{fiveinf}
\end{eqnarray}
where $\alpha(\eta_{\infty})=\alpha(\eta^{r,l_{r}-1}(T+u))$ keeps a
constant $i\in{\cal K}$ for all $u\in[0,2T]$ and all large enough
$r\in{\cal R}_{5}$ and moreover, the chosen time $T$ satisfies
\begin{eqnarray}
&&T\geq T_{1}\geq T_{\max\{b_{1,1},b_{1,2}\},\sigma/2}\nonumber
\end{eqnarray}
with $T_{1}$ defined in \eq{tttone}. Thus, for all large enough
$r\in{\cal R}_{5}$ and all $u\in[0,T]$, we have
\begin{eqnarray}
&&\left\|\bar{Q}^{r,l_{r}-1}(T+u)-q^{*}(\bar{W}^{r,l_{r}-1}(T+u),
\rho(\alpha(\eta^{r,l_{r}-1}(T+u))))\right\|<\sigma \elabel{fiveine}
\end{eqnarray}
where the inequality follows from the similar explanations as used
for \eq{estmnw}. Therefore, by \eq{fiveine}, \eq{sumrhoi} in the
second part of Lemma~\ref{costcc}, and the fact that
\begin{eqnarray}
&&\bar{W}^{r,l_{r}-1}(T+u)=\bar{W}^{r,l_{r}}(u)>\epsilon\;\;\mbox{for
all}\;\;u\in[0,T], \nonumber
\end{eqnarray}
we know that $\bar{Y}^{r,l_{r}-1}(T+u)$ does not increase over
$u\in[0,T]$ for all large enough $r\in{\cal R}_{5}$, i.e.,
\begin{eqnarray}
&&\bar{Y}^{r,l_{r}}(u)-\bar{Y}^{r,l_{r}}(0)=0\;\;\mbox{for
all}\;\;u\in[0,T]. \elabel{lastbary}
\end{eqnarray}
To finish the remaining proof based on \eq{lastbary}, we need to
consider the following two mutually exclusive cases for a given
large enough $r\in{\cal R}_{5}$.

Case One: the condition in \eq{equiclaimII} is true for all
$l\in\{0,1,...,l_{r}\}$. Then we know that $\bar{Y}^{r,l_{r}}(u)$
does not increase over $u\in[0,T]$ for all $l\in\{0,1,...,l_{r}\}$
due to the induction assumption and \eq{lastbary}. So, for large
enough $r\in{\cal R}_{5}$ and all $u\in[0,T]$, we have,
\begin{eqnarray}
\bar{W}^{r,l_{r}}(u)&=&\bar{W}^{r,0}(0)
+\sum_{l=0}^{l_{r}-1}\left(\bar{W}^{r,l}(T)-\bar{W}^{r,l}(0)\right)
+\left(\bar{W}^{r,l_{r}}(u)-\bar{W}^{r,l_{r}}(0)\right)
\elabel{fcaseI}\\
&=&\hat{W}^{r}(\tau)+\left(\hat{X}^{r}(\eta^{r,l_{r}}(u))
-\hat{X}^{r}(\eta^{r,0}(0))\right)\nonumber\\
&\leq&(\nu+\epsilon)+(C+\epsilon) \nonumber
\end{eqnarray}
where the second equality in \eq{fcaseI} follows from
\eq{difws}-\eq{hatys}, and the inequality in \eq{fcaseI} follows
from \eq{barwnu}, \eq{subetac}, \eq{hatxs} and \eq{unidifb}.

Case Two: the condition in \eq{equiclaimI} is true for some
$l\in\{0,1,...,l_{r}-1\}$ and use $l^{m}_{r}$ to denote the largest
such integer. Then both the condition and the claim in
\eq{equiclaimII} are true for all $l\in\{l_{r}^{m}+1,...,l_{r}\}$
and therefore the corresponding $\bar{Y}^{r,l}(u)$ does not increase
over $u\in[0,T]$ due to the induction assumption and the discussion
as in \eq{fiveine}-\eq{lastbary}. Moreover, by the same discussion
as used in \eq{subetac}, there is a subsequence ${\cal
R}_{5}'\subset{\cal R}_{5}$ such that $\eta^{r,l_{r}^{m}}(T)$
converges along $r\in {\cal R}_{5}'$. Thus, similar to \eq{fcaseI},
for large enough $r\in{\cal R}_{5}'$ and all $u\in[0,T]$, we have
\begin{eqnarray}
\bar{W}^{r,l_{r}}(u)
&=&\bar{W}^{r,l_{r}^{m}}(T)+\left(\hat{X}^{r}(\eta^{r,l_{r}}(u))
-\hat{X}^{r}(\eta^{r,l_{r}^{m}}(T))\right)\elabel{fcaseII}\\
&\leq&b_{2,1}+(C+\epsilon) \nonumber
\end{eqnarray}
where the inequality in \eq{fcaseII} follows from \eq{subfbw}, the
induction assumption (since $l_{r}^{m}<l_{r}$), \eq{hatxs} and
\eq{unidifb}.

Therefore, it follows from both of the discussions in Case One and
Case Two that, for large enough $r\in{\cal R}_{5}'$ and all
$u\in[0,T]$, we have,
\begin{eqnarray}
\bar{W}^{r,l_{r}}(u)&\leq&\max\{(\nu+\epsilon)+(C+\epsilon),
b_{2,1}+(C+\epsilon)\}\elabel{lfinalI}\\
&=&b_{2,2},\nonumber\\
\left\|\bar{Q}^{r,l_{r}}(u)\right\|&\leq&
\left\|q^{*}(\bar{W}^{r,l_{r}}(u),\rho(\alpha(\eta^{r,l_{r}}(u))))
\right\|+\epsilon\elabel{lfinalII}\\
&\leq&c_{2}\bar{W}^{r,l_{r}}(u)+\epsilon\nonumber\\
&\leq&b_{1,2},\nonumber
\end{eqnarray}
where the first inequality in \eq{lfinalII} follows from
\eq{lastbqbw}, the second inequality in \eq{lfinalII} follows from
\eq{ciciiine}, and the third inequality in \eq{lfinalII} follows
from \eq{lfinalI}. Thus, by \eq{lastbary} and
\eq{lfinalI}-\eq{lfinalII}, we know that \eq{equiclaimII} is true
with $l=l_{r}$ for large enough $r\in{\cal R}_{5}'$.

In the end, take ${\cal R}_{2}={\cal R}_{4}'\cup{\cal R}_{5}'$ and
then we have that \eq{barquI} and \eq{equiclaimI}-\eq{equiclaimII}
are true in terms of $l=l_{r}$ for large enough $r\in{\cal
R}_{2}\subset{\cal R}_{1}$. This is a contradiction and hence we
finish the proof of Lemma~\ref{uniattract}. $\Box$
\end{proof}

\subsection{\bf Proof of Theorem~\ref{rsdifth}}\label{mainthp}

As in the proof of Lemma~\ref{uniattract}, our discussion will base
on each particular sample path. For convenience, we divide the proof
into two parts.

{\bf Part One.} In this part, we prove the convergence in
distribution as stated in \eq{qwyweakc} and the related properties
\eq{mainmodel}-\eq{qqast}. First of all, since it may be not true
that any subsequence of $\{\hat{Y}^{r}(t),r\in\{1,2,...,\}\}$ exists
a further subsequence that converges to a continuous and
nondecreasing limit $\hat{Y}(t)$ when $\hat{Y}^{r}(t)$ are unbounded
(e.g., $\hat{Y}^{r}(t)=(logr)t$). So, we employ
Lemmas~\ref{fluidlemma}-\ref{uniattract} to provide a justification
in terms of u.o.c. convergence for
$\{\hat{Y}^{r}(t),r\in\{1,2,...,\}\}$, which can be considered as an
supplementary illustration to the corresponding claims used in
\cite{yeyao:heatra}, \cite{sto:maxsch}, and etc. As a matter of
fact, since $Q^{r}(0)=0$ for all $r\in\{1,2,...\}$, we can conclude
that the conditions stated in \eq{barwnu} of Lemma~\ref{uniattract}
are satisfied. Moreover, due to \eq{hatxs}, we know that
\eq{unidifb} is true for an arbitrarily chosen constant $C>0$ over
any given interval $[0,T]\supset[0,T_{1}]$, where $T_{1}$ is defined
in \eq{tttone}. So, by \eq{difws}, \eq{barwhatw},
Lemmas~\ref{fluidlemma}-\ref{uniattract}, we know that, for any
$t\in[0,T]$ and each large enough $r\in\{1,2,...\}$, there is a
$u\in[0,T]$ and $l\in\{0,1,...,\lceil r\delta/T\rceil-1\}$ such
that, for any given small enough $\epsilon>0$,
\begin{eqnarray}
&&0\leq\hat{Y}^{r}(t)=\bar{W}^{r,l}(u)-\hat{X}^{r}(t)\leq
C+O(\epsilon)+K,\elabel{hatynuib}
\end{eqnarray}
where $K$ is some positive constant due to \eq{hatxs} and the
continuity of $\hat{X}(t)$. Therefore we know that $\hat{Y}^{r}(t)$
is uniformly bounded over the given interval $[0,T]$ for all
$r\in\{1,2,...\}$. Moreover, since $\hat{Y}^{r}(t)$ for each
$r\in\{1,2,...\}$ is nondecreasing and continuous with
$\hat{Y}^{r}(0)=0$, it follows from the Helly's Theorem (e.g.,
Theorem 2 in page 319 of \cite{shi:pro}) that, for any subsequence
of these processes, there is a further subsequence ${\cal
R}\subset\{1,2,...\}$ such that
\begin{eqnarray}
&&\hat{Y}^{r}(t)\rightarrow\hat{Y}(t)\;\;\mbox{for
every}\;\;t\in[0,T]\;\;\mbox{along}\;\;r\in{\cal R} \elabel{shatyc}
\end{eqnarray}
where $\hat{Y}(t)$ is also nondecreasing and continuous with
$\hat{Y}(t)=0$ over $[0,T]$.

Next, take $T\in\{1,2,...\}$ and let $T\rightarrow\infty$ since
$T\geq T_{1}$ is arbitrarily taken, we know that there is a further
subsequence ${\cal R}_{1}\subset{\cal R}$ such that the convergence
in \eq{shatyc} is extended to the whole interval $[0,\infty)$ along
$r\in{\cal R}_{1}$, and $\hat{Y}(t)$ is nondecreasing and continuous
over $[0,\infty)$. Thus, it follows from Theorem 2.15 in page 342,
Corollary 2.24 in page 345, and Proposition 1.17(b) of
\cite{jacshi:limthe} that the convergence in \eq{shatyc} is u.o.c.
over $[0,\infty)$. Consequently, it follows from \eq{difws} and
\eq{hatxs} that, along $r\in{\cal R}_{1}$,
\begin{eqnarray}
&&\hat{W}^{r}(t)\rightarrow\hat{W}(t)=\hat{X}(t)+\hat{Y}(t)\geq
0\;\;\mbox{u.o.c. over}\;\;t\in[0,\infty),\elabel{hatsubw}
\end{eqnarray}
which is continuous in $t\in[0,\infty)$.

Thus it follows from \eq{shatyc}-\eq{hatsubw},
Lemma~\ref{fluidlemma}, and the similar argument as used in
\cite{yeyao:heatra}, we know that the complementary property as
stated in Theorem~\ref{rsdifth} is true. Furthermore, take a number
$\delta>0$, then for a given $\epsilon>0$, it follows from
\eq{barquI} in Lemma~\ref{uniattract} that, for large enough
$r\in{\cal R}_{1}$,
\begin{eqnarray}
&&\sup_{t\in[0,\delta]}\left\|\hat{Q}^{r}(t)-q^{*}(\hat{W}^{r}(t),
\rho(\alpha(t)))\right\|\leq\epsilon, \elabel{difqqw}
\end{eqnarray}
or equivalently, for each $j\in\{1,2\}$, $t\in[0,\delta]$ and large
enough $r\in{\cal R}_{1}$, we have
\begin{eqnarray}
&&q^{*}_{j}(\hat{W}^{r}(t), \rho(\alpha(t)))-\epsilon\leq
\hat{Q}^{r}_{j}(t)\leq q^{*}_{j}(\hat{W}^{r}(t),
\rho(\alpha(t)))+\epsilon.\elabel{difqqwI}
\end{eqnarray}
Then it follows from Lemma~\ref{costcc} and \eq{hatsubw} that the
following convergence is true (e.g., let $r\rightarrow\infty$ first
and let $\epsilon\rightarrow 0$ later in \eq{difqqwI}),
\begin{eqnarray}
&&\hat{Q}^{r}(t)\rightarrow\hat{Q}(t)\equiv
q^{*}(\hat{W}(t),\rho(\alpha(t)))\;\;\mbox{uniformly
over}\;\;t\in[0,\delta]. \elabel{lastqcon}
\end{eqnarray}
Since $\delta$ is arbitrarily taken, the convergence stated in
\eq{lastqcon} can be considered true u.o.c. over $[0,\infty)$.
Therefore, we have
\begin{eqnarray}
&&(\hat{Q}^{r}(t),\hat{W}^{r}(t),\hat{Y}^{r}(t))
\rightarrow(\hat{Q}(t),\hat{W}(t),\hat{Y}(t))\;\;\mbox{u.o.c.
over}\;\;[0,\infty)\;\;\mbox{along}\;\;r\in{\cal R}_{1}
\elabel{jointwxycon}
\end{eqnarray}
with the limit satisfying all the requirements as stated in
Theorem~\ref{rsdifth}. Consequently, due to the uniqueness of
solution to the associated Skorohod problem (see, e.g.,
\cite{cheyao:funque}, or \cite{dai:broapp} and
\cite{daidai:heatra}), we know that the convergence in
\eq{jointwxycon} is true along $r\in\{1,2,...\}$.

\vskip 0.1cm {\bf Part Two.} In this part, we prove the optimality
claims stated in \eq{optimaleqn}-\eq{optimaleqnI} along the line of
\cite{yeyao:heatra}, however, the justification logic and technical
treatment are somewhat different. First of all, suppose that all the
processes related to an arbitrarily given feasible allocation scheme
$G$ will be superscripted by an additional $G$. Moreover, for each
$t\in[0,\infty)$, we define
\begin{eqnarray}
&&\hat{W}^{G}(t)\equiv\liminf_{r\rightarrow\infty}\hat{W}^{r,G}(t)
\elabel{infhatw}
\end{eqnarray}
which may be infinitely-valued. In other words, for any particularly
given $t\in[0,\infty)$, there exists a subsequence ${\cal
T}\subset\{1,2,...\}$ such that
\begin{eqnarray}
&&\hat{W}^{G}(t)=\lim_{r\rightarrow\infty}\hat{W}^{r,G}(t)\;\;
\mbox{along}\;\;r\in{\cal T}. \elabel{hatlimitw}
\end{eqnarray}
Moreover, let ${\cal Q}$ denote the set of all the nonnegative
rational numbers. Thus there exists a subsequence ${\cal R}\in{\cal
T}$ such that
\begin{eqnarray}
&&\hat{W}^{r,G}(s)\rightarrow\hat{W}^{G}(s)\;\;\mbox{along}\;\;
r\in{\cal R}\;\;\mbox{for each}\;\;s\in{\cal Q}.
\elabel{rationalcon}
\end{eqnarray}
In addition, by applying the similar discussion as in
Lemma~\ref{fluidlemma}, we can select a subsequence ${\cal
R}_{1}\subset{\cal R}$ such that, along $r\in{\cal R}_{1}$,
\begin{eqnarray}
&&\bar{T}^{r,G}(s)\rightarrow\bar{T}^{G}(s)\;\;\mbox{u.o.c.
over}\;\;s\in[0,\infty)\;\;\mbox{as}\;\;r\rightarrow\infty
\elabel{finalfluidI}
\end{eqnarray}
where $\bar{T}^{G}(s)$ is Lipschitz continuous and increasing with
$\bar{T}^{G}(0)=0$. Furthermore, we can see that $\bar{Q}^{r,G}(s)$,
$\bar{W}^{r,G}(s)$ and $\bar{Y}^{r,G}(s)$ also converge u.o.c. to
$\bar{Q}^{G}(s)$, $\bar{W}^{G}(s)$ and $\bar{Y}^{G}(s)$ along
$r\in{\cal R}_{1}$, which are Lipschitz continuous and satisfy the
following relationships
\begin{eqnarray}
&&\bar{Q}^{G}_{j}(s)=\bar{\lambda}_{j}(s)-\mu_{j}
\bar{T}^{G}_{j}(s)\geq 0\;\;\mbox{for each}\;\;j\in{\cal J},
\elabel{flimqtnon}\\
&&\bar{W}^{G}(s)=\sum_{j=1}^{J}\frac{\bar{Q}^{G}_{j}(s)}{\mu_{j}}
=\bar{Y}^{G}(s),\elabel{flimwtnon}\\
&&\bar{Y}^{G}(s)=\sum_{j=1}^{J}\left(\int_{0}^{s}
\rho_{j}(\alpha(u))du-\bar{T}^{G}_{j}(s)\right), \elabel{flimwtnonI}
\end{eqnarray}
where $\bar{Y}^{G}(s)$ is nondecreasing with $\bar{Y}^{G}(0)=0$. To
further investigate, we define
\begin{eqnarray}
&&\zeta=\inf\left\{s\geq
0:\bar{T}^{G}_{j}(s)\neq\bar{c}_{j}(s)\;\;\mbox{for
some}\;\;j\in{\cal J}\right\}\elabel{ftaui}
\end{eqnarray}
where $\bar{c}_{j}(s)$ is defined in \eq{barvc}. Then, under the
policy $G$, it follows from the similar discussion as in \eq{hatxs}
that
\begin{eqnarray}
&&\hat{X}^{r,G}(s)\rightarrow\hat{X}^{G}(s)\;\;\mbox{u.o.c.
over}\;\;s\in[0,\zeta)\;\;\mbox{along}\;\;r\in{\cal R}_{1}.
\elabel{hatxg}
\end{eqnarray}
So it follows from \eq{rationalcon} that
\begin{eqnarray}
&&\hat{Y}^{r,G}(s)\rightarrow\hat{\gamma}^{G}(s)\;\;
\mbox{along}\;\;r\in{\cal R}_{1}\;\;\mbox{for each}\;\;s\in{\cal
Q}\elabel{rhatycon}
\end{eqnarray}
where $\hat{\gamma}^{G}(s)$ is some discrete function in $s\in{\cal
Q}$ and is nondecreasing since $\hat{Y}^{r,G}(s)$ is nondecreasing
for each $r\in{\cal R}_{1}$. Moreover, define
\begin{eqnarray}
&&\zeta_{1}=\inf\left\{s\geq 0:\hat{\gamma}^{G}(s)=
+\infty,s\in{\cal Q}\right\},\elabel{zetaone}
\end{eqnarray}
then we know that $\{\hat{Y}^{r,G}(s),r\in{\cal R}_{1}\}$ is
uniformly bounded over any compact set of
$[0,\zeta\wedge\zeta_{1})$. Thus it follows from the similar
explanation as used for \eq{shatyc} that there is a subsequence
${\cal R}_{2}\subset{\cal R}_{1}$ such that
\begin{eqnarray}
&&\hat{Y}^{r,G}(s)\rightarrow\hat{Y}^{G}(s)\;\;\mbox{for
each}\;\;s\in[0,\zeta\wedge\zeta_{1})\;\;\mbox{along}\;\; r\in{\cal
R}_{2} \elabel{intuoc}
\end{eqnarray}
where $\hat{Y}^{G}(s)$ is continuous and nondecreasing with
$\hat{Y}^{G}(0)=0$, and moreover, it satisfies
\begin{eqnarray}
&&\hat{Y}^{G}(s)=\hat{\gamma}^{G}(s)\;\;\mbox{for all}\;\;s\in{\cal
Q}\cap[0,\zeta\wedge\zeta_{1}).\nonumber
\end{eqnarray}
Then it follows from \eq{hatxg}, \eq{intuoc} and the similar
expression as in \eq{difws} that, along $r\in{\cal R}_{2}$ and for
each $s\in[0,\zeta\wedge\zeta_{1})$,
\begin{eqnarray}
&&\hat{\beta}^{G}(s)\equiv\lim_{r\rightarrow\infty}\hat{W}^{r,G}(s)
=\hat{X}^{G}(s)+\hat{Y}^{G}(s)\geq 0. \elabel{betaweq}
\end{eqnarray}
However, the complementarity may not be true for
$(\hat{W}^{G}(t),\hat{Y}^{G}(t))$. Therefore it follows from
\eq{intuoc}-\eq{betaweq} and the minimality of the Skorohod problem
(see, e.g., \cite{cheyao:funque}, \cite{dai:broapp},
\cite{daidai:heatra}, and \cite{harrei:refbro}) that
\begin{eqnarray}
&&\hat{\beta}^{G}(s)\geq\hat{W}(s)\;\;\mbox{for all}\;\;
s\in[0,\zeta\wedge\zeta_{1}).\elabel{fzetazetao}
\end{eqnarray}
Hence, if $t\in[0,\zeta\wedge\zeta_{1})$, then we know that, along
$r\in{\cal R}_{2}$,
\begin{eqnarray}
&&\hat{W}^{G}(t)=\lim_{r\rightarrow\infty,\;r\in{\cal
R}}\hat{W}^{r,G}(t)=\lim_{r\rightarrow\infty,\;r\in{\cal
R}_{2}}\hat{W}^{r,G}(t) =\hat{\beta}^{G}(t)\geq\hat{W}(t),
\elabel{hatwhatb}
\end{eqnarray}
which is always true if $\zeta=\zeta_{1}=\infty$.

Furthermore, if $\zeta<\zeta_{1}$ or $\zeta=\zeta_{1}<\infty$, and
$t\in[\zeta,\infty)$, then we can take a time $\tau\in[\zeta,t]$
such that $\bar{T}^{G}_{j}(\tau)\neq\bar{c}_{j}(\tau)$ for some
$j\in{\cal J}$. So it follows from \eq{flimqtnon} that
$\bar{T}^{G}_{j}(\tau)<\bar{c}_{j}(\tau)$ and
$\bar{Q}_{j}^{G}(\tau)>0$ for the $j$. Then it follows from
\eq{flimwtnon}-\eq{flimwtnonI} that
$\bar{W}^{G}(t)\geq\bar{W}^{G}(\tau)>0$. Therefore, along $r\in{\cal
R}_{2}$, we have
\begin{eqnarray}
&&\hat{W}^{G}(t)=\lim_{r\rightarrow\infty,\;r\in{\cal
R}_{2}}\hat{W}^{r,G}(t) =\lim_{r\rightarrow\infty,\;r\in{\cal
R}_{2}}r\bar{W}^{r,G}(t)=+\infty\geq\hat{W}(t). \elabel{zetaaf}
\end{eqnarray}
In addition, if $\zeta>\zeta_{1}$ and $t\in[\zeta_{1},\infty)$, then
it follows from \eq{ftaui} that
\begin{eqnarray}
&&\liminf_{r\rightarrow\infty,\;r\in{\cal
R}_{2}}\hat{Y}^{r,G}(t)\geq \lim_{r\rightarrow\infty,\;r\in{\cal
R}_{2}}\hat{Y}^{r,G}(\zeta_{1})
=\hat{\gamma}^{G}(\zeta_{1})=+\infty. \elabel{zetaoney}
\end{eqnarray}
Thus, by \eq{hatxg}, we know that
\begin{eqnarray}
&&\hat{W}^{G}(t)=\lim_{r\rightarrow\infty,\;r\in{\cal
R}_{2}}\hat{W}^{r,G}(t)=+\infty\geq\hat{W}(t). \elabel{finalwg}
\end{eqnarray}
Since the given time $t\in[0,\infty)$ is arbitrarily taken, it
follows from \eq{hatwhatb}, \eq{zetaaf} and \eq{finalwg} that the
claim \eq{optimaleqn} in the theorem is true for any $t\geq 0$.

In the end, it follows from \eq{optimaleqn} and \eq{qqast} that
\eq{optimaleqnI} is true. Hence we finish the proof of
Theorem~\ref{rsdifth}. $\Box$

\section{Proofs of Lemmas~\ref{smoothsurfaces} and~\ref{smoothsurfacesI}}
\label{prooflemmas}

First of all, to be convenient for readers, we outline the proofs of
the two lemmas as follows:

For the proof of Lemma~\ref{smoothsurfaces}, we first use the
optimization technique studied in \cite{goljaf:caplim} and
\cite{yurhe:itewat} to characterize the boundary of the MAC capacity
region presented in \eq{macform}, i.e., the region in \eq{macform}
is convex and thus the boundary of it can be fully characterized by
maximizing the function $\sum_{j=1}^{J}\nu_{j}r_{j}$ over all rate
vectors in the region and for all nonnegative priority vectors
$\nu=(\nu_{1},...,\nu_{J})$ such that $\sum_{j=1}^{J}\nu_{j}=1$.
Then, based on priority vectors and permutation schemes, we can
determine the number of boundary pieces of the region, which is
consistent with what is obtained in \cite{liuhou:weipro}. Finally,
by applying the KKT optimality conditions and the implicit function
theorem, we can prove that the boundary of the MAC capacity region
consists of the derived number of linear or smooth curved facets.

For the proof of Lemma~\ref{smoothsurfacesI}, we use the duality of
the capacity regions between MAC and BC to transform the discussion
for BC to the one for MAC (see, e.g., \cite{goljaf:caplim}).

\subsection{Proof of Lemma~\ref{smoothsurfaces}}

Notice that, for a fixed priority vector $\nu$, the optimization
characterization described in the outline is equivalent to finding
the point on the capacity boundary that is tangent to a line whose
slope is defined by the priority vector. Due to the structure of the
capacity region, we can see that all boundary points of the region
are corner points of polyhedrons corresponding to different sets of
covariance matrices. In addition, the corner point should correspond
to successive decoding in order of increasing priority, i.e., the
user with the highest priority should be decoded last and,
therefore, sees no interference. Hence, by \cite{goljaf:caplim} and
\cite{yurhe:itewat}, the problem of finding the boundary point on
the capacity region associated with a descending ordered priority
vector $\nu$ can be written as
\begin{eqnarray}
&&\max_{\{\Gamma_{j}(i)\geq 0,\mbox{Tr}(\Gamma_{j}(i))\leq
P_{j},j\in{\cal
J}\}}f(\Gamma_{1}(i),...,\Gamma_{J}(i),\nu)\elabel{bounopt}
\end{eqnarray}
where
\begin{eqnarray}
&&f(\Gamma_{1}(i),...,\Gamma_{J}(i),\nu)\elabel{fgamma}\\
&=&\nu_{J}\mbox{log}
\left|I+\sum_{j=1}^{J}H_{j}^{\dagger}(i)\Gamma_{j}(i)H_{j}(i)\right|
+\sum_{j=1}^{J-1}\left((\nu_{j}-\nu_{j+1})\mbox{log}
\left|I+\sum_{l=1}^{j}H_{l}^{\dagger}(i)\Gamma_{l}(i)H_{l}(i)
\right|\right) \nonumber
\end{eqnarray}
which is concave in the covariance matrices.

Now, let $\tilde{\nu}_{j}=\nu_{j}-\nu_{j+1}$ for $j\in\{1,...,J-1\}$
and $\tilde{\nu}_{J}=\nu_{J}$. Then, for any integer
$m\in\{1,...,J-1\}$, let $S(k_{1},...,k_{m})$ denote the following
set corresponding to exactly having $m$ indices
$k_{1},...,k_{m}\in\{1,...,J-1\}$ such that
$\tilde{\nu}_{k_{1}}=...=\tilde{\nu}_{k_{m}}=0$, i.e.,
\begin{eqnarray}
&&S(k_{1},...,k_{m})\equiv\{f(\Gamma_{1}(i),...,
\Gamma_{J}(i),\nu):\Gamma_{j}(i)\geq 0,\tilde{\nu}_{j}\geq
0\;\;\mbox{for}\;\;j\in{\cal
J},\tilde{\nu}_{k_{1}}=...=\tilde{\nu}_{k_{m}}=0,\elabel{mseteq}\\
&&\;\;\;\;\;\;\;\;\;\;\;\;\;\;\;\;\;\;\;\;\;\;\;\;\;\;\;\;\;\;\;\;\;\;\;\;
k_{j}\neq k_{l}\;\;\mbox{for}\;\; j\neq
l\;\;\mbox{and}\;\;j,l\in\{1,...,m\},\;\nu_{J}>0,\;
\sum_{j=1}^{J}\nu_{j}=1\}.\nonumber
\end{eqnarray}
Moreover, if $m=0$, we use $S(k_{0})$ to denote the set
corresponding to $\tilde{\nu}_{j}>0$ for all $j\in{\cal J}$, i.e.,
\begin{eqnarray}
&&S(k_{0})\equiv\left\{f(\Gamma_{1}(i),...,
\Gamma_{J}(i),\nu):\Gamma_{j}(i)\geq
0,\tilde{\nu}_{j}>0\;\;\mbox{for}\;\;j\in{\cal
J},\nu_{J}>0,\sum_{j=1}^{J}\nu_{j}=1\right\}. \elabel{mseteqI}
\end{eqnarray}
In addition, if $m=J$, $\nu_{J}=0$, we use $S(k_{J})$ to denote the
following set corresponding to $\nu_{J}=0$,
\begin{eqnarray}
&&S(k_{J})\equiv\left\{f(\Gamma_{1}(i),...,
\Gamma_{J}(i),\nu):\Gamma_{j}(i)\geq 0\;\;\mbox{for}\;\;j\in{\cal
J},\;\nu_{J}=0,\;\sum_{j=1}^{J}\nu_{j}=1\right\}. \elabel{mseteqII}
\end{eqnarray}
Eventually, we can define
\begin{eqnarray}
&&S(k_{0},k_{1},...,k_{m})=\left\{\begin{array}{ll}
S(k_{1},...,k_{m})&\mbox{if}\;\;m\in{\cal J},\\
S(k_{0})&\mbox{if}\;\;m=0,\\
S(k_{J})&\mbox{if}\;\;m=J.
\end{array}\right.\elabel{mseteqII}
\end{eqnarray}
Thus we have
\begin{eqnarray}
&&\left\{f(\Gamma_{1}(i),...,\Gamma_{J}(i),\nu): \Gamma_{j}(i)\geq
0,\;\nu_{j}\geq 0,\;\tilde{\nu}_{j}\geq 0\;\;\mbox{for}\;\;j\in{\cal
J},
\sum_{j=0}^{J}\nu_{j}=1\right\} \elabel{fgdecomp}\\
&&=\bigcup_{m=1}^{J} \bigcup_{k_{1},...,k_{m}\in{\cal
J}}S(k_{0},k_{1},...,k_{m}). \nonumber
\end{eqnarray}

Note that the $J$ users can be arbitrarily ordered, so we have $J!$
such priority orders, e.g.,
$\nu_{j_{1}}\geq\nu_{j_{2}}\geq...\geq\nu_{j_{J}}$, where
$(j_{1},...,j_{J})$ is a permutation of $(1,...,J)$. Thus we can see
that our capacity region is bounded by $L$ boundary pieces with $L$
given by \eq{numberL}. In fact, the first term $J!$ on the
right-hand side of the first equality in \eq{numberL} is the number
of boundary pieces corresponding to all
$\nu_{j_{1}}>\nu_{j_{2}}>...>\nu_{j_{J}}$, $C_{J}^{j}(J-j+1)!$
($j\in\{2,...,J\}$) is the number of boundary pieces corresponding
to all $\nu_{k_{1}}=...=\nu_{k_{j}}$ with
$k_{1},...,k_{j}\in\{1,...,J\}$ and $k_{l}\neq k_{h}$ for $l\neq h$
when $\nu_{j_{J}}>0$, and the last term $J$ on the right-hand side
of \eq{numberL} is the number of boundary pieces corresponding to
$\nu_{j_{J}}=0$. Here we remark that the number of boundary pieces
obtained through the above method is consistent with the one derived
in \cite{liuhou:weipro}.

Next we show the smoothness of these boundary pieces. Without loss
of generality, our discussion will focus on a specific set
$S(k_{0})$ in a particular user priority order since the discussions
for all other cases are similar. Therefore, we have that
$\nu_{1}>\nu_{2}>...>\nu_{J}>0$. Moreover, let
$y=(y_{1},...,y_{2NNJ})'$ denote the $(2NNJ)$-dimensional vector
formed by the real part and the imaginary part of entries of
$\Gamma_{1}(i)$,...,$\Gamma_{J}(i)$ in a suitable order. Thus we
know that $f(\Gamma_{1}(i),...,\Gamma_{J}(i),\nu)=f(y,\nu)$ is
concave in $y$ for each given $\nu\geq 0$. Then the optimization
problem in \eq{bounopt} can be restated as follows.
\begin{eqnarray}
&&\max_{y\geq 0}f(y,\nu)\elabel{transopt}
\end{eqnarray}
subject to
\begin{eqnarray}
&&f_{j}(y)\equiv\mbox{Tr}(\Gamma_{j}(i))-P_{j}\leq 0\;\;\;\mbox{for
all}\;\;\;j\in{\cal J}.\elabel{fconstraint}
\end{eqnarray}
So it follows from the KKT optimality conditions (see, e.g.,
\cite{lue:linnon}) that the solution to the optimization problem in
\eq{transopt}-\eq{fconstraint} for a function $f(y,\nu)\in S(k_{0})$
with the associated $\nu\geq 0$ can be obtained through the
following equations,
\begin{eqnarray}
&&y_{l}\left(\frac{\partial f(y,\nu)}{\partial
y_{l}}+\sum_{j=1}^{J}\eta_{j}\frac{\partial f_{j}(y)}{\partial
y_{l}}\right)=0\;\;\mbox{for each}\;\;l\in\{1,2,...,2NNJ\},
\elabel{fuoptsoI}\\
&&\eta_{j}f_{j}(y)=0\;\;\mbox{for each}\;\;j\in{\cal J},
\elabel{fuoptsoII}
\end{eqnarray}
where $\eta_{j}\geq 0$ for $j\in{\cal J}$ are the Lagrangian
multipliers. Then our remaining discussion can be divided into the
following two steps.

{\bf Step One:} If there exists some $\nu\in{\cal N}$ such that the
problem in \eq{transopt}-\eq{fconstraint} for the function
$f(y,\nu)\in S(k_{0})$ has at least one optimal solution located in
the interior of the associated feasible region, where
\begin{eqnarray}
&&{\cal
N}\equiv\left\{\nu=(\nu_{1},...,\nu_{J})':\nu_{1}>...>\nu_{J}>0,
\sum_{j=1}^{J}\nu_{j}=1\right\}, \elabel{setoneI}
\end{eqnarray}
then we have the following discussions.

Firstly, we suppose that the optimal solution is unique given by
$y^{*}=(y^{*}_{1},...,y^{*}_{2NNJ})'$. Then we know that $f(y,\nu)$
is strictly concave since it is sufficiently smooth in $y$ for the
given $\nu$ due to the definition of $f$. So it follows from
\eq{fuoptsoI}-\eq{fuoptsoII} that
\begin{eqnarray}
&&F_{l}(y^{*},\nu)\equiv\frac{\partial f(y^{*},\nu)}{\partial
y_{l}}=0\;\;\mbox{for all}\;\;l\in\{1,...,2NNJ\}.
\elabel{minuscompI}
\end{eqnarray}
Moreover, it follows from Theorem 4.3.1 in page 115 of
\cite{hirlem:funcon} that the following Hessian matrix
\begin{eqnarray}
&&\bigtriangledown^{2}f(y,\nu)\equiv
\left(\frac{\partial^{2}f(y,\nu)}{\partial y_{l}
\partial y_{k}}\right)_{(2NNJ)\times(2NNJ)}\;\;\mbox{for
all}\;\;l,k\in\{1,...,2NNJ\}. \elabel{hessianmI}
\end{eqnarray}
is positive definite at all $y$ within the $(2NNJ)$-dimensional
feasible region. Now define
\begin{eqnarray}
&&F(y,\nu)\equiv\{F_{l}(y,\nu),l\in{\{1,...,2NNJ\}}\}.\nonumber
\end{eqnarray}
Thus we know that $F(y^{*},\nu)=0$ and the Jacobian determinant of
$F(y,\nu)$ with respect to $y$ at $(y^{*},\nu)$ is nonzero due to
\eq{hessianmI}, i.e.,
\begin{eqnarray}
&&\frac{D(F_{1},...,F_{2NNJ})}{D(y_{1},...,y_{2NNJ})}\neq
0.\elabel{jacdet}
\end{eqnarray}
Therefore $F(y,\nu)$ satisfies all the conditions as stated in the
implicit function theorem. Hence $F(y,\nu)=0$ uniquely determines a
$(2NNJ)$-dimensional function $y^{*}(\nu)$ that is continuous and
differentiable with respect to $\nu$ in a neighborhood
$O(\nu,\epsilon)$ of $\nu$. Moreover, \eq{jacdet} and
\eq{minuscompI} hold in $O(\nu,\epsilon)$, which implies that
$y^{*}(\nu)$ is an optimal solution to the problem in
\eq{transopt}-\eq{fconstraint} for each $\nu\in O(\nu,\epsilon)$.

Secondly, we suppose that the problem in
\eq{transopt}-\eq{fconstraint} for the function $f(y,\nu)\in
S(k_{0})$ has multiple optimal solutions located in the interior of
the associated feasible region. Without loss of generality, we
suppose that these optimal points are all in a $m$-dimensional
hyperplane that is parallel to each coordinate-axis corresponding to
those $y$ with part of its components, $y_{s_{l}}\in{\cal Y}$, where
\begin{eqnarray}
&&{\cal Y}\equiv\{y_{s_{l}}\in R,l\in\{1,...,
m\},s_{l}\in\{1,...,2NNJ\}\}\;\;\mbox{for some}\;\;
m\in\{1,...,2NNJ\}. \nonumber
\end{eqnarray}
(Here we remark that, if this is not the case, we can employ the
method of rotation transformation to make this case true.)
Therefore, due to \eq{fgamma} and the concavity of $f(y,\nu)$ in
$y$, we know that $f(y,\nu)$ is independent of $y_{s_{l}}\in{\cal
Y}$. Thus there exists a $(2NNJ-m)$-dimensional set $P_{\nu}$
corresponding to each $\nu$ such that $f(y,\nu)$ only depends on
$y_{s_{l}}\in{\cal Y}^{c}$ (the complementary set of ${\cal Y}$) and
is strictly concave in those $y_{s_{l}}$. Therefore for any optimal
point $y^{*}(\nu)$ in the set $P_{\nu}$ and by considering the
similar $(2NNJ-m)$-dimensional problem as in
\eq{minuscompI}-\eq{hessianmI}, we can conclude that $y^{*}(\nu)$ is
continuous and differentiable in a neighborhood $O(\nu,\epsilon)$ of
$\nu$.

So, if the optimal points of $f(y,\nu)$ are all strictly located
within the feasible region when $\nu$ moves in ${\cal N}$, it
follows from the above discussion that $f(y,\nu)$ keeps either
strictly concave or flat with respect to $y_{s_{l}}\in{\cal Y}^{c}$
or $y_{s_{l}}\in{\cal Y}$ for all $\nu\in{\cal N}$. Thus we can
conclude that all the optimal paths $y^{*}(\nu)$ are continuous and
differentiable with respect to $\nu\in{\cal N}$. In other words, any
set $\{\Gamma^{*}_{1}(\nu,i),...,\Gamma^{*}_{J}(\nu,i)\}$  of the
optimal covariance matrices is continuous and differentiable with
respect to $\nu\in{\cal N}$. Hence it follows from \eq{macform} that
the corner points of the capacity region, which are determined by
the following equations, form a smooth curved facet
$f(\Gamma_{1}^{*}(\nu,i),...,\Gamma_{J}^{*}(\nu,i))$ when $\nu$
moves in the region ${\cal N}$, and moreover, the facet does not
depend on the choice of the set
$\{\Gamma_{1}^{*}(\nu,i),...,\Gamma_{J}^{*}(\nu,i)\}$ along
$\nu\in{\cal N}$. In addition, due to \eq{macform}, for all
$j\in{\cal J}$, we have
\begin{eqnarray}
&&c_{j}(\nu)=\mbox{log}\left|I+\sum_{l=1}^{j}
H_{l}^{\dagger}(i)\Gamma^{*}_{l}(\nu,i)H_{l}(i)\right|
-\mbox{log}\left|I+\sum_{l=1}^{j-1}
H_{l}^{\dagger}(i)\Gamma^{*}_{l}(\nu,i)H_{l}(i)\right|
\elabel{nusurf}
\end{eqnarray}

However, if some optimal point of $f(y,\nu)$ reaches one of the
boundaries of the feasible region when $\nu$ moves in ${\cal N}$,
then the associated justification for this case is part of the proof
in the following Step Two.

{\bf Step Two:} Without loss of generality, we suppose that
$f(y,\nu)$ is strictly concave for all $\nu\in{\cal N}$ and
otherwise we can employ the similar argument as above. Therefore, if
$y^{*}=(y^{*}_{1},...,y^{*}_{2NNJ})'$ is the solution to the
optimization problem in \eq{transopt}-\eq{fconstraint}, which is
located on one of the boundary pieces, and if ${\cal
Y}^{*}\equiv\{y_{s_{l}}^{*},l\in\{1,...,
m\},s_{l}\in\{1,...,2NNJ\}\}$ for some $m\in\{1,...,2NNJ\}$ is the
set of components of $y^{*}$, which are either $0$ or on the surface
$f_{j}(y)=0$ for some $j\in{\cal J}$ since $f_{j}(y)$ depends only
on part of the components of $y$, then the remaining components of
$y^{*}$ are located in the interior of the corresponding
$(2NNJ-m)$-dimensional feasible region and satisfy
\begin{eqnarray}
&&F_{l}(y^{*},\eta^{*},\nu)\equiv\frac{\partial
f(y^{*},\nu)}{\partial y_{l}}+\sum_{j\in{\cal
L}_{1}}\eta^{*}_{j}\frac{\partial f_{j}(y^{*})}{\partial
y_{l}}=0\elabel{minuscomp}
\end{eqnarray}
for each $l\in {\cal L}\equiv\{l:y^{*}_{s_{l}}\in
y^{*}\setminus{\cal Y}^{*}\}$, where ${\cal L}_{1}\equiv{\cal
J}\cap\{j:f_{j}(y^{*})=0\}$ and $\eta^{*}_{j}$ ($j\in{\cal L}_{1}$)
are the Lagrangian multipliers corresponding to $y^{*}$. Now let
$y\in R_{+}^{2NNJ}$ denote the vector whose components $y_{s_{l}}$
for all $l\in\{1,...,m\}$ are confined in ${\cal Y}^{*}$. Then
$f(y,\nu)$ is strictly concave in the components of $y$ except those
$y_{s_{l}}\in{\cal Y}^{*}$ since it is sufficiently smooth in $y$
and since $y^{*}_{s_{l}}$ with $l\in{\cal L}$ is in the interior of
the corresponding $(2NNJ-m)$-dimensional feasible region. Thus it
follows from Theorem 4.3.1 in page 115 of \cite{hirlem:funcon} that
the following Hessian matrix
\begin{eqnarray}
&&\bigtriangledown^{2}f(y,\nu)\equiv
\left(\frac{\partial^{2}f(y,\nu)}{\partial y_{l}
\partial y_{k}}\right)_{(2NNJ-m)\times(2NNJ-m)}\;\;\mbox{for
all}\;\;l,k\in{\cal L}. \nonumber
\end{eqnarray}
is positive definite at all $y$ whose components $y_{s_{l}}$ for all
$l\in\{1,...,m\}$ are confined in ${\cal Y}^{*}$. Moreover, if we
define
\begin{eqnarray}
&&F(y,\eta,\nu)\equiv\left\{\begin{array}{ll}
F_{l}(y,\eta,\nu)&\mbox{if}\;\;l\in{\cal L},\\
F_{l}(y)=f_{l}(y)&\mbox{if}\;\;l\in{\cal L}_{1},
\end{array}\right.\elabel{fttt}
\end{eqnarray}
we can conclude that the Jacobian determinant of $F(y,\eta,\nu)$
with respect to $y_{l}$ $(l\in{\cal L})$ and $\eta_{l}$ ($l\in{\cal
L}_{1}$) is nonzero at $y^{*}$. Moreover, due to the definition of
$f(y,\nu)$ and $f_{j}(y)$ for $j\in{\cal J}$, we know that
$F(y,\eta,\nu)$ satisfies all the conditions as stated in the
implicit function theorem. Hence $F(y,\eta,\nu)=0$ uniquely
determines a $(2NNJ+\bar{J})$-dimensional function
$(y^{*}(\nu),\eta^{*}(\nu))$ that is continuous and differentiable
in $\nu\in{\cal N}$ (where $\bar{J}$ is the number of $f_{l}$
($l\in{\cal L}_{1}$) such that $f_{l}(y^{*})=0$), and moreover, all
the components $y^{*}_{s_{l}}(\nu)$ with $l\in\{1,...,m\}$ are
confined in ${\cal Y}^{*}$ when $\nu\in{\cal N}$ moves. Therefore
the remaining proof of this boundary situation can be divided into
the following three cases.

{\em Case One:} When $u\in{\cal N}$ continuously moves to a vector
$\nu\in{\cal N}$, the optimal point $y^{*}(u)$ moves from the
interior of the feasible region to the optimal point
$y^{*}(\nu)(=y^{*})$ on the boundary of the feasible region. Then we
need to prove that $y^{*}(u)$ and its associated derivatives
converge to $y^{*}(\nu)$ and its corresponding derivatives as $u$
converges to $\nu$ continuously within a neighborhood of
$\nu\in{\cal N}$ in the whole $2NNJ$-dimensional feasible region,
which implies that the components $y^{*}_{s_{l}}(u)$ for all
$l\in\{1,...,m\}$ are not necessarily confined in ${\cal Y}^{*}$
when $u\in{\cal N}$ moves.

In fact, let ${\cal
L}_{2}\equiv\{k\in\{1,...,m\}:y^{*}_{s_{k}}(\nu)=0,
y^{*}_{s_{k}}(\nu)\in{\cal Y}^{*}\}$ and define the following
constraints of parallel surfaces,
\begin{eqnarray}
&&\tilde{f}_{j}(y,b)\equiv f_{j}(y)-b_{j}=0\;\;\mbox{for}\;\;
j\in{\cal L}_{1},
\elabel{newconstraintI}\\
&&g_{k}(y,b)\equiv y_{s_{k}}-b_{k}=0\;\;\mbox{for}\;\;k\in{\cal
L}_{2},\elabel{newconstraintII}
\end{eqnarray}
where $b$ is an arbitrary constant vector whose components are given
by $b_{j}$ ($j\in{\cal L}_{1}$) and $b_{k}$ ($k\in{\cal L}_{2}$).
Therefore, by applying the KKT optimality conditions, the optimal
solution $\tilde{y}^{*}$ to the the problem \eq{transopt} with the
constraints \eq{newconstraintI}-\eq{newconstraintII} should be given
by the following equations
\begin{eqnarray}
&&F_{l}(\tilde{y}^{*},\tilde{\eta}^{*},\nu,b)\equiv\frac{\partial
f(\tilde{y}^{*},\nu)}{\partial y_{l}}+\sum_{j\in{\cal
L}_{1}}\tilde{\eta}^{*}_{j}\frac{\partial
\tilde{f}_{j}(\tilde{y}^{*},b)}{\partial y_{l}}+\sum_{k\in{\cal
L}_{2}}\tilde{\eta}^{*}_{k}\frac{\partial
g_{k}(\tilde{y}^{*},b)}{\partial y_{l}}=0\elabel{minuscompII}
\end{eqnarray}
for each $l\in {\cal L}\equiv\{l:y^{*}_{s_{l}}\in
y^{*}\setminus{\cal Y}^{*}\}$, where $\tilde{\eta}^{*}_{j}$
($j\in{\cal L}_{1}$) and $\tilde{\eta}_{k}^{*}$ ($k\in{\cal L}_{2}$)
are the related Lagrangian multipliers corresponding to
$\tilde{y}^{*}$. Now, for each $\tilde{y}\in R^{2NNJ}$, define
\begin{eqnarray}
&&F(\tilde{y},\tilde{\eta},\nu,b)\equiv\left\{\begin{array}{ll}
F_{l}(\tilde{y},\tilde{\eta},\nu,b)&\mbox{if}\;\;l\in{\cal L},\\
F_{l}(\tilde{y},b)=\tilde{f}_{l}(\tilde{y},b)&\mbox{if}\;\;
l\in{\cal L}_{1},\\
F_{l}(\tilde{y},b)=g_{l}(\tilde{y},b)&\mbox{if}\;\; l\in{\cal
L}_{2}.
\end{array}\right.\elabel{ftttI}
\end{eqnarray}
Then, by the similar argument as used for \eq{fttt}, we know that
there is a unique $(2NNJ+\bar{J})$-dimensional optimal path
$(\tilde{y}^{*}(u,b),\tilde{\eta}^{*}(u,b))$ which is continuous and
differentiable with respect to $(u,b)\in{\cal N}\times
R^{\tilde{J}}$ (where $\tilde{J}$ is the dimension of $b$), and
moreover, all the components of $\tilde{y}^{*}(u,b)$ corresponding
to $y^{*}_{s_{l}}\in{\cal Y}^{*}$ satisfy the constraints
\eq{newconstraintI}-\eq{newconstraintII}, either being $b_{l}$
($l\in{\cal L}_{2}$) or on the boundary $f_{l}(y)=b_{l}$ ($l\in{\cal
L}_{1}$). Thus we know that $\tilde{y}^{*}(u,b)$ and its associated
derivatives converge to $y^{*}(\nu)$ and its corresponding
derivatives as $(u,b)$ converges to $(\nu,0)$ continuously.
Moreover, notice that $y^{*}(u)=\tilde{y}^{*}(u,b)$ when $b_{l}>0$
($l\in{\cal L}_{2}$) and $b_{l}<0$ ($l\in{\cal L}_{1}$) are all
close to zero, which implies that all the components of
$\tilde{y}^{*}(u,b)$ corresponding to $y^{*}_{s_{l}}\in{\cal Y}^{*}$
are also continuous and differentiable with respect to $u\in{\cal
N}$ when $b_{l}>0$ ($l\in{\cal L}_{2}$) and $b_{l}<0$ ($l\in{\cal
L}_{1}$) are all close to zero. So we can conclude that $y^{*}(u)$
and its associated derivatives converge to $y^{*}(\nu)$ and its
associated derivatives as $u\rightarrow\nu$, which implies that
$y^{*}(\nu)$ is continuous and differentiable at a neighborhood of
$\nu$ in the whole $2NNJ$-dimensional feasible region.

{\em Case Two:} When $u\in{\cal N}$ moves to a vector $\nu\in{\cal
N}$, the optimal point $y^{*}(u)$ moves to the optimal point
$y^{*}(\nu)$ ($=y^{*}$) from a boundary piece of the feasible region
next to the boundary piece on which $y^{*}(\nu)$ is located. The
proof for this case is similar to the one as used in {\em Case One}.
Hence we omit it.

{\em Case Three:} When $u\in{\cal N}$ moves to a vector $\nu\in{\cal
N}$, the optimal point $y^{*}(u)$ moves to the optimal point
$y^{*}(\nu)$ ($=y^{*}$) from a boundary piece of the feasible region
that is not next to the boundary piece on which $y^{*}(\nu)$ is
located. Due to the concavity of $f(y,u)$, the optimal point
$y^{*}(u)$ must go first into the interior of the feasible region
and then to the other boundary piece. Therefore, the proof for this
case is the same as the one as used in {\em Case One}.

In the end, we note that the boundary piece corresponding to
$S(k_{1},...,k_{J})$ is a $J$-dimensional linear facet, which is
determined by the sum-rate capacity bound (see, e.g.,
\cite{yurhe:itewat} for more details). Hence we we finish the proof
of Lemma~\ref{smoothsurfaces}. $\Box$

\subsection{Proof of Lemma~\ref{smoothsurfacesI}}

It follows from \cite{goljaf:caplim} that the capacity region for
the $J$-user MIMO BC with $N=1$ and each $i\in{\cal K}$ is given by
\begin{eqnarray}
&&{\cal R}(i)={\cal C}_{BC}(P,H(i))
\elabel{bccapcity}\\
&=&\bigcup_{\{(P_{1},...,P_{J}):\sum_{j=1}^{J}P_{j}=P\}} {\cal
C}_{MAC}(P_{1},...,P_{J},H^{\dagger}(i))\nonumber\\
&=&\bigcup_{\{(P_{1},...,P_{J}):\sum_{j=1}^{J}P_{j}=P\}} \left\{c\in
R_{+}^{J}:\sum_{j\in S}c_{j}\leq\frac{1}{2}\log\left|I+\sum_{j\in
S}H_{j}^{\dagger}(i)P_{j}H_{j}(i)\right|,\forall\;\;S\subset{\cal
J}\right\}.\nonumber
\end{eqnarray}
So, due to the similarity of structures between ${\cal M}(i)$ in
\eq{macform} and ${\cal R}(i)$ in \eq{bccapcity}, we can apply the
similar discussion as for the MIMO MAC and the discussion in
\cite{visjin:duaach} to conclude that the claims in the lemma are
true. Hence we we finish the proof of Lemma~\ref{smoothsurfacesI}.
$\Box$

%
%
%







\bibliographystyle{nonumber}

\end{document}